\documentclass[a4paper,11pt]{amsart}

\usepackage[english]{babel}
\usepackage[OT1]{fontenc}
\usepackage{graphicx}
\usepackage{indentfirst}
\usepackage{amsmath,amssymb,amsthm,empheq}
\usepackage{ dsfont }
\usepackage[italian]{varioref}
\usepackage{booktabs}
\usepackage{siunitx}
\usepackage{microtype}
\usepackage{epstopdf}
\usepackage[margin=2cm]{geometry} 
\usepackage[hidelinks]{hyperref}
\usepackage{fancyhdr}
\usepackage{version}
\usepackage[utf8]{inputenc}
\usepackage{tikz}
\usepackage{tikz-cd}
\usepackage[mathscr]{eucal}
\usepackage{enumerate}
\usepackage{addfont}
\usepackage{aurical}
\usepackage[titletoc]{appendix}
\usepackage[alphabetic, initials]{amsrefs}

\numberwithin{equation}{section}

\newtheorem{theorem}{Theorem}[section]
\newtheorem{lemma}[theorem]{Lemma}
\newtheorem{corollary}[theorem]{Corollary}
\newtheorem{proposition}[theorem]{Proposition}
\theoremstyle{definition}
\newtheorem{definition}[theorem]{Definition}
\newtheorem{remark}[theorem]{Remark}

\newcommand{\e}{\varepsilon}
\newcommand{\R}{\mathbb{R}}
\newcommand{\N}{\mathbb{N}}
\newcommand{\Z}{\mathbb{Z}}  
\renewcommand{\S}{\mathbb{S}}  

\newcommand{\norm}[1]{\|#1\|}
\newcommand{\haus}[1]{\mathcal{H}^{#1}}
\newcommand{\cont}{\mathcal{C}}

\renewcommand{\P}{\operatorname{P}}
\newcommand{\E}{\mathscr{E}} 
\newcommand{\cE}{\text{\Fontauri{E}}}
\newcommand{\F}{\mathscr{F}}
\newcommand{\J}{\mathcal{J}}

\DeclareMathOperator*{\essinf}{ess\,inf}
\DeclareMathOperator*{\esssup}{ess\,sup}
\DeclareMathOperator*{\osc}{osc}

\newcommand{\loc}{\rm{loc}}
\renewcommand{\leq}{\leqslant}
\renewcommand{\le}{\leqslant}
\renewcommand{\geq}{\geqslant}
\renewcommand{\ge}{\geqslant}
\renewcommand{\epsilon}{\varepsilon}
\allowdisplaybreaks

\title[Non-local planelike minimizers]{Non-local planelike minimizers \\
and 
	$\Gamma$-convergence of periodic energies\\ to a local anisotropic perimeter}
\date{}

\author{\href{https://research-repository.uwa.edu.au/en/persons/serena-dipierro}{Serena Dipierro}}
\address{
	Serena Dipierro,
	Department of Mathematics and Statistics, 
	The University of Western Australia, 
	35 Stirling Highway, Perth, WA 6009, Australia}
\email{serena.dipierro@uwa.edu.au}

\author{Matteo Novaga}
\address{
	Matteo Novaga,
	Dipartimento di Matematica,
	Universit\`a di Pisa,
	Largo Bruno Pontecorvo 5, 56126 Pisa, Italy}
\email{matteo.novaga@unipi.it}

\author{\href{https://research-repository.uwa.edu.au/en/persons/enrico-valdinoci}{Enrico Valdinoci}}
\address{
	Enrico Valdinoci,
	Department of Mathematics and Statistics, 
	The University of Western Australia, 
	35 Stirling Highway, Perth, WA 6009, Australia}
\email{enrico.valdinoci@uwa.edu.au}

\author{Riccardo Villa}
\address{
	Riccardo Villa,
	Department of Mathematics and Statistics, 
	The University of Western Australia, 
	35 Stirling Highway, Perth, WA 6009, Australia}
\email{riccardo.villa@research.uwa.edu.au}

\begin{document}
	
	\maketitle
	
	\begin{abstract}
		We investigate a homogenization problem related to a non-local interface energy with a periodic forcing term. We show the existence of planelike minimizers for such energy. 
		
		Moreover, we prove that, under suitable assumptions on the non-local kernel and the external field, the sequence of rescaled energies~$\Gamma$-converges to a suitable local anisotropic perimeter, where the anisotropy is defined as the limit of the normalized energy of a planelike minimizer in larger and larger cubes (i.e., what is called in jargon ``stable norm'').
		
		To obtain this, we also establish several auxiliary results, including: the minimality of the level sets of the minimizers, explicit bounds on the oscillations of the minimizers, density estimates for almost minimizers, and non-local perimeter estimates in the large.
	\end{abstract}
	
	\setcounter{tocdepth}{1}
	{
		\hypersetup{linkcolor=black}
		\tableofcontents
	}

	\section{Introduction}
	
	Planelike minimizers are objects (such as surfaces or interfaces) that minimize a given energy functional and lie at a bounded distance from a hyperplane. They tipically occur
	in periodic media, where the energy density repeats periodically in space, allowing for suitable cancellations of the inhomogeneous minutiae of the environment that produce stable structures which, albeit not being perfectly flat, appear as flat at a large scale.
	In this spirit, planelike minimizers show how a possibly complicated structure at a fine scale
	behaves like a simple, averaged object in the large, and this phenomenon in turn provides an essential building
	block to understand homogenization.
	
	Moreover, planelike minimizers link the microscopic features of a medium to its macroscopic behavior, often reducing the analysis to that of ``effective energies'', called
	in jargon ``stable norms'', which play the role of homogenized ``surface tensions''
	capturing the effective anisotropy of the medium at a large scale (see below for further details).
	
	Also, when planelike minimizers possess some geometric organization
	(such as foliations or laminations of space), they can help describe the global geometry of minimizers.
	
	In certain variational settings arising in the theory of dynamical systems
	(like twist maps or Frenkel-Kontorova models), planelike minimizers correspond to Aubry-Mather sets,
	namely to configurations that minimize action and exhibit periodic and quasi-periodic order
	(see~\cite{MR670747, MR719634, MR2356117}).
	In analogy to this, planelike minimizers have been constructed for
	elliptic functionals and minimal surfaces
	(see~\cite{MR0847308, MR1852978, MR2197072, MR2542727}),
	phase transition models and partial differential equations
	(see~\cite{MR2099113, MR2342272, MR2809349}), fluid jets (see~\cite{MR2126143}), statistical mechanics and spin systems (see~\cite{MR1620543, MR3652519}).
	
	In this paper, we investigate a homogenization problem related to a non-local interface energy with a periodic forcing term, extending the results presented in~\cite{chambolle_thouroude} to a non-local setting. Considering some assumptions on the non-local kernel~$K$ and the external field~$g$, we show the existence of planelike minimizers for such energy, i.e. minimizers that are at bounded distance from a plane (see Definition~\ref{def::planelike} below). Moreover, we prove that the sequence of rescaled energies~$\{\F_\epsilon\}_\epsilon$
$\Gamma$-converges, as~$\epsilon$ approaches~$0$, to a suitable local anisotropic perimeter~$\F_\phi$. The anisotropy~$\phi$ is also known as ``stable norm'' (see~\cite{MR3223561}), and is defined as the limit of the (normalized) energy of a planelike minimizer in larger and larger cubes.
	
	Adapting some ideas presented in~\cite{MR1852978, chambolle_thouroude}
	to our setting, we construct planelike minimizers for our energy~$\F$ via a cell-problem. Then, we show that the stable norm~$\phi$ is well-defined and we prove the main~$\Gamma$-convergence result.
	
	\subsection*{Setting and main assumptions}
	
	The mathematical framework adopted in this paper goes as follows.
	Let~$Q:=(0,1)^n$ be the~$n$-dimensional cube of side~$1$, 
	and let~$g\in L^\infty(\R^n)$ be a~$\Z^n$-periodic function such that
	\begin{equation}\label{eq::g_zero_avg}
		\int_Q g(x)\,dx=0.
	\end{equation}
	
In all the paper, the following assumptions on the kernel~$K$ will be in force. 
	
	Suppose\footnote{Hypotheses~\eqref{eq::K_invariance} and~\eqref{eq::K_integrable} are very natural assumptions to work with. Compare for instance with~\cite[Paragraph~1.2]{MR1634336}.} that, for all~$x$, $y$, $w\in\R^n$, and for any rotation~$R\in SO(n)$,
	\begin{equation} \label{eq::K_invariance}
		K(y,x)=K(x,y)=K(x+w,y+w)=K(Rx,Ry)\ge0,
	\end{equation}
	and
	\begin{equation} \label{eq::K_integrable}
		%\iint_{(Q\times\R^n)\cup(\R^n\times Q)}|x-y|\,K(x,y)\,dx\,dy<+\infty.
		\int_{\R^n} |h|K(h,0)\,dh<+\infty.
	\end{equation}
	Namely, we require that the kernel under consideration is non-negative, symmetric, translation and rotation invariant, and has some integrability properties. 
	
	Moreover, we assume that there exist parameters~$s_1$ and~$s_2$, and positive constants~$\delta$, $\kappa_1$, $\kappa_2$ and~$\kappa_3$, such that
	\begin{equation*}
		0<s_1<\frac{1}{2}<s_2<1,\qquad\kappa_1\leq\kappa_2,
	\end{equation*}
	and, 
	for every~$x$, $y\in\R^n$,
	\begin{eqnarray}
			\label{eq::K_behavior}&&\kappa_1 \frac{\chi_{(0,\delta)}(|x-y|)}{|x-y|^{n+2s_1}} \leq K(x,y) \leq \kappa_2\min\left\{ \frac{1}{|x-y|^{n+2s_1}},\frac{1}{|x-y|^{n+2s_2}}\right\}\\
			 \label{eq::K_lower_bound_Q}\mbox{and}\quad&&\inf_{(x,y)\in Q\times Q} K(x,y)\geq \kappa_3.
	\end{eqnarray}
	
	\begin{remark}
		We point out that, in light of~\eqref{eq::K_invariance}, the lower bound in~\eqref{eq::K_lower_bound_Q}	is equivalent to
		\begin{equation}  \label{eq::K_lower_bound_Q_2}
			K(x,y)\geq \kappa_3\chi_{(0,\sqrt{n})}(|x-y|),
			\quad{\mbox{for all }} (x,y)\in\R^n\times\R^n.
		\end{equation}
		Indeed, for every~$x$, $y\in Q$, we have that~$|x-y|\leq \mbox{diam}(Q)=\sqrt{n}$. Thus, \eqref{eq::K_lower_bound_Q_2} yields~\eqref{eq::K_lower_bound_Q}. 
		
Moreover, for every~$x$, $y\in\R^n$ such that~$|x-y|<\sqrt{n}$, we have that~$x-y\in Q$. Hence, from~\eqref{eq::K_invariance} and~\eqref{eq::K_lower_bound_Q}, we infer that
		\begin{equation*}
			K(x,y) = K(x-y,0) \geq \inf_{(w,z)\in Q\times Q} K(w,z) \geq \kappa_3,
		\end{equation*}
		which entails~\eqref{eq::K_lower_bound_Q_2}.

We also point out that	if the lower bound in~\eqref{eq::K_behavior} holds true with~$\delta\geq \sqrt{n}$, then~\eqref{eq::K_lower_bound_Q} and~\eqref{eq::K_lower_bound_Q_2} are always verified with~$\kappa_3:=\kappa_1 \delta^{-n-2s_1}$.
	\end{remark}

	As model cases of kernels that satisfy~\eqref{eq::K_invariance}, \eqref{eq::K_integrable}, \eqref{eq::K_behavior}, and~\eqref{eq::K_lower_bound_Q}, one can think about 
	\begin{align*}
		&K_1(x,y) := \frac{\chi_{(0,\sqrt{n})}(|x-y|)}{|x-y|^{n+2s}},\\
		&K_2(x,y) := \frac{\chi_{(0,\delta)}(|x-y|)}{|x-y|^{n+2s}}+\chi_{[\delta,+\infty)}(|x-y|)e^{-|x-y|},\\
		\mbox{and }\quad&K_3(x,y) := \frac{\chi_{(0,\delta)}(|x-y|)}{|x-y|^{n+2s}}+\frac{\chi_{[\delta,+\infty)}(|x-y|)}{|x-y|^{n+2S}},	
	\end{align*}
	with~$0<s<\frac{1}{2}<S<1$.
	
	\begin{remark} \label{rem::high_localing_K}
		We point out that~\eqref{eq::K_behavior} entails
		\begin{equation*}
			K(x,y) \leq \frac{\kappa_2 }{|x-y|^{n+2s_2}},\quad{\mbox{for all }} (x,y)\in\R^n\times\R^n.
		\end{equation*}
		Hence, $K$ is integrable at infinity (when~$x$ and~$y$ are very far apart). As we will see, this condition yields a local energy in the~$\Gamma$-limit.
	\end{remark}
	
	From now on, we will call a \textit{domain} any open and bounded set, not necessarily connected.
	
	For our purposes, we now recall the notion of~$K$-nonlocal interaction between disjoint sets~$A$, $B\subseteq\R^n$, that is
	$$ \mathcal{L}_K(A,B) := \iint_{A\times B} K(x,y)\, dx\,dy,$$
	and the notion of~$K$-nonlocal perimeter of a set~$E\subseteq\R^n$ with respect to a Lipschitz domain~$\Omega\subseteq\R^n$ (see Definition~\ref{def::unif_lipschitz} for the precise notion set with Lipschitz boundary), which is defined as
	\begin{equation*}
		\begin{split} 
			\P_K(E,\Omega)&:=\iint_{\Omega_\sharp}\frac{1}{2}|\chi_E(x)-\chi_E(y)|\,K(x,y)\,dx\,dy\\ 
			&=\mathcal{L}_K(\Omega\cap E,\Omega\cap E^c)+\mathcal{L}_K(\Omega\cap E,\Omega^c\cap E^c)+\mathcal{L}_K(\Omega^c\cap E,\Omega\cap E^c),
		\end{split}
	\end{equation*}
	where
	$$\Omega_\sharp:=(\Omega\times \Omega)\cup(\Omega^c\times \Omega)\cup(\Omega\times \Omega^c),$$
	with~$\Omega^c:=\R^n\setminus \Omega$.
	
We will also use the notation
$$\P_K(E):=\mathcal{L}_K(E,E^c).$$
	
	Moreover, we define the energy functional~$\J$ in~$\Omega$ as the perturbation of~$\P_K$ with the periodic external forcing term~$g$, given by
	\begin{equation*}
		\J(E,\Omega) := \P_K(E,\Omega) + \int_{E\cap\Omega} g(x)\,dx.
	\end{equation*}
	
	However, taking into account the periodicity of~$g$, in order to account for the energy
	contributions of~$\partial\Omega$, we will also consider the energy functional~$\F$ defined as
	\begin{equation}\label{1.4BIS}
		\F(E,\Omega) := \P_K(E,\Omega) + \int_{E\cap\mathcal{Q}(\Omega)_1} g(x)\,dx,
	\end{equation}
	where
	\begin{equation*}
		\mathcal{Q}(\Omega)_1 :=\{k+Q \mbox{ s.t. }k\in\Z^n,\ k+Q\subseteq\Omega\}.
	\end{equation*}
	
	\subsection*{Construction of planelike minimizers} 
	
	To construct minimizers for~$\J$ and~$\F$, given~$p\in\R^n$, we define
	\begin{equation*}
		\cE_p(u):=\iint_{Q\times \R^n} \frac{1}{2}|u(x)-u(y)+p\cdot(x-y)|\,K(x,y)\,dx\,dy+\int_Q g(x)\,u(x)\,dx.  
	\end{equation*}
	Then, we will address the cell problem:
	\begin{equation} \label{eq::cell_problem}
		\begin{split}
			&\mbox{find}\quad u_p\in \mathcal{W}:=\left\{u\in L^1_{\loc}(\R^n) \mbox{ s.t. $u$ is $\Z^n$-periodic and }\int_Qu(x)\,dx=0\right\}\\&
			\mbox{such that}\quad \cE_p(u_p) = \min_\mathcal{W} \cE_p.
		\end{split}
	\end{equation}
	
	\begin{remark}
		We stress that the choice of the domain of integration in the definition of $\cE_p$ is tailored so that the energy functional under consideration is additive with respect to unions of disjoint sets (compare~$Q\times\R^n$ with~$Q_\sharp$ and~$Q \times Q$).
		
		Moreover, this choice of $\cE_p$ is the right one for deriving the desired Euler–Lagrange equation (see Corollary~\ref{cor::EL_eq}) and for constructing (class-$A$) minimizers of~$\J$.
	\end{remark}
			
	The problem in~\eqref{eq::cell_problem} is well-defined, according to the following result:
	\begin{theorem}[Existence of minimizers for~$\cE_p$]\label{th::existence_minimizer}
		Let~$K:\R^n\times \R^n\to\R$ satisfy~\eqref{eq::K_invariance}, \eqref{eq::K_integrable}, \eqref{eq::K_behavior}, and~\eqref{eq::K_lower_bound_Q}. %Assume that~$g\in L^{\frac{n}{s}}(Q)$.
		
		Then, there exists~$\gamma>0$, depending only on~$n$, $\kappa_1$, $\delta$, and~$s_1$, such that, if~$\norm{g}_{L^{\frac{n}{2s_1}}(Q)}\le\gamma$, there exists~$u\in\mathcal{W}$ such that
		\begin{align}
			\label{eq::u_min_cE}& \cE_p(u) = \min_{\mathcal{W}}\cE_p\\
			\label{eq::u_lebesgue_holder}\mbox{and}\quad&u\in L^{\frac{n}{n-2s_1}}(Q).
		\end{align}
		% among the set of functions $u\in L^1_{\loc}(\R^n)$ which are $\Z^n$-periodic and such that~$\int_Q u(x)\,dx=0$.
	\end{theorem}
	
	A key step to find minimizers for~$\J$ is the following integral condition for minimizers~$u_p$ constructed in Theorem~\ref{th::existence_minimizer}.
	
	\begin{proposition}\label{prop::existence_calibration}
		Let~$K:\R^n\times \R^n\to\R$ satisfy~\eqref{eq::K_invariance}, \eqref{eq::K_integrable}, \eqref{eq::K_behavior}, and~\eqref{eq::K_lower_bound_Q} and
		let~$u_p$ be the minimizer for~$\cE_p$ in~$\mathcal{W}$ given by Theorem~\ref{th::existence_minimizer}. %among the set of functions $u\in L^1_{\operatorname{loc}}(\R^n)$ which are $\Z^n$-periodic and such that~$\int_Q u(x)\,dx=0$.
		
		Then, there exists~$z:\R^n\times\R^n\to[-1,1]$ such that, for all~$\eta\in C^\infty(\R^n)$ that are~$\Z^n$-periodic,
		we have that
		\begin{equation}\label{SUBD7}
			\iint_{Q\times\R^n}\frac{1}{2}z(x,y) (\eta(x)-\eta(y))\,K(x,y)\,dx\,dy+\int_Q g(x)\,\eta(x)\,dx=0.
		\end{equation}
		
		Moreover, for a.e.~$(x,y)\in\R^n\times\R^n$ and for all~$k\in\Z^n$,
		\begin{align}
			\label{eq::z_layering}
			&u_p(x)-u_p(y)+p\cdot(x-y) = z(x,y) \big|u_p(x)-u_p(y)+p\cdot(x-y)\big|,\\
			\label{eq::z_antisymm}
			&z(x,y)=-z(y,x),\\
			\label{eq::z_translation_invariance}
			\mbox{and}\quad&z(x+k,y+k)=z(x,y).
		\end{align}
	\end{proposition}
	A proof of Theorem~\ref{th::existence_minimizer} will be presented in Section~\ref{sec::existence_minimizer}, while Section~\ref{sec::existence_calibration} is devoted to showing Proposition~\ref{prop::existence_calibration}. 
	\medskip
	
We utilize the function~$z$ constructed in Proposition~\ref{prop::existence_calibration} as a ``calibration'' to show that the level sets of the minimizers provided by Theorem~\ref{th::existence_minimizer} are in turn (class-$A$) minimizers of our geometric problem.
	
	As a counterpart of the classical theory of minimal surfaces,
	and in the wake of~\cite{MR4117514, MR4142859},
	one can define calibrations in our setting as follows:
	
	\begin{definition} \label{def::calibration}
		Let~$E \subseteq \R^n$ be a set of finite~$K$-perimeter. We say that a function~$z:\R^n\times\R^n\to\R$ is a calibration for~$E$ in~$\Omega$ if
		\begin{equation*} \label{eq::def_calibration}
			\P_K(E,\Omega) = \iint_{\Omega_\sharp} \frac{1}{2} z(x,y)\left(\chi_E(x)-\chi_E(y)\right) K(x,y)\,dx\,dy.
		\end{equation*} 
	\end{definition}
	
	By analogy with the notation introduced in~\cite{MR1852978}, we also give the following definition of minimality for~$\J$.

	\begin{definition} \label{def::classA_minimilatity}
		We say that a set~$E\subseteq \R^n$ is a class-$A$ minimizer for~$\J$ if, for any domain~$\Omega\subseteq \R^n$, we have that
		\begin{equation*} \label{eq::classA_minimilatity}
			\J(E,\Omega) \leq \J(F,\Omega),\quad\mbox{ for any~$F$ such that }F\setminus \Omega= E\setminus \Omega.
 		\end{equation*}
	\end{definition}
	\begin{remark} \label{rem::class_A_F}
		Observe that if $E$ is a class-$A$ minimizer for $\J$, then, for any Lipschitz domain~$\Omega$, it holds that
		\begin{equation*} %\label{eq::cluster_minimality}
			\F(E,\Omega) \leq \F(F,\Omega),
		\end{equation*}
		for every~$F\subseteq\R^n$ such that~$F\setminus \mathcal{Q}(\Omega)_1=E\setminus \mathcal{Q}(\Omega)_1$. 
		
		Indeed, by the properties of~$\P_K$, for any~$A$ and~$B$ subsets of $\R^n$ such that~$A\Delta B\subset \mathcal{Q}(\Omega)_1$, we have that
		$$ \P_K(A,\Omega)-P_K(B,\Omega)=\P_K(A,\mathcal{Q}(\Omega)_1)-\P_K(B,\mathcal{Q}(\Omega)_1).$$
		
		Alternatively, observe that if $F\setminus \mathcal{Q}(\Omega)_1=E\setminus \mathcal{Q}(\Omega)_1$, then in particular $F\setminus \Omega= E\setminus \Omega$ and
		\begin{equation*}
			\int_{E\cap \Omega} g(x)\,dx-\int_{F\cap \Omega} g(x)\,dx = \int_{E\cap \mathcal{Q}(\Omega)_1} g(x)\,dx - \int_{F\cap \mathcal{Q}(\Omega)_1} g(x)\,dx.
		\end{equation*}
	\end{remark}
	
	For the rest of our discussion, we adopt the short notation
	\begin{equation} \label{eq::def_v}
		v_p(x):=u_p(x)+p\cdot x
	\end{equation}
	and define, for any~$t\in\R$, 
	\begin{equation} \label{eq::def_lev_set}
		E_{p,t}:=\big\{{\mbox{$x\in\R^n$ s.t. $v_p(x)>t$}}\big\}.
	\end{equation}
	
	Then, we have the following result:
	
	\begin{theorem}[Minimality of level sets] \label{th::level_sets_minimality}
Let~$K:\R^n\times \R^n\to\R$ satisfy~\eqref{eq::K_invariance}, \eqref{eq::K_integrable}, \eqref{eq::K_behavior}, and~\eqref{eq::K_lower_bound_Q} and
		let~$u_p$ be a minimizer for~$\cE_p$ in~$\mathcal{W}$. %among the set of functions $u\in L^1_{\operatorname{loc}}(\R^n)$ which are $\Z^n$-periodic and such that~$\int_Q u(x)\,dx=0$.
		Let also~$v_p$ and~$E_{p,t}$ be as in~\eqref{eq::def_v} and~\eqref{eq::def_lev_set}, respectively.
		
		Then, for every~$t\in\R$, the set~$E_{p,t}$ is a class-$A$ minimizer for~$\J$ in the sense of Definition~\ref{def::classA_minimilatity}.
	\end{theorem}
	
	A proof of Theorem~\ref{th::level_sets_minimality} can be found in Section~\ref{sec::level_sets_minimality}. The idea is to first prove a weaker version of Theorem~\ref{th::level_sets_minimality}, in which the conclusion holds for a.e.~$t\in\R$ only. Then, exploiting the closedness of class-$A$ minimizers with respect to the~$L^1_{\loc}$-convergence, we will obtain the desired result.
	\medskip 
	
	Thanks to suitable density estimates (see Proposition~\ref{prop::UDE} below), we also infer that a minimizer~$u_p$ of~$\cE_p$ has controlled oscillations in~$Q$. To this end, we recall that, for any~$\varphi\in\mbox{BV}(Q)$, 
	$$\osc_Q(\varphi) := \esssup_Q \varphi - \essinf_Q \varphi. $$
	
	\begin{theorem}[Minimizers of~$\cE_p$ have controlled oscillations] \label{th::controlled_osc}
		Let~$K:\R^n\times \R^n\to\R$ satisfy~\eqref{eq::K_invariance}, \eqref{eq::K_integrable}, \eqref{eq::K_behavior}, and~\eqref{eq::K_lower_bound_Q} and
		let~$u_p$ be a minimizer for~$\cE_p$ in~$\mathcal{W}$. % among the set of functions $u\in L^1_{\operatorname{loc}}(\R^n)$ which are $\Z^n$-periodic and such that~$\int_{Q} u(x)\,dx=0$.
		
		Let also~$v_p$ be as~\eqref{eq::def_v}.
		
		Then, there exists a constant~$c>0$ independent of~$p$ such that
		$$\osc_Q(v_p)\leq c|p|\qquad
		{\mbox{and}}\qquad\osc_Q(u_p)\leq (c+\sqrt{n})|p|.$$
	\end{theorem}
	
	As a consequence of Theorem~\ref{th::controlled_osc}, we obtain that the collection of level sets~$\{E_{p,t}\}_{t\in\mathcal{T}_p}$, where $$\mathcal{T}_p:=\{t\in\R \mbox{ s.t. } \partial E_{p,t}\cap Q\neq\varnothing\},$$ can be trapped within a strip of height~$M$, for some~$M>0$ independent of~$t$ and~$p$. For this reason, we call such level sets ``planelike'', according to the setting below.
	
	\begin{definition} \label{def::planelike}
		Let~$p\in\R\setminus\{0\}$ and let~$E\subseteq \R^n$ be a class-$A$ minimizer for the energy functional~$\J$.
		
		We say that~$E$ is a planelike minimizer for~$\J$ in direction~$p/|p|$ if there exists~$M>0$ (independent of~$p$) such that
		\begin{align*}
			\mbox{either}\quad& \{x\in\R^n \mbox{ s.t. } x\cdot p\leq -M|p|\} \subseteq E \subseteq \{x\in\R^n \mbox{ s.t. } x\cdot p\leq M|p|\}\\
			\mbox{or}\quad& \{x\in\R^n \mbox{ s.t. } x\cdot p\geq M|p|\} \subseteq E \subseteq \{x\in\R^n \mbox{ s.t. } x\cdot p\geq -M|p|\}.
		\end{align*}	
		In both cases, it also holds
		\begin{equation*}
			\partial E \subseteq \{x\in\R^n \mbox{ s.t. } |x\cdot p|\leq M|p|\}.
		\end{equation*}
	\end{definition}
	
	\begin{corollary}[Level sets are planelike minimizers] \label{cor::K_planelike}
		Let~$p\in\R^n$ and let~$K:\R^n\times \R^n\to\R$ satisfy~\eqref{eq::K_invariance}, \eqref{eq::K_integrable}, \eqref{eq::K_behavior}, and~\eqref{eq::K_lower_bound_Q}. Let~$u_p$ be a minimizer for~$\cE_p$ in~$\mathcal{W}$
		and let also~$v_p$ be as in~\eqref{eq::def_v}, $t\in\R$ and~$E_{p,t}$ be defined as in~\eqref{eq::def_lev_set}.
		
		If~$t\in\mathcal{T}_p$, then~$E_{p,t}$ is a planelike minimizer in direction~$p/|p|$ in the sense of Definition~\ref{def::planelike}. Namely, there exists~$M>0$ independent of~$p$ and~$t$ such that 
		\begin{eqnarray} 
			\label{eq::planelike_claim_1}&E_{p,t} \subseteq \{x\in\R^n \mbox{ s.t. } x\cdot p\geq -M|p|\} \\
			\label{eq::planelike_claim_2}\mbox{and}\quad &\{x\in\R^n \mbox{ s.t. } x\cdot p\geq M|p|\} \subseteq E_{p,t} .
		\end{eqnarray}
		In particular, it holds
		\begin{equation*}
			\partial E_{p,t} \subseteq \{x\in\R^n \mbox{ s.t. } |x\cdot p|\leq M|p|\}.  
		\end{equation*}
	\end{corollary}
	
	Section~\ref{sec::density estimates} is committed to showing the density estimates needed for Theorem~\ref{th::controlled_osc}, while a proof of the latter and Corollary~\ref{cor::K_planelike} is contained in Section~\ref{sec::osc}.
	
	\subsection*{$\Gamma$-convergence results and the stable norm}
	With the planelike minimizers of Corollary~\ref{cor::K_planelike}, we are able to prove the main result of this paper, which is the~$\Gamma$-convergence of a sequence of rescaled energies to the so-called stable norm.
	
	Let~$\Omega\subseteq \R^n$ be a Lipschitz domain and let, for any~$\epsilon>0$,
	$$ \mathcal{Q}(\Omega)_\epsilon :=\{\epsilon(k+Q) \mbox{ s.t. }k\in\Z^n,\ \epsilon(k+Q)\subseteq\Omega\}.$$ 
	Then, for any set~$E$, we define the rescaled energy~$\F_{\epsilon}$ in~$\Omega$ as 
	\begin{equation} \label{eq::rescaled_energy2}
		\begin{split}
			\F_{\epsilon} (E,\Omega) &:= \iint_{\Omega_\sharp} \frac{1}{2}|\chi_E(x)-\chi_E(y)|K_\epsilon(x,y)\,dx\,dy+\int_{\mathcal{Q}(\Omega)_\epsilon \cap E} g_\epsilon(x)\,dx\\
			&=\P_{K_\epsilon}(E,\Omega)+\int_{\mathcal{Q}(\Omega)_\epsilon \cap E} g_\epsilon(x)\,dx,
		\end{split}
	\end{equation}
	where 
	$$ K_\epsilon(x,y):=\frac1{\epsilon^{n+1}}K\left(\frac{x}{\epsilon},\frac{y}{\epsilon}\right)\qquad{\mbox{and}} \qquad g_\epsilon(x):=\frac1{\epsilon}g\left(\frac{x}{\epsilon}\right).$$ 
	
	For later convenience, we also set
	\begin{equation} \label{eq::truncated_energy}
		\begin{split}
			\E_\epsilon(E,\Omega) &:= \F_\epsilon(E,\Omega) - \mathcal{L}_{K_\epsilon}(E\cap\Omega,E^c\cap\Omega^c)\\
			&=  \iint_{\Omega\times\R^n} \chi_E(x)\chi_{E^c}(y) K_\epsilon(x,y)\,dx\,dy+\int_{\mathcal{Q}(\Omega)_\epsilon \cap E} g_\epsilon(x)\,dx.
		\end{split}
	\end{equation}
	
	We will show that the sequence~$\{\F_{\epsilon}\}_{\epsilon>0}$
	$\Gamma$-converges to a local anisotropic perimeter. 
	
	To this end, let~$p\in \S^{n-1}$ and let~$Q^p$ be a cube of side~$1$ with a face perpendicular to~$p$. Moreover, let~$E_p$ be a planelike minimizer for $\J$.
	%any level set of a minimizer~$u_p$ for~$\cE_p$, as constructed in Theorem~\ref{th::existence_minimizer}. 
	Then, in the wake of~\cite{MR3223561}, we define the stable norm~$\phi:\S^{n-1}\to\R$ as
	\begin{equation} \label{eq::anisotropy}
		\begin{split}
			\phi(p) := \lim_{\epsilon\to0^+} \F_{\epsilon}(\epsilon E_p,Q^p) = \lim_{R\to+\infty} R^{1-n}\F_1(E_p,Q^p_{R})  ,
			%\inf\Bigg\{\frac{\F_1(E_p,T_C)}{|C|} \mbox{ s.t. } C\in\cont_e \mbox{, and~$E_p$ is any level set of the minimizer}& \\ \mbox{$u_p$ constructed as in Theorem~\ref{th::existence_minimizer}} &\Bigg\},
		\end{split}
	\end{equation}
	whenever such limit exists. Here, we are adopting the notation~$Q^p_R$ for the cube of side~$R$ with a face perpendicular to~$p$.
	
	Besides, for any set~$E\subseteq \Omega$ of finite perimeter, we define 
	\begin{equation} \label{eq::anisotropic_per}
		\F_\phi(E,\Omega) := \int_{\partial E\cap\Omega} \phi(\nu_E(x)) d\haus{n-1}(x).
	\end{equation}
	Then, the following holds true:
	
	\begin{theorem}[$\Gamma$-convergence of the energies~$\F_\epsilon$] \label{th::gamma_conv}
		We have that~$\F_\epsilon\xrightarrow{\Gamma} \F_\phi$, as~$\epsilon\to0$, with respect to the~$L^1_{\loc}(\R^n)$-convergence of sets. Namely, for every set~$E\subseteq\R^n$, we have that
		\begin{enumerate}[(i)]
			\item \label{item::gamma_conv_inf} 
			for any sequence of sets~$\{E_\epsilon\}_\epsilon$ such that~$E_\epsilon\to E$ in~$L^1_{\loc}(\R^n)$ (i.e.~$|E_\epsilon\Delta E|\to0$ as~$\epsilon\to0$),
			\begin{equation} \label{eq::gamma_conv_inf}
				\liminf_{\epsilon\to0} \F_\epsilon(E_\epsilon,\Omega) \geq \F_\phi(E,\Omega);
			\end{equation}
			
			\item \label{item::gamma_conv_sup} 
			there exists a sequence of sets~$\{E_\epsilon\}_\epsilon$ such that~$E_\epsilon\to E$ in~$L^1_{\loc}(\R^n)$ and
			\begin{equation} \label{eq::gamma_conv_sup}
				\limsup_{\epsilon\to0} \F_\epsilon(E_\epsilon,\Omega) \leq \F_\phi(E,\Omega).
			\end{equation}
		\end{enumerate}
	\end{theorem}
	
	Some preliminary results needed for Theorem~\ref{th::gamma_conv} can be found in Section~\ref{sec::preliminary_results_gamma_conv}, while we will present some properties of~$\phi$ such as well-definedness in Section~\ref{sec::phi_well_defined} and continuity in Section~\ref{sec::phi_continuous}. Then, Sections~\ref{sec::gamma_liminf} and~\ref{sec::gamma_limsup} are devoted to proving Theorem~\ref{th::gamma_conv}-\eqref{item::gamma_conv_inf} and Theorem~\ref{th::gamma_conv}-\eqref{item::gamma_conv_sup} respectively.
	
	We stress that Proposition~\ref{prop::energy_lower_bound_BR} in Section~\ref{sec::preliminary_results_gamma_conv} entails that the energy~$\F_1$ is non-negative in every large domain. However, there is no Poincar\'e Inequality available for generic non-local kernels as the ones
considered by our\footnote{Check~\cite[Proposition~4.1]{MR3732175} for a Poincar\'e Inequality for a non-local kernel~$K$ under some additional assumptions.} setting. Thus, given a sequence of sets~$\{E_\epsilon\}_\epsilon$ with~$\sup_\epsilon \F_\epsilon(E_\epsilon,\Omega)<+\infty$,  to the best of our knowledge, Proposition~\ref{prop::energy_lower_bound_BR} does not provide any useful uniform bound either on the $BV$-norm~$\norm{\chi_{E_\epsilon}}_{BV(\Omega)}$, or on any Gagliardo seminorm~$[\chi_{E_\epsilon}]_{W^{s,p}(\Omega)}$ (compare with~\cite[Remark~2.2]{chambolle_thouroude}). For this reason, it remains an open question whether a compactness result holds for sequences of sets with uniformly bounded energy.
\begin{comment}
This lack of ellipticity entails that a uniform bound on the energy~$\F_\epsilon$ is not sufficient to infer compactness even in dimension~$1$ (and even in presence of mass constraints!). We exhibit an explicit example of a non-compact sequence of sets with uniformly bounded energy in Appendix~\ref{sec::non_compactness}.
\end{comment}

	\subsection*{Complementary results}
	
	Additionally, we state and prove in the appendices of this work a few interesting results concerning some uniform density estimates for almost minimizers (defined in a suitable sense, to be compared with~\cite[Definition~1.1]{sequoia}) of~$\P_K$
	and the behavior of our non-local perimeter in large balls (i.e. in~$B_R$, with~$R\to+\infty$). 
	
	\subsection*{Comparison with the existing literature}
	In relation with the existing literature for the non-local setting, we recall that
	problems entangled with phase-transition, both at the macro-scale and at the meso-scale, have been addressed in~\cite{MR1612250, MR1634336} and~\cite{pagliardini_non_local_planelike}, respectively. However, our setting is quite different, since in~\cite{MR1612250, MR1634336} the authors only consider a model with no external fields, while in~\cite{pagliardini_non_local_planelike} no~$\Gamma$-convergence results have been investigated. Moreover, we drop the presence of the double-well potential, that is peculiar of the Allen-Cahn model. 
	\medskip
	
	In our setting, we have to face two main difficulties. First, we observe that the forcing term~$g$ requires a careful treatment. Indeed, even though the energy contribution of~$g$ does not need to be positive, we still have to guarantee that the energies~$\F_\epsilon$ eventually are. Moreover, we want to avoid the contribution due to~$g$ near the boundary of the domain. 
	
	The second main concern is that, since the kernel~$K$ accounts for long-range interactions, we will need specific assumptions to ensure that~$K$ ``localizes in the limit''. In fact, on the one hand, to have a well-defined stable norm~$\phi$, we need that the ``volume term'' measured by~$K$ and the contribution due to~$g$ scale compatibly. On the other hand, we want to avoid an infinite energy contribution coming from outside the domain.
	
	\subsection*{Further developments}
	A possible further research direction could be investigating the existence of class-$A$ minimizers for~$\J$ under a mass constraint. Suppose that such minimizers exist and let~$\{E_m\}_{m>0}$ be a sequence of class-$A$ minimizers for~$\J$ such that~$|E_m|=m$. Then, in the fashion of the classical case, it is reasonable to expect that a suitable rescaling of~$E_m$ converges, as~$m\to+\infty$, to the Wulff shape (i.e. the isoperimetric set) for~$\P_\phi$. 
	\smallskip
	
	As another further development of this work, we point out that it would be interesting to investigate the differentiability of the stable norm. We refer to~\cite{MR2197072} for the analogous in the classical case on manifolds. See also~\cite{MR3223561} for a different approach using the local version of the cell problem~\eqref{eq::cell_problem}. 
	\smallskip
	
	Finally, it would be interesting to determine whether it is possible to prove the existence of planelike minimizers when the kernel~$K$ is integrable (dropping the lower bound in~\eqref{eq::K_behavior}). For instance, in~\cite{MR4360596}, the authors have proved that a suitable rescaling of the fractional perimeter $\mbox{Per}_s$ $\Gamma$-converges, as~$s\to0$, to an integral operator with an $L^1$ kernel, now commonly referred to as $0$-Perimeter.
	However, addressing this problem goes beyond the purpose of this paper.
	
	% --------------------------------------------------------------------------------------------------------------------------------------------------------
	
	\section{Existence of minimizers for~$\cE_p$ - Proof of Theorem~\ref{th::existence_minimizer}} \label{sec::existence_minimizer}
	
	The first step to construct planelike minimizers for the energy functional~$\J$ is to prove that the cell problem~\eqref{eq::cell_problem} is well-defined, as stated in Theorem~\ref{th::existence_minimizer}.
	
	To do this, we observe that
	\begin{equation*}
		\begin{split}
			&\cE_p(0) =\iint_{Q\times\R^n} \frac{1}{2}|p\cdot(x-y)|\,K(x,y)\,dx\,dy\leq\frac{|p|}{2}\iint_{Q\times\R^n} |x-y|\,K(x,y)\,dx\,dy\\
			&\qquad\qquad=\frac{|p|\,|Q|}{2}\int_{\R^n} |h| K(h,0)\,dh <+\infty,
		\end{split}
	\end{equation*}
	thanks to~\eqref{eq::K_invariance} and~\eqref{eq::K_integrable}. Hence, the minimization process takes place on a nonempty set of competitors.
	
	Accordingly, we take a minimizing sequence~$u_j\in\mathcal{W}$ with~$\cE_p(u_j)\le \cE_p(0)$. Thus, using~\eqref{eq::K_behavior}, the fact that~$g\in L^\infty(\R^n)\subseteq L^{\frac{n}{2s_1}}(Q)$, and the H\"older inequality, we obtain
	\begin{equation}\label{Pa1}
		\begin{split}
			\cE_p(0)
			&\geq \iint_{Q\times\R^n} \frac{1}{2}|u_j(x)-u_j(y)|\,K(x,y)\,dx\,dy-\iint_{Q\times\R^n} \frac{1}{2}|p\cdot(x-y)|\,K(x,y)\,dx\,dy\\
			&\quad-\int_Q |g(x)|\,|u_j(x)|\,dx\\
			&\geq \frac{ \kappa_1}{2}\int_{(Q\times\R^n)\cap\{|x-y|<\delta\}} \frac{|u_j(x)-u_j(y)|}{|x-y|^{n+2s_1}}\,dx\,dy-\cE_p(0)-\norm{g}_{L^{\frac{n}{2s_1}}(Q)}\norm{u_j}_{L^{\frac{n}{n-2s_1}}(Q)}.
		\end{split}
	\end{equation}
	
	Also, by the fractional Poincar\'e inequality (see e.g.~\cite[Theorem~6.33]{MR4567945}),
	\begin{equation}\label{K:01} 
		\norm{u_j}_{L^1(Q)}\le C\iint_{Q\times Q} \frac{|u_j(x)-u_j(y)|}{|x-y|^{n+2s_1}}\,dx\,dy
	\end{equation}
	and, thanks to the fractional Sobolev inequality (see e.g.~\cite[Theorem~6.7]{MR2944369}),
	\begin{equation}\label{K:02} 
		\norm{u_j}_{L^{\frac{n}{n-2s_1}}(Q)}\le C\left(\norm{u_j}_{L^1(Q)}+\iint_{Q\times Q} \frac{|u_j(x)-u_j(y)|}{|x-y|^{n+2s_1}}\,dx\,dy\right),
	\end{equation}
	for some constant~$C>0$ depending only on~$n$ and~$s_1$ and possibly changing from line to line.
	
	Combining~\eqref{K:01} and~\eqref{K:02}, we infer that
	\begin{equation*}
		\norm{u_j}_{L^{\frac{n}{n-2s_1}}(Q)}\le C\iint_{Q\times Q} \frac{|u_j(x)-u_j(y)|}{|x-y|^{n+2s_1}}\,dx\,dy.
	\end{equation*}
	This, together with Lemma~\ref{lemma::short_range_norm}, gives that
	\begin{equation*}\begin{split}
			\norm{u_j}_{L^{\frac{n}{n-2s_1}}(Q)}&\le C\iint_{(Q\times Q)\cap\{|x-y|<\delta\}} \frac{|u_j(x)-u_j(y)|}{|x-y|^{n+2s_1}}\,dx\,dy\\&\le C\int_{(Q\times\R^n)\cap\{|x-y|<\delta\}} \frac{|u_j(x)-u_j(y)|}{|x-y|^{n+2s_1}}\,dx\,dy,\end{split}
	\end{equation*}
	up to renaming~$C>0$, in dependence only of~$\delta$, $n$, and~$s_1$.
	
	Therefore, plugging the latter inequality into~\eqref{Pa1}, we deduce that
	\begin{equation} \label{eq::short_gagliardo_bound}
		\begin{split}
			&\frac{\kappa_1}{4} \int_{(Q\times\R^n)\cap\{|x-y|<\delta\}} \frac{|u_j(x)-u_j(y)|}{|x-y|^{n+2s_1}}\,dx\,dy\\
			&\qquad\leq \left(\frac{\kappa_1}{2}-C\norm{g}_{L^{\frac{n}{2s_1}}(Q)}\right)\int_{(Q\times\R^n)\cap\{|x-y|<\delta\}} \frac{|u_j(x)-u_j(y)|}{|x-y|^{n+2s_1}}\,dx\,dy \leq 2\cE_p(0),
		\end{split}
	\end{equation}
	as long as~$\norm{g}_{L^{\frac{n}{2s_1}}(Q)}\le\frac{\kappa_1}{4C}$.
	
	Lemma~\ref{lemma::short_range_norm}, \eqref{K:01}, and~\eqref{eq::short_gagliardo_bound} provide a uniform bound for~$u_j$ in~$W^{s_1,1}(Q)$.
	Accordingly, by~\eqref{K:02}, the sequence~$\{u_j\}_j$ is also bounded in~$L^{\frac{n}{n-2s_1}}(Q)$, and hence precompact in~$L^1(Q)$ (see e.g.~\cite[Theorem~7.1]{MR2944369}). It thereby follows that, up to a subsequence, the sequence~$\{u_j\}_j$ converges to some~$u$ in~$L^1(Q)$ and, by periodicity, a.e. in~$\R^n$. In particular,
	we have that~$u\in\mathcal{W}$.
	
	Now, by construction, and thanks to the Fatou Lemma, we have that
	\begin{equation}\label{K:10}
		\begin{split}
			&\inf_{\mathcal{W}} \cE_p =\lim_{j\to+\infty}\cE_p(u_j)\\
			&\qquad\geq \iint_{Q\times\R^n} \frac{1}{2}|u(x)-u(y)+p\cdot(x-y)|\,K(x,y)\,dx\,dy+\lim_{j\to+\infty}\int_Q g(x)\,u_j(x)\,dx.
		\end{split}
	\end{equation}
	
	Notice also that, again by the Fatou Lemma,
	\begin{equation}\label{stae54786hfbn987654}
		S:=\sup_{j\in\N}\norm{u_j}_{L^{\frac{n}{n-2s_1}}(Q)}\ge\lim_{j\to+\infty}\norm{u_j}_{L^{\frac{n}{n-2s_1}}(Q)}\ge
		\norm{u}_{L^{\frac{n}{n-2s_1}}(Q)}.\end{equation}
	
	Moreover, given~$\epsilon>0$, using the absolute continuity of the Lebesgue integral (see e.g.~\cite[Corollary~3.6]{MR1681462}), we find~$\eta>0$ such that, whenever a subset~$V$ of~$Q$ has measure less than~$\eta$, we have that
	\begin{equation}\label{K:09}
		\norm{g}_{L^{\frac{n}{2s_1}}(V)}\leq\epsilon.
	\end{equation}
	
	At this stage, we invoke the Severini-Egorov Theorem (see e.g.~\cite[Theorem~2.33]{MR1681462}) to find a set~$U\subseteq Q$ of measure less than~$\eta$ such that~$u_j$ converges to~$u$ uniformly in~$Q\setminus U$.
	As a result, we come to
	\begin{equation*}\begin{split}
			& \lim_{j\to+\infty}\left|\int_Q g(x)\,\big(u_j(x)-u(x)\big)\,dx\right|=\lim_{j\to+\infty}\left|\int_U g(x)\,\big(u_j(x)-u(x)\big)\,dx\right|\\
			&\qquad\qquad\leq\lim_{j\to+\infty}\left( \norm{u_j}_{L^{\frac{n}{n-2s_1}}(U)}+\norm{u}_{L^{\frac{n}{n-2s_1}}(U)}\right)\norm{g}_{L^{\frac{n}{2s_1}}(U)}\le2S\norm{g}_{L^{\frac{n}{2s_1}}(U)}.
	\end{split}\end{equation*}
	
	From this and~\eqref{K:09}, we infer that
	\begin{equation*}
		\lim_{j\to+\infty}\left|\int_Q g(x)\,\big(u_j(x)-u(x)\big)\,dx\right|\le\e.
	\end{equation*}
	By the arbitrariness of~$\epsilon$, we find that
	\begin{equation*} 
		\lim_{j\to+\infty}\int_Q g(x)\,\big(u_j(x)-u(x)\big)\,dx=0.
	\end{equation*}
	
	Plugging this information into~\eqref{K:10}, we obtain that 
	$$\inf_{\mathcal{W}} \cE_p \ge \iint_{Q\times\R^n} |u(x)-u(y)+p\cdot(x-y)|\,K(x,y)\,dx\,dy+\int_Q g(x)\,u(x)\,dx=\cE_p(u),$$
	showing that~$u$ is a minimizer (i.e.~\eqref{eq::u_min_cE}).
	The statement in~\eqref{eq::u_lebesgue_holder} is a consequence of~\eqref{stae54786hfbn987654}.
	
	\begin{remark} \label{rem::estimate_g_norm}
		We point out that, by inspection of the proofs of Theorem~\ref{th::existence_minimizer} and Lemma~\ref{lemma::short_range_norm}, we infer that
		\begin{equation*}
			\norm{g}_{L^{\frac{n}{2s_1}}(Q)} \leq \frac{\kappa_1}{4C}\leq \widetilde{c}\,\delta^{n+2s_1}\varphi(\delta),
		\end{equation*}
		where~$\kappa_1$, $s_1$, and~$\delta$ are as in~\eqref{eq::K_behavior}, $C$ is a multiple of the constant appearing in Lemma~\ref{lemma::short_range_norm}, $\widetilde{c}=\widetilde{c}(n,s_1,\kappa_1)$ is a positive constant, and~$\varphi$ is a function such that
		$$\lim_{\delta\to0^+}\varphi(\delta)=0.$$
	\end{remark}
	
	\section{Proof of Proposition~\ref{prop::existence_calibration} and Euler-Lagrange equation for the cell problem~\eqref{eq::cell_problem}} \label{sec::existence_calibration} %{The Euler-Lagrange equation}
	
	Here, we present the proof of Proposition~\ref{prop::existence_calibration}. 
	
	\begin{proof}[Proof of Proposition~\ref{prop::existence_calibration}]
		We first show \eqref{SUBD7} for every~$\eta\in C^\infty(\R^n)$ that is~$\Z^n$-periodic, and such that $\int_Q \eta(x)\,dx=0$.
		Let us consider
		\begin{equation*}
			\begin{split}
				\mathcal{X} := \Bigg\{u\in L^{\frac{n}{n-2s_1}}(Q)\quad &\mbox{s.t. $u$ is $\Z^n$-periodic, }\int_Q u(x)\,dx=0,\\
				&\mbox{and }\iint_{Q\times\R^n} \frac{1}{2}|u(x)-u(y)|\,K(x,y)\,dx\,dy<+\infty\Bigg\}.
			\end{split}
		\end{equation*}
		
		For any~$u\in{\mathcal{X}}$, in the spirit of~\cite[Theorem~2.3]{MR3491533} (see also~\cite{MR3930619, MR3996039, MR4537323, MR4645235}), we define the subdifferential of~$\cE_p$ at~$u$ as the collection of all~$\varphi\in{\mathcal{X}}$
		satisfying
		\begin{equation}\label{eq::subdifferential_property} 
			\cE_p(w)-\cE_p(u)\ge\int_Q \big(w(x)-u(x)\big)\,\varphi(x)\,dx,\quad\mbox{ for all~$w\in{\mathcal{X}}$,}
		\end{equation} and we denote it by~$\partial \cE_p(u)$.
		
		Now, given a minimizer~$u_p$ for~$\cE_p$ (as constructed in Theorem~\ref{th::existence_minimizer}), let us define a function~$z:\R^n\times\R^n\to\R$ as
		\begin{equation}\label{eq::z_def} 
			z(x,y):=
			\begin{cases}
				\displaystyle\frac{u_p(x)-u_p(y)+p\cdot(x-y)}{|u_p(x)-u_p(y)+p\cdot(x-y)|},\quad &{\mbox{if }} u_p(x)-u_p(y)+p\cdot(x-y)\ne0,\\
				\\
				0,\quad&{\mbox{ otherwise.}}
			\end{cases}
		\end{equation}
		
		We have that
		\begin{equation}\label{bvcnxwew87y3r3486ytgfhajleks} 
			\left|u_p(x)-u_p(y)+p\cdot(x-y)\right|z(x,y)=u_p(x)-u_p(y)+p\cdot(x-y),
		\end{equation}
		namely~\eqref{eq::z_layering} holds true.
		
		Also,
		\begin{align*} %\label{eq::z_antisymm}
			&z(x,y)=-z(y,x)\in[-1,1]\quad\mbox{for a.e. }(x,y)\in\R^n\times\R^n,\\
			\mbox{and}\quad&z(x+k,y+k)=z(x,y)\quad\mbox{for all }k\in\Z^n,
		\end{align*}
		that are~\eqref{eq::z_antisymm} and~\eqref{eq::z_translation_invariance}, respectively.
		
		Now, we claim that, for all~$\varphi\in\partial \cE_p(u_p)$ and~$\eta\in{\mathcal{X}}\cap C^\infty(\R^n)$, the function~$z$ satisfies
		\begin{equation} \label{eq::z_subdiff}
			\iint_{Q\times\R^n} \frac{1}{2} z(x,y) (\eta(x)-\eta(y))\,K(x,y)\,dx\,dy+\int_Q g(x)\,\eta(x)\,dx=\int_Q \varphi(x)\,\eta(x)\,dx.
		\end{equation}
		To check this, we recall the notation for~$v_p$ in~\eqref{eq::def_v} and we stress that, by~\eqref{bvcnxwew87y3r3486ytgfhajleks},
		\begin{equation*} 
			|v_p(x)-v_p(y)|\,z(x,y)=v_p(x)-v_p(y).
		\end{equation*}
		
		Now, let~$\epsilon\in(-1,1)\setminus\{0\}$ and~$w:=u_p+\epsilon\eta$. 
		Notice that~$w\in{\mathcal{X}}$, since both~$u_p$ and~$\eta$ are
		(recall~\eqref{eq::u_lebesgue_holder} in Theorem~\ref{th::existence_minimizer}), therefore
		we can employ~\eqref{eq::subdifferential_property} with~$u:=u_p$ and~$w$ as above and we obtain that
		\begin{equation}\label{s.0.1}
			\begin{split}
				&\epsilon\int_Q \varphi(x)\,\eta(x)\,dx-\e\int_Q g(x)\,\eta(x)\,dx\\
				\leq\;& \iint_{Q\times\R^n} \frac{1}{2}\Big(|v_p(x)-v_p(y)+\epsilon(\eta(x)-\eta(y))| -|v_p(x)-v_p(y)|\Big)\,K(x,y)\,dx\,dy\\ 
				= \;&\iint_{Q\times\R^n} \frac{1}{2}\Big(\big| |v_p(x)-v_p(y)|z(x,y)+\epsilon(\eta(x)-\eta(y))\big| -|v_p(x)-v_p(y)|\Big)\,K(x,y)\,dx\,dy\\ 
				=\;& \iint_{Q\times\R^n} \frac{1}{2}\left(\int_0^1\frac{d}{dt}\Big||v_p(x)-v_p(y)|z(x,y)+t\epsilon(\eta(x)-\eta(y))\Big| \,dt\right)\,K(x,y)\,dx\,dy\\ 
				=\;& \iint_{Q\times\R^n} \frac{1}{2}\left(\int_0^1 \epsilon\frac{\Big( |v_p(x)-v_p(y)|z(x,y)+t\e(\eta(x)-\eta(y))\Big) (\eta(x)-\eta(y))}{\Big||v_p(x)-v_p(y)|z(x,y)+t\epsilon(\eta(x)-\eta(y))\Big|} \,dt\right)\,K(x,y)\,dx\,dy.
			\end{split}
		\end{equation}
		
		In addition, for all~$t\in[0,1]$,
		\begin{equation}\label{s.0.2}
			\begin{split}
				&\left|\frac{\Big(|v_p(x)-v_p(y)|z(x,y)+t\e(\eta(x)-\eta(y))\Big) (\eta(x)-\eta(y))}{\Big||v_p(x)-v_p(y)|z(x,y)+t\epsilon(\eta(x)-\eta(y))\Big|}\right| \\
				&\qquad\qquad\qquad \leq |\eta(x)-\eta(y)|\leq 2\norm{\eta}_{C^1(\R^n)}\min\{|x-y|,1\}.
			\end{split}
		\end{equation}
		On this account, dividing by~$\epsilon$ in~\eqref{s.0.1} and discussing the sign of~$\epsilon$, thanks to the Dominated Convergence Theorem (whose application is justified by~\eqref{eq::K_integrable} and~\eqref{s.0.2}), we deduce that
		\begin{equation*}
			\begin{split}
				& \int_Q \varphi(x)\,\eta(x)\,dx-\int_Q g(x)\,\eta(x)\,dx\\
				&\qquad=\lim_{\epsilon\to0} \iint_{Q\times\R^n} \frac{1}{2}\left(\int_0^1\frac{\Big( |v_p(x)-v_p(y)|z(x,y)+t\e(\eta(x)-\eta(y))\Big) (\eta(x)-\eta(y))}{\Big||v_p(x)-v_p(y)|z(x,y)+t\e(\eta(x)-\eta(y))\Big|} \,dt\right)\,K(x,y)\,dx\,dy\\
				&\qquad=\iint_{Q\times\R^n} \frac{1}{2}\left(\int_0^1\frac{ |v_p(x)-v_p(y)|z(x,y) (\eta(x)-\eta(y))}{|v_p(x)-v_p(y)|} \,dt\right)\,K(x,y)\,dx\,dy\\
				&\qquad=\iint_{Q\times\R^n} \frac{1}{2}z(x,y) (\eta(x)-\eta(y))\,K(x,y)\,dx\,dy,
			\end{split}
		\end{equation*}
		showing~\eqref{eq::z_subdiff}.
		
		Now, by the minimality of~$u_p$, for all~$w\in \mathcal{X}\subseteq \mathcal{W}$, we have that~$\cE_p(w)-\cE_p(u)\ge0$, which
		entails that~$0\in\partial \cE_p(u_p)$.
		Hence, employing~\eqref{eq::z_subdiff} with~$\varphi\equiv0$ yields~\eqref{SUBD7} for every~$\eta\in C^\infty(\R^n)$ that is~$\Z^n$-periodic, and such that $\int_Q \eta(x)\,dx=0$.
		\smallskip
		
		Let now $\eta\in C^\infty(\R^n)$ be~$\Z^n$-periodic, and define
		\begin{equation*}
			\widetilde{\eta}(x) := \eta(x) - \int_{Q} \eta(y)\, dy.
		\end{equation*}
		In particular, $\widetilde{\eta}\in \cont^\infty(\R^n)\cap \mathcal{X}$ and~\eqref{SUBD7} holds true for $\widetilde{\eta}$.
		
		Therefore, taking advantage of \eqref{eq::g_zero_avg} and the definition of $\widetilde{\eta}$, we infer that
		\begin{equation*}
			\begin{split}
				&\iint_{Q\times\R^n}\frac{1}{2}z(x,y) (\eta(x)-\eta(y))\,K(x,y)\,dx\,dy+\int_Q g(x)\,\eta(x)\,dx \\
				&\qquad = \iint_{Q\times\R^n}\frac{1}{2}z(x,y) (\widetilde{\eta}(x)-\widetilde{\eta}(y))\,K(x,y)\,dx\,dy+\int_Q g(x)\,\widetilde{\eta}(x)\,dx=0.
			\end{split}
		\end{equation*}
		concluding the proof.
	\end{proof}
	
	By a direct computation, from Proposition~\ref{prop::existence_calibration}, we also obtain the following:
	\begin{corollary} \label{cor::existence_calibration_2}
		Let~$K:\R^n\times \R^n\to\R$ satisfy~\eqref{eq::K_invariance}, \eqref{eq::K_integrable}, \eqref{eq::K_behavior}, and~\eqref{eq::K_lower_bound_Q} and
		let~$u_p$ be the minimizer for~$\cE_p$ in~$\mathcal{W}$ given by Theorem~\ref{th::existence_minimizer}. 
		Let also ~$z:\R^n\times\R^n\to[-1,1]$ be as in Proposition~\ref{prop::existence_calibration}.
		
		Then, for all~$\eta\in C^\infty(\R^n)$ that are~$\Z^n$-periodic,
		we have that
		\begin{equation}\label{SUBD7_2}
			\begin{split}
				&\iint_{Q\times Q}\frac{1}{2}z(x,y) (\eta(x)-\eta(y))\,K(x,y)\,dx\,dy \\
				&\qquad+ \iint_{Q\times Q^c}z(x,y) \eta(x)\,K(x,y)\,dx\,dy +\int_Q g(x)\,\eta(x)\,dx=0.
			\end{split}
			%\iint_Q \eta(x) \left[\int_{\R^n} z(x,y)\,K(x,y)dy\, +g(x)\right]\,dx  = 0.
		\end{equation}
	\end{corollary}
	
	\begin{proof}
		Let~$\eta\in C^\infty(\R^n)$ be~$\Z^n$-periodic. Observe that~\eqref{SUBD7_2} will follow from~\eqref{SUBD7} if we show that 
		\begin{equation*} 
			\begin{split}
				&\iint_{Q\times\R^n}\frac{1}{2}z(x,y) (\eta(x)-\eta(y))\,K(x,y)\,dx\,dy \\
				&\qquad\qquad = \iint_{Q\times Q}\frac{1}{2}z(x,y) (\eta(x)-\eta(y))\,K(x,y)\,dx\,dy + \iint_{Q\times Q^c}z(x,y) \eta(x)\,K(x,y)\,dx\,dy,
			\end{split}
		\end{equation*}
		that is, if we show that 
		\begin{equation} \label{eq::bulk_rearrangment}
			\iint_{Q\times Q^c}z(x,y) \eta(y)\,K(x,y)\,dx\,dy = - \iint_{Q\times Q^c}z(x,y) \eta(x)\,K(x,y)\,dx\,dy.
		\end{equation}
		 
		To achieve this, notice that $Q^c = \bigcup_{j\in\Z^n\setminus\{0\}} j+Q$. Moreover, according to \cite[Remark 2.1]{MR3732175}, we have that
		\begin{equation*}
			\iint_{Q\times Q^c} \left|z(x,y) \eta(y)\,K(x,y)\right|\,dx\,dy \leq \norm{\eta}_{L^\infty(Q)} \P_K(Q) <+\infty
		\end{equation*}
		Thus, using Fubini-Tonelli's Theorem together with the periodicity of $\eta$, $z$, and~$K$ (see~\eqref{eq::z_translation_invariance} and~\eqref{eq::K_invariance}), we obtain 
		\begin{equation*}
			\begin{split}
				& \iint_{Q\times Q^c} z(x,y) \eta(y)\,K(x,y)\,dx\,dy =  \sum_{j\in\Z^n\setminus\{0\}} \iint_{Q\times (j+Q)} z(x,y) \eta(y)\,K(x,y)\,dx\,dy\\
				&\qquad\qquad = \sum_{j\in\Z^n\setminus\{0\}} \iint_{(Q-j)\times Q} z(x,y) \eta(y)\,K(x,y)\,dx\,dy = \iint_{Q^c\times Q} z(x,y) \eta(y)\,K(x,y)\,dx\,dy\\
				&\qquad\qquad = \iint_{Q\times Q^c} z(y,x) \eta(x)\,K(y,x)\,dx\,dy = -\iint_{Q\times Q^c} z(x,y) \eta(x)\,K(x,y)\,dx\,dy,
			\end{split}
		\end{equation*}
		where we have also used~\eqref{eq::z_antisymm} and the symmetry of~$K$ to obtain the last equality.
		This shows~\eqref{eq::bulk_rearrangment}, concluding the proof.
	\end{proof}
	
	As a byproduct of Corollary~\ref{cor::existence_calibration_2}, we deduce the Euler-Lagrange equation for the cell problem~\eqref{eq::cell_problem}.
	
		\begin{corollary}[Euler-Lagrange equation for the cell problem~\eqref{eq::cell_problem}]\label{cor::EL_eq}
		Let~$K:\R^n\times \R^n\to\R$ satisfy~\eqref{eq::K_invariance}, \eqref{eq::K_integrable}, \eqref{eq::K_behavior}, and~\eqref{eq::K_lower_bound_Q}. Let also ~$z:\R^n\times\R^n\to[-1,1]$ be as in Proposition~\ref{prop::existence_calibration}.
		
		Then, for every~$\eta\in C_c^\infty(\R^n)$, we have that
		\begin{equation}\label{XXXu4i3bcwq987645etwe}
			\iint_{\R^n\times\R^n} \frac{1}{2}z(x,y) (\eta(x)-\eta(y))\,K(x,y)\,dx\,dy+\int_{\R^n} g(x)\,\eta(x)\,dx=0.
			% \int_Q \eta(x) \left[\int_{\R^n} z(x,y)\,K(x,y)dy\, + \int_{Q^c} z(x,y)\,K(x,y)dy\, +g(x)\right]\,dx  = 0.
		\end{equation}
	\end{corollary}
	
	\begin{proof}
		We use the notation~$Q^{(j)}:=k_j+Q$, for some~$k_j\in\Z^n$
		and we first claim that \begin{equation}\label{bvnc483o6yt4ghjkdsgbvj}{\mbox{formula~\eqref{XXXu4i3bcwq987645etwe} holds
		true for~$\eta\in C^\infty_c(Q^{(j)})$.}}\end{equation}
		
		To do this, we extend~$\eta$ periodically in~$\R^n$ and we call this new function~$\widetilde\eta$. 
		
		Taking advantage of the periodicity of $\widetilde{\eta}$, $z$, and~$K$, we obtain 
		\begin{equation} \label{eq::EL_cube1_1}
			\begin{split}
				&\iint_{Q^{(j)}\times Q^{(j)}} \frac{1}{2}z(x,y) (\eta(x)-\eta(y))\,K(x,y)\,dx\,dy + \iint_{Q^{(j)}\times (Q^{(j)})^c}z(x,y) \eta(x)\,K(x,y)\,dx\,dy \\
				=&\iint_{Q^{(j)}\times Q^{(j)}}\frac{1}{2}z(x,y) (\widetilde\eta(x)-\widetilde\eta(y))\,K(x,y)\,dx\,dy + \iint_{Q^{(j)}\times (Q^{(j)})^c}z(x,y) \widetilde\eta(x)\,K(x,y)\,dx\,dy \\
				=&\iint_{Q\times Q}\frac{1}{2}z(x,y) (\widetilde\eta(x)-\widetilde\eta(y))\,K(x,y)\,dx\,dy + \iint_{Q\times Q^c}z(x,y) \widetilde\eta(x)\,K(x,y)\,dx\,dy .
			\end{split}
		\end{equation}
		Besides, from the periodicity of~$\widetilde{\eta}$ and~$g$, we deduce that
		\begin{equation} \label{eq::EL_cube1_2}
				\int_{Q^{(j)}} g(x)\,\eta(x)\,dx = \int_{Q^{(j)}} g(x)\,\widetilde\eta(x)\,dx =\int_{Q} g(x)\,\widetilde\eta(x)\,dx.
		\end{equation}
		Therefore, using~\eqref{SUBD7_2} with~$\eta:=\widetilde\eta$,~\eqref{eq::EL_cube1_1}, and~\eqref{eq::EL_cube1_2}, we infer that
		\begin{equation} \label{eq::EL_cube1}
			\begin{split}
				&\iint_{Q^{(j)}\times Q^{(j)}} \frac{1}{2}z(x,y) (\eta(x)-\eta(y))\,K(x,y)\,dx\,dy \\
				&\qquad+ \iint_{Q^{(j)}\times \left(Q^{(j)}\right)^c}z(x,y) \eta(x)\,K(x,y)\,dx\,dy +\int_{Q^{(j)}} g(x)\,\eta(x)\,dx =0.
			\end{split}
		\end{equation}
		\begin{comment}
			&\int_{Q^{(j)}} \eta(x) \left[\int_{\R^n} z(x,y)\,K(x,y)dy\, +g(x)\right]\,dx \\
			&\qquad = \int_{Q^{(j)}} \widetilde\eta(x) \left[\int_{\R^n} z(x,y)\,K(x,y)dy\, +g(x)\right]\,dx  \\
			&\qquad = \int_Q \widetilde\eta(x) \left[\int_{\R^n} z(x,y)\,K(x,y)dy\, +g(x)\right]\,dx  = 0.
		\end{comment}
		
		Now, recall that 
		\begin{equation*}
			(Q^{(j)})_\sharp = (\R^n\times\R^n) \setminus \big((Q^{(j)})^c\times (Q^{(j)})^c\big).
		\end{equation*}
		Thus, since $\mbox{supp}(\eta)\Subset Q^{(j)}$, \eqref{eq::EL_cube1} entails that
		\begin{equation*}
			\begin{split}
				&\iint_{\R^n\times\R^n} \frac{1}{2}z(x,y) \left(\eta(x)-\eta(y)\right)\,K(x,y)\,dx\,dy+\int_{\R^n} g(x)\,\eta(x)\,dx\\
			 &\qquad= \iint_{(Q^{(j)})_\sharp} \frac{1}{2}z(x,y) \left(\eta(x)-\eta(y)\right)\,K(x,y)\,dx\,dy+\int_{Q^{(j)}} g(x)\,\eta(x)\,dx\\
	&\qquad= \iint_{Q^{(j)}\times Q^{(j)}} \frac{1}{2}z(x,y) \left(\eta(x)-\eta(y)\right)\,K(x,y)\,dx\,dy \\&\qquad\qquad+ \iint_{Q^{(j)}\times (Q^{(j)})^c} z(x,y) \eta(x)\,K(x,y)\,dx\,dy
+\int_{Q^{(j)}} g(x)\,\eta(x)\,dx\\
&\qquad= 0,
			\end{split}
		\end{equation*}
		where we have also used \eqref{eq::z_antisymm} to obtain the second equality. This establishes the claim in~\eqref{bvnc483o6yt4ghjkdsgbvj}.
		
		Let now~$\eta\in C^\infty_c(\R^n)$ with~$\mbox{supp}(\eta)\Subset\Omega$, for some domain~$\Omega\subseteq\R^n$, and consider a finite covering of~$\Omega$ made of unit cubes $Q^{(1)},\dots,Q^{(N)}$, with~$Q^{(j)}=k_j+Q$, for some~$k_j\in\Z^n$.
		
		Let also~$\xi_m$ be a sequence of functions in~$C_c^\infty(Q,[0,1])$ such that~$\xi_m\to\chi_Q$ (pointwise), as~$m\to+\infty$. 
		
		We set
		$$\eta_{m,j}(x):=\eta(x)\xi_m(x-k_j)$$  %$$\eta_{m,j}(x+k_j):=\eta(x+k_j)\xi_m(x)$$ 
		and observe that
		\begin{equation} \label{eq::compact_approx}
\eta_{m,j}\in C^\infty_c(Q^{(j)}).
		\end{equation} 
		Moreover, for all $x\in\R^n$,
		\begin{equation} \label{eq::eta_partition_of_unity}
			\begin{split}
				&\eta(x)=\eta(x)\sum_{j=1}^N\chi_{Q^{(j)}}(x)=\eta(x)\sum_{j=1}^N\chi_{Q}(x-k_j)\\
				&\qquad=\lim_{m\to+\infty}\eta(x)\sum_{j=1}^N\xi_m(x-k_j)=\lim_{m\to+\infty}\sum_{j=1}^N\eta_{m,j}(x).
			\end{split}
		\end{equation}
		
		In light of~\eqref{eq::compact_approx} and~\eqref{bvnc483o6yt4ghjkdsgbvj}, we can employ~\eqref{XXXu4i3bcwq987645etwe} with~$\eta:=\eta_{m,j}$, obtaining that
		\begin{equation} \label{eq::EL_approx_cube}
				\iint_{\R^n\times\R^n} \frac{1}{2}z(x,y) \left(\eta_{m,j}(x)-\eta_{m,j}(y)\right)\,K(x,y)\,dx\,dy+\int_{\R^n} g(x)\,\eta_{m,j}(x)\,dx = 0.
		\end{equation}
		Therefore, recalling~\eqref{eq::eta_partition_of_unity}, we sum~\eqref{eq::EL_approx_cube} over~$j$ and pass to the limit as~$m\to+\infty$ (using the Dominated Convergence Theorem) to obtain the desired result.
\end{proof}

	\section{Minimality of level sets - Proof of Theorem~\ref{th::level_sets_minimality}} \label{sec::level_sets_minimality}
	
	This section is committed to the proof of Theorem~\ref{th::level_sets_minimality}. Our strategy goes as follows: we first prove a weaker version of Theorem~\ref{th::level_sets_minimality}, in which the conclusion holds only for a.e.~$t\in\R$, then we conclude thanks to the closedness of class-$A$ minimizers with respect to the~$L^1_{\loc}$-convergence. Therefore, we now focus on showing the following:
	
	\begin{proposition} \label{prop::level_sets_minimality_ae}
		Let~$K:\R^n\times \R^n\to\R$ satisfy~\eqref{eq::K_invariance}, \eqref{eq::K_integrable}, \eqref{eq::K_behavior}, and~\eqref{eq::K_lower_bound_Q} and
		let~$u_p$ be a minimizer for~$\cE_p$ in~$\mathcal{W}$. 
		Let also~$v_p$ and~$E_{p,t}$ be as in~\eqref{eq::def_v} and~\eqref{eq::def_lev_set}, respectively.
		
		Then, for a.e.~$t\in\R$, the set~$E_{p,t}$ is a class-$A$ minimizer for~$\J$ in the sense of Definition~\ref{def::classA_minimilatity}.
	\end{proposition}
	
	Our argument relies on the following technical results:
	
	\begin{lemma}\label{eq::finite_P_K}
	Let~$K:\R^n\times \R^n\to\R$ satisfy~\eqref{eq::K_integrable} and~\eqref{eq::K_behavior}.
	
	Then, for every set~$E$, we have that
	$$ \P_K(E,\Omega)<+\infty.$$
	\end{lemma}
	
	\begin{proof} Thanks to the assumptions on~$K$
	in~\eqref{eq::K_integrable} and~\eqref{eq::K_behavior}, we have that
		\begin{equation*} 
			\begin{split}
				&\P_K(E,\Omega)
				=\iint_{{\Omega}_\sharp} \frac{1}{2}|\chi_E(x)-\chi_E(y)|\,K(x,y)\,dx\,dy = \iint_{{\Omega}_\sharp}\chi_E(x)\chi_{E^c}(y)K(x,y)\,dx\,dy \\
				&\qquad= \iint_{{\Omega}_\sharp\cap\{|x-y|<\delta\}}\chi_E(x)\chi_{E^c}(y)K(x,y)\,dx\,dy+\iint_{{\Omega}_\sharp\cap\{|x-y|\geq\delta\}} \chi_E(x)\chi_{E^c}(y)K(x,y)\,dx\,dy\\
			&\qquad\leq  \iint_{{\Omega}_\sharp} \frac{\kappa_2\chi_E(x)\chi_{E^c}(y)}{|x-y|^{n+2s_1}}\,dx\,dy +\frac{2}{\delta}\iint_{\Omega\times\R^n}|x-y|K(x,y)\,dx\,dy\\
			&\qquad	\leq  \iint_{E\times E^c} \frac{\kappa_2}{|x-y|^{n+2s_1}}\,dx\,dy +\frac{2|\Omega|}{\delta}\int_{\R^n}|h|K(h,0)\,dh\\&\qquad<+\infty,
			\end{split}
		\end{equation*} as desired.
	\end{proof}
	
	\begin{lemma} \label{lemma::level_sets_finite_perimeter}
		Let~$K:\R^n\times \R^n\to\R$ satisfy~\eqref{eq::K_integrable} and~\eqref{eq::K_behavior}
		and let~$u_p$ be a minimizer for~$\cE_p$ in~$\mathcal{W}$. 		
		Let also~$v_p$ and~$E_{p,t}$ be as in~\eqref{eq::def_v} and~\eqref{eq::def_lev_set}, respectively.
		
		Then, for a.e.~$t\in\R$,
		\begin{equation} \label{eq::Per_K_alternative}
			\P_K(E_{p,t},\Omega) = \iint_{\Omega_\sharp} \frac{1}{2} z(x,y)\Big(\chi_{E_{p,t}}(x)-\chi_{E_{p,t}}(y)\Big)\,K(x,y)\,dx\,dy <+\infty.
		\end{equation}
	\end{lemma}
	
	\begin{proof} 
		First, observe that, thanks to Lemma~\ref{eq::finite_P_K},
		$$ \P_K(E_{p,t},\Omega) <+\infty.$$
			
		Now, by the layer cake representation (or by direct inspection of the integrals), for any function~$f$, we have that
		$$ f(x)-f(y)=\int_{-\infty}^{+\infty}\Big(\chi_{\{f> t\}}(x)-\chi_{\{f> t\}}(y)\Big)\,dt,$$
		as well as
		$$ |f(x)-f(y)|=\int_{-\infty}^{+\infty}\Big|\chi_{\{f> t\}}(x)-\chi_{\{f> t\}}(y)\Big|\,dt.$$
		Thus, by virtue of~\eqref{eq::z_layering}, we come to
		\begin{equation} \label{eq::z_layer_cake}
			\begin{split}
				z(x,y)\int_{-\infty}^{+\infty}\Big(\chi_{E_{p,t}}(x)-\chi_{E_{p,t}}(y)\Big)\,dt
				&= z(x,y)\int_{-\infty}^{+\infty}\Big(\chi_{\{v_p> t\}}(x)-\chi_{\{v_p> t\}}(y)\Big)\,dt\\
				&= z(x,y)\big(v_p(x)-v_p(y)\big)\\
				&= \big|v_p(x)-v_p(y)\big|\\
				&=\int_{-\infty}^{+\infty}\Big|\chi_{\{v_p> t\}}(x)-\chi_{\{v_p> t\}}(y)\Big|\,dt\\
				&=\int_{-\infty}^{+\infty}\Big|\chi_{E_{p,t}}(x)-\chi_{E_{p,t}}(y)\Big|\,dt.
			\end{split}
		\end{equation}
		
		Let us introduce the short notation
		$$\Upsilon(t):=\Big|\chi_{E_{p,t}}(x)-\chi_{E_{p,t}}(y)\Big|-z(x,y)\Big(\chi_{E_{p,t}}(x)-\chi_{E_{p,t}}(y)\Big).$$
		Then, using that~$|z(x,y)|\le1$ and~\eqref{eq::z_layer_cake}, we infer that~$\Upsilon\ge0$, and
		$$\int_{-\infty}^{+\infty}\Upsilon(t)\,dt=0.$$
		
		As a consequence, $\Upsilon$ vanishes for a.e.~$t\in\R$, and hence
		$$\Big|\chi_{E_{p,t}}(x)-\chi_{E_{p,t}}(y)\Big|=z(x,y)\Big(\chi_{E_{p,t}}(x)-\chi_{E_{p,t}}(y)\Big).$$
		This leads to
		\begin{equation*}
			\P_K(E_{p,t},\Omega) =\iint_{{\Omega}_\sharp} \frac{1}{2} z(x,y)\Big(\chi_{E_{p,t}}(x)-\chi_{E_{p,t}}(y)\Big)\,K(x,y)\,dx\,dy , 
		\end{equation*}
		as desired.
	\end{proof}
	
	\begin{proof} [Proof of Proposition~\ref{prop::level_sets_minimality_ae}]
		%Let us consider a domain~$\Omega\subset \R^n$, and assume that~$\P_K(F,\Omega)<+\infty$, otherwise there is nothing to prove.
		Let~$t\in\R$. We consider a set~$F$ such that~$F\setminus\Omega=E_{p,t}\setminus\Omega$. By Lemma~\ref{eq::finite_P_K}, we have that~$\P_K(F,\Omega)<+\infty$.
		
		We define the function~$\widetilde{\chi}:=\chi_{E_{p,t}}-\chi_F$. We notice that~$\widetilde{\chi}=0$ in~$\Omega^c$ and
		\begin{equation*}
			\norm{\widetilde{\chi}}_{L^1(\Omega)} + \iint_{{\Omega_\sharp}} \frac{1}{2}|\widetilde{\chi}(x)-\widetilde{\chi}(y)| K(x,y)\,dx\,dy
			\leq 2|\Omega| +\P_K(E_{p,t},\Omega)+\P_K(F,\Omega) < +\infty.
		\end{equation*}
		Thus, by~\cite[Theorem 6]{MR3310082}, 
		there exists\footnote{On the one hand, we point out that Theorem 6 in~\cite{MR3310082} is stated and proved for general open sets~$\Omega\subseteq\R^n$ with continuous boundary and the space 
			$$ X^{s,p}_0 := \left\{u\in L^p(\Omega) \mbox{ s.t. } \norm{u}_{L^p(\Omega)} + \iint_{\R^n\times\R^n} |u(x)-u(y)|K(x-y)\,dx\,dy <+\infty \mbox{ and }u=0 \mbox{ in }\Omega^c \right\},$$
			for some~$p\geq1$, and~$K:\R^n\setminus\{0\}\to\R$. In particular, the shape of the kernel~$K$ plays a role in proving~\cite[Theorem~6]{MR3310082} (see in particular~\cite[Lemma~12]{MR3310082}).
			
			On the other hand, by inspection of the proofs contained there, we see that analogous results with a more general kernel~$K:\R^n\times\R^n\to\R$ satisfying~\eqref{eq::K_invariance}, \eqref{eq::K_integrable},~\eqref{eq::K_behavior}, and~\eqref{eq::K_lower_bound_Q} can be established using the same arguments and considering the space~$X^{s,p}_0\cap L^\infty(\Omega)$, as in the setting studied here.}  
		a sequence of functions~$\eta_\ell\in C_c^\infty(\Omega,[-1,1])$ such that
		\begin{eqnarray} 
			\label{eq::approx_K} &&	\lim_{\ell\to+\infty} \iint_{{\Omega}_\sharp}  \frac{1}{2}\Big|\big(\chi_{E_{p,t}}-\chi_F-\eta_\ell\big)(x)-\big(\chi_{E_{p,t}}-\chi_F-\eta_\ell\big)(y)\Big|\,K(x,y)\,dx\,dy=0	\\
			\label{eq::approx_L1} \mbox{and}\quad &&\lim_{\ell\to+\infty} \norm{\chi_{E_{p,t}}-\chi_F-\eta_\ell}_{L^1(\Omega)}=0.
		\end{eqnarray}
		Owing to Corollary~\ref{cor::EL_eq}, we have that, for every~$\ell\in\N$,
		%$$\iint_{\Omega} \eta_\ell(x) \left[\int_{\R^n} z(x,y)\,K(x,y)\,dy\,+\int_{\Omega^c} z(x,y)\,K(x,y)\,dy\,+g(x)\right] \,dx=0.$$
		$$\iint_{\R^n\times\R^n} \frac{1}{2} z(x,y) \big(\eta_\ell(x)-\eta_\ell(y)\big)\,K(x,y)\,dx\,dy+\int_{\R^n} g(x)\,\eta_\ell(x)\,dx=0.$$
		
		In this way, using also the Dominated Convergence Theorem, we obtain
		\begin{equation}\label{eq::limit_calibration}
			\begin{split}
				&\int_{F\cap \Omega}g(x)\,dx-\int_{E_{p,t}\cap \Omega}g(x)\,dx =-\int_{\R^n} g(x)\,\big(\chi_{E_{p,t}}(x)-\chi_{F}(x)\big)\,dx\\
				&\qquad =-\lim_{\ell\to+\infty} \int_{\R^n} g(x)\,\eta_\ell(x)\,dx\\&\qquad=\lim_{\ell\to+\infty}\iint_{\R^n\times\R^n} \frac{1}{2} z(x,y) \big(\eta_\ell(x)-\eta_\ell(y)\big)\,K(x,y)\,dx\,dy.
			\end{split}
		\end{equation}
		
		Now, from Lemma~\ref{lemma::level_sets_finite_perimeter}, we know that
		\begin{equation*}
			\P_K(E_{p,t},\Omega)=\iint_{{\Omega}_\sharp} \frac{1}{2} z(x,y)\Big(\chi_{E_{p,t}}(x)-\chi_{E_{p,t}}(y)\Big)\,K(x,y)\,dx\,dy.
		\end{equation*}
		Since~$|z|\le1$, we also have that
		\begin{equation*}
			\P_K(F,\Omega) %\geq\iint_{{\Omega}_\sharp} |z(x,y)|\Big|\chi_{F}(x)-\chi_{F}(y)\Big|\,K(x,y)\,dx\,dy,
			\geq\iint_{{\Omega}_\sharp} \frac{1}{2} z(x,y)\Big(\chi_{F}(x)-\chi_{F}(y)\Big)\,K(x,y)\,dx\,dy
		\end{equation*}
		and hence
		\begin{equation*}
				 \P_K(E_{p,t},\Omega)-\P_K(F,\Omega)
				 \le\iint_{{\Omega}_\sharp} \frac{1}{2} z(x,y)\Big(\big(\chi_{E_{p,t}}-\chi_{F}\big)(x)-\big(\chi_{E_{p,t}}-\chi_{F}\big)(y)\Big)\,K(x,y)\,dx\,dy.
		\end{equation*}
		
		Using this and~\eqref{eq::approx_K}, we find that
		\begin{equation*}
			\begin{split}
				\P_K(E_{p,t},\Omega)-\P_K(F,\Omega)
				&\le\lim_{\ell\to+\infty} \iint_{{\Omega}_\sharp} \frac{1}{2} z(x,y)\big(\eta_\ell(x)-\eta_\ell(y)\big)\,K(x,y)\,dx\,dy\\
				& =
				\lim_{\ell\to+\infty} \iint_{\R^n\times\R^n} \frac{1}{2} z(x,y)\big(\eta_\ell(x)-\eta_\ell(y)\big)\,K(x,y)\,dx\,dy,
			\end{split}
		\end{equation*}
		where the last equality is due to the fact that~$\eta_\ell$ vanishes outside~$\Omega$.
		
		This, in tandem with~\eqref{eq::limit_calibration}, yields 
		\begin{equation} \label{eq::cluster_minimality2}
			\P_K(E_{p,t},\Omega)-\P_K(F,\Omega)\le\int_{F\cap \Omega}g(x)\,dx-\int_{E_{p,t}\cap \Omega}g(x)\,dx,
		\end{equation}
		showing the class-$A$ minimality of $E_{p,t}$. 
	\end{proof}
	
	Now, since we think it is interesting  \textit{per se}, we show that the functional~$\J$ is continuous with respect to the~$L^1_{\loc}$-convergence of class-$A$ minimizers. As a byproduct, we obtain that  class-$A$ minimizers are closed with respect to the~$L^1_{\loc}$-convergence, i.e. the limit of a convergent sequence of class-$A$ minimizers in~$L^1_{\loc}$ is a class-$A$ minimizer itself. Our proof follows the ideas of~\cite[Theorem~3.3]{MR2675483} (see also~\cite[Proposition~2.1]{sequoia}).
	
	\begin{lemma} \label{lemma::still_class_A}
		Let~$K:\R^n\times \R^n\to\R$ satisfy~\eqref{eq::K_integrable} and~\eqref{eq::K_behavior} and let~$\{E_j\}_j$ be a sequence of class-$A$ minimizers for~$\J$ such that~$E_j\to E$ in~$L^1_{\loc}(\R^n)$, for some set~$E\subseteq\R^n$.
		
		Then, $E$ is a class-$A$ minimizer for~$\J$. 
		
		Moreover, for any domain~$\Omega\subseteq\R^n$, we have
		\begin{equation} \label{eq::continuity_J}
			\lim_{j\to+\infty} \J(E_j,\Omega) = \J(E,\Omega).
		\end{equation}
	\end{lemma}
	
	\begin{proof}
		Let~$\Omega$ be a smooth domain of~$\R^n$. Let also~$F\subseteq \R^n$ be a competitor for~$E$ in~$\Omega$, namely~$F\setminus\Omega= E\setminus\Omega$. Define the sequence of sets~$\{F_j\}_j$ as
		\begin{equation*}
			F_j := (F\cap\Omega)\cup(E_j\cap \Omega^c).
		\end{equation*}
		In particular, we have~$F_j\setminus\Omega= E_j\setminus\Omega$ and therefore,
	by the class-$A$ minimality of~$E_j$,
	$$  \J(E_j,\Omega)\le \J(F_j,\Omega).$$
Consequently, thanks to
the lower semicontinuity of~$\P_K$ and the~$L^1_{\loc}$-convergence of~$E_j$ to~$E$ (together with the fact that~$g$ is bounded), we infer that 
		\begin{equation} \label{eq::fatou_upper_bound}
			\J(E,\Omega) \leq \liminf_{j\to+\infty} \J(E_j,\Omega) \leq \liminf_{j\to+\infty} \J(F_j,\Omega).
		\end{equation}
		
		Now we observe that
		\begin{equation*}\begin{split}
			\J(F_j,\Omega) &= \P_K(F_j,\Omega)+\int_{F_j\cap\Omega} g(x)\,dx\\&=\P_K(F,\Omega)
			 +2\iint_{\Omega\times\Omega^c} \chi_{F^c}(x)\big(\chi_{E_j}(y)-\chi_{F}(y)\big)K(x,y)\,dx\,dy
				+
		\int_{F\cap\Omega} g(x)\,dx		
				\\&=
				\J(F,\Omega) + 2\iint_{\Omega\times\Omega^c} \chi_{F^c}(x)\big(\chi_{E_j}(y)-\chi_{E}(y)\big)K(x,y)\,dx\,dy.\end{split}
		\end{equation*}
		This, recalling also~\eqref{eq::K_behavior}, yields that
		\begin{equation} \label{eq::estimate_energy_competitor}
			\begin{split}
				|\J(F_j,\Omega) - \J(F,\Omega)| 
				\leq\;&  2\iint_{\Omega\times\Omega^c} \chi_{F^c}(x)\chi_{E_j \Delta E}(y)K(x,y)\,dx\,dy\\
				\leq\;&2 \kappa_2\iint_{\Omega\times\Omega^c}\frac{\chi_{E_j \Delta E}(y)}{|x-y|^{n+2s_1}}\,dx\,dy\\=:\;&b_j.
		\end{split}
	\end{equation}
		
		Now, we have that		\begin{equation*} 
			\lim_{j\to+\infty }b_j = 0,
		\end{equation*}
		see~\cite[Appendix~A]{sequoia} (see also the proof of~\cite[Theorem~3.3]{MR2675483}).

It follows from this and~\eqref{eq::estimate_energy_competitor} that
		\begin{equation}\label{eq::competitor}
			\limsup_{j\to+\infty} \J(F_j,\Omega) \leq \limsup_{j\to+\infty} \left(\J(F,\Omega) + b_j\right) = \J(F,\Omega).
		\end{equation}
		This and~\eqref{eq::fatou_upper_bound} give that 
		\begin{equation} \label{eq::min_Omega}
			\J(E,\Omega) \leq \J(F,\Omega),
		\end{equation}
hence~$E$ is a minimizer for~$\J$ in~$\Omega$.
		
		In addition, choosing
		\begin{equation*}
			F_j := (E\cap\Omega) \cup (E_j\cap \Omega^c),
		\end{equation*}
		we deduce from~\eqref{eq::fatou_upper_bound} and~\eqref{eq::competitor} that
		\begin{equation*}
				\J(E,\Omega) \leq \liminf_{j\to+\infty} \J(E_j,\Omega) \leq \limsup_{j\to+\infty} \J(E_j,\Omega) \leq \J(E,\Omega),
		\end{equation*}
		which entails~\eqref{eq::continuity_J} for every smooth domain~$\Omega$.
		
		Moreover, if~$\partial \Omega$ is not~$\cont^\infty$, by interior smooth approximation, there exists a sequence~$\{\Omega^h\}_h$ of smooth domains such that~$\Omega^h\subseteq\Omega$, for every~$h$, and~$\Omega^h\to\Omega$ locally in~$\R^n$. Then, since~$\P_K$ is continuous with respect to the~$L^1_{\loc}$-convergence of the domain, we obtain~\eqref{eq::continuity_J} and~\eqref{eq::min_Omega} for every domain~$\Omega\subseteq\R^n$.
		
		By the arbitrariness of~$\Omega$, we conclude that~$E$ is class-$A$ minimizer for~$\J$.
		\end{proof}
	
	\begin{proof} [Proof of Theorem~\ref{th::level_sets_minimality}]
		Let~$t\in\R$
		and consider a sequence~$\{t_j\}_j\subseteq\R$  such that~$t_j\to t$, as~$j\to+\infty$, and Proposition~\ref{prop::level_sets_minimality_ae} holds true for every~$t_j$. Moreover, up to take a subsequence, suppose that
		$$t\leq t_j\leq t_{j-1}\leq \dots\leq t_0.$$
		
		Thus, recalling that~$E_{p,t}=\{v_p >t\}$ (see~\eqref{eq::def_lev_set}),
		we have that, for a.e.~$x\in\R^n$,
		\begin{equation*}
	\lim_{j\to+\infty}
	|\chi_{E_{p,t}}(x)-\chi_{E_{p,t_j}}(x)| 
	= \lim_{j\to+\infty}|\chi_{\{v_p>t\}}(x)-\chi_{\{v_p>t_j\}}(x)| 
	= \lim_{j\to+\infty}\chi_{\{t<v_p\leq t_j\}}(x)=0.
		\end{equation*}
Therefore, for any domain~$\Omega\subseteq\R^n$, thanks to the Dominated Convergence Theorem (with dominant~$\chi_{E_{p,t}}\in L^1(\Omega)$), we obtain
		\begin{equation*}
			\lim_{j\to+\infty} \int_\Omega |\chi_{E_{p,t}}(x)-\chi_{E_{p,t_j}}(x)|\, dx =0.
		\end{equation*}
		
		Hence,
		\begin{equation*} %\label{eq::conv_loc_level_sets}
			E_{p,t_j}\to E_{p,t} \quad\mbox{in }L^1_{\loc}(\R^n).
		\end{equation*}
	Thus, thanks to Lemma~\ref{lemma::still_class_A}, we infer that also~$E_{p,t}$ is a class-$A$ minimizer for~$\J$
	% Moreover~$E_{p,t}$ satisfies~\eqref{eq::cluster_minimality} 
	and this concludes the proof.
	\end{proof}
	
	% ----------------------------------------------------------------------------------------------------------------------------------------------------
	\section{Density estimates for minimizers of the functional~$\F$} \label{sec::density estimates}
	
	In this section we show that the level sets~$E_{p,t}$ defined in~\eqref{eq::def_lev_set} satisfy some uniform density estimates. Heuristically, if~$x_0\in\partial E_{p,t}$, then, for any small~$r>0$, the sets~$B_r(x_0)\cap E_{p,t}$ and~$B_r(x_0)\cap E_{p,t}^c$ have comparable measure. The precise statement goes as follows:
	
	\begin{proposition} \label{prop::UDE}
		Let~$K:\R^n\times \R^n\to\R$ satisfy~\eqref{eq::K_invariance}, \eqref{eq::K_integrable}, \eqref{eq::K_behavior}, and~\eqref{eq::K_lower_bound_Q} and
		let~$u_p$ be a minimizer for~$\cE_p$ in~$\mathcal{W}$.
		% among the set of functions $u\in L^1_{\operatorname{loc}}(\R^n)$ which are $\Z^n$-periodic and such that~$\int_{Q} u(x)\,dx=0$.
		
		Let also~$v_p(x)$ and~$E_{p,t}$ be defined as in~\eqref{eq::def_v} and~\eqref{eq::def_lev_set}, respectively. 
		
		Then, there exists a constant~$c_0\in(0,1)$, depending only on~$n$, $s_1$, and~$\kappa_2$, such that, for any~$x_0\in(\partial E)\cap Q$ and~$r\in\left(0,\min\{\delta/4,\mbox{dist}(x_0,\partial Q)\}\right)$, 
		\begin{equation} \label{eq::UDE}
			c_0r^n \leq |E_{p,t}\cap B_r(x_0)| \leq (1-c_0)r^n,
		\end{equation}
		whenever~$\norm{g}_{L^{n/2s_1}(Q)}<(c_0/2)^{-n/2s_1}$. 
		\end{proposition}
	
	This result is a direct consequence of the fact that
	every set~$E_{p,t}$ is almost minimal in the following sense:
	
	\begin{definition}[Almost-minimality] \label{def::almost_minimality}
		Let~$\Omega\subseteq \R^n$ be a Lipschitz domain.
		Let~$\Lambda>0$ and~$\xi\in\R$. We say that a set~$E\subseteq \R^n$ with finite~$K$-perimeter is~$(\Lambda,\xi)$-minimal in~$\Omega$ if
		\begin{equation*}
			P_K(E,\Omega) \leq \P_K(F,\Omega) + \Lambda|E\Delta F|^{1-\xi},
		\end{equation*}
		for any set~$F\subseteq \R^n$ such that~$F\setminus \Omega = E\setminus \Omega$.
	\end{definition}
	
	In~\cite{sequoia}, the authors have proved the
	H\"older-regularity for~$(\Lambda,0)$-minimizers of the~$2s$-perimeter (compare Definition~\ref{def::almost_minimality} here with Definition~1.1 in~\cite{sequoia}). In particular, they have shown in~\cite[Theorem~2.2]{sequoia} that~$(\Lambda,0)$-minimal sets satisfy suitable uniform density estimates. A slight modification of the argument used there produces uniform density estimates for~$\left(\Lambda,\frac{2s}{n}\right)$-minimal sets, for any~$s\in(0,1/2)$. We refer to Appendix~\ref{appendix::ude} for a detailed proof. Here, we focus on proving the following:
	
	\begin{lemma} \label{lemma::level_sets_almost_minimality}
		Let~$K:\R^n\times \R^n\to\R$ satisfy~\eqref{eq::K_invariance}, \eqref{eq::K_integrable}, \eqref{eq::K_behavior}, and~\eqref{eq::K_lower_bound_Q} and
		let~$u_p$ be a minimizer for~$\cE_p$ in~$\mathcal{W}$.
		% among the set of functions $u\in L^1_{\operatorname{loc}}(\R^n)$ which are $\Z^n$-periodic and such that~$\int_{Q} u(x)\,dx=0$.
		
		Let also~$v_p$ and~$E_{p,t}$ be defined as in~\eqref{eq::def_v} and~\eqref{eq::def_lev_set}, respectively. 
		
		Then, $E_{p,t}$ is a~$\left(\norm{g}_{L^{n/2s_1}(Q)},\frac{2s_1}{n}\right)$-minimal set in~$Q$ (in the sense of Definition~\ref{def::almost_minimality}).
	\end{lemma}
	
	\begin{proof}
		Let~$F\subseteq \R^n$ be a set such that~$F\setminus Q = E_{p,t}\setminus Q$.	Then, by Theorem~\ref{th::level_sets_minimality} and using the H\"older inequality, we infer that
		\begin{equation*}
			\begin{split}
				&\P_K(E_{p,t},Q)-\P_K(F,Q) = \J(E_{p,t},Q) - \J(F,Q) - \int_{Q} g(x)\left(\chi_{E_{p,t}}(x)-\chi_F(x)\right)dx\\
				&\qquad\qquad \leq \int_{Q} |g(x)| \left|\chi_{E_{p,t}}(x)-\chi_F(x)\right|dx \leq \norm{g}_{L^{n/2s_1}(Q)} |E_{p,t}\Delta F|^{1-\frac{2s_1}{n}}.\qedhere
			\end{split}
		\end{equation*}
	\end{proof}

With this, we can now complete the proof of Proposition~\ref{prop::UDE}.
	
	\begin{proof} [Proof of Proposition~\ref{prop::UDE}]
		According to Lemma~\ref{lemma::level_sets_almost_minimality}, $E_{p,t}$ is a~$\left(\norm{g}_{L^{n/2s_1}(Q)},\frac{2s_1}{n}\right)$-minimal set in~$Q$ in the sense of Definition~\ref{def::almost_minimality}. Thus, Proposition~\ref{prop::ude_almost_min} yields the desired result provided that~$\norm{g}_{L^{n/2s_1}(Q)}<(c_0/2)^{-n/2s_1}$ (e.g up to considering~$\gamma$ in Theorem~\ref{th::existence_minimizer} small enough).
	\end{proof}
	
	\section{Minimizers of~$\cE_p$ have controlled oscillations - Proof of Theorem~\ref{th::controlled_osc}} \label{sec::osc}
	
	In this section, we focus on the proof of Theorem~\ref{th::controlled_osc}. Our argument relies on the following result which provides a lower bound for~$\P_K$. 
	
In the following, we will use the notation~$\widetilde{Q}_{{q},\zeta}:={q}+(-\zeta,\zeta)^n$, with~$q\in\R^n$ and~$\zeta>0$.
	
	\begin{lemma} \label{lemma::per_lower_bound}
		%Let $\Omega$ be a Lipschitz domain, and l
		Let~$E\subseteq\R^n$ satisfy the uniform density estimates in Proposition~\ref{prop::UDE} in a Lipschitz domain~$\Omega\subseteq\R^n$. 
		
		Moreover, suppose that there exists a cube~$\widetilde{Q}_{\bar{q},\bar\zeta}\subseteq \Omega$, for some~$\bar q\in\Omega$ and some small~$\bar\zeta>0$, such that~$\widetilde{Q}_{\bar q,3\bar\zeta}\subseteq \Omega$ and~$\widetilde{Q}_{\bar q,\bar \zeta}\cap\partial E\neq\varnothing$.
		
		Then, there exists a constant~$C>0$, depending only on~$n$, $s_1$, $\kappa_1$, $\kappa_2$, $\delta$, and~$\Omega$, such that
		\begin{equation} \label{eq::per_lower_bound}
			P_K(E,\Omega) \geq C.
		\end{equation}
	\end{lemma}
	
	\begin{proof}
		Let us decompose~$\R^n$ into disjoint	
		cubes of side~$\zeta$, for some given~$\zeta\in\left(0,\min\left\{\frac{\delta}{6\sqrt{n}},\bar\zeta\right\}\right)$, with~$\delta$ as in~\eqref{eq::K_behavior}.
Consider the set
	$$ \text{\Fontauri{Q}}_{\ \zeta}:=\big\{\widetilde{Q}_{q,\zeta} %\widetilde{Q}_{x,\zeta} \text{ is a cube of side~$\zeta$ 
	\text{ such that }
\widetilde{Q}_{q,\zeta}\subseteq \widetilde{Q}_{q,3\zeta}\subseteq \Omega,\ \widetilde{Q}_{q,\zeta}\cap\partial E\neq\varnothing\big\} .$$
			Notice that~$ \text{\Fontauri{Q}}_{\ \zeta}\neq\varnothing$
			since~$\widetilde{Q}_{\bar q,\bar\zeta}\in\text{\Fontauri{Q}}_{\ \zeta}$.
Define~$N_\zeta:=\#(\text{\Fontauri{Q}}_{\ \zeta})$, namely the number of elements of~$\text{\Fontauri{Q}}_{\ \zeta}$.
		
		Let~$x_0\in \widetilde{Q}_{q,\zeta}\cap\partial E$, for some~$\widetilde{Q}_{q,\zeta}\in\text{\Fontauri{Q}}_{\ \zeta}$. Then, we have that~$B_\zeta(x_0)\subseteq \widetilde{Q}_{q,3\zeta}\subseteq \Omega$. Thus, since~$E$ satisfies~\eqref{eq::UDE} by assumption, we obtain that
		$$|E\cap \widetilde{Q}_{q,3\zeta}|\geq c_0\zeta^n\qquad{\mbox{and}}\qquad |E^c\cap \widetilde{Q}_{q,3\zeta}|\geq c_0\zeta^n.$$
		
		From this, using also that~$|x-y|\leq 3\sqrt{n}\zeta<\delta$, for every~$x$, $y\in \widetilde{Q}_{q,3\zeta}$, and~\eqref{eq::K_behavior}, we deduce that
		\begin{equation}\label{eq::per_lower_bound2}
			\begin{split}
				\mathcal{L}_K(E\cap \widetilde{Q}_{q,3\zeta},E^c\cap \widetilde{Q}_{q,3\zeta}) 
				& \geq \kappa_1 \int_{E\cap \widetilde{Q}_{q,3\zeta}}\int_{E^c\cap \widetilde{Q}_{q,3\zeta}} \frac{dx\,dy}{|x-y|^{n+2s_1}} \\
				& \geq \kappa_1 \int_{E\cap \widetilde{Q}_{q,3\zeta}}\int_{E^c\cap \widetilde{Q}_{q,3\zeta}} \big(3\sqrt{n}\zeta\big)^{-n-2s_1}\,dx\,dy\\
				& = \kappa_1 \big(3\sqrt{n}\zeta\big)^{-n-2s_1}|E\cap \widetilde{Q}_{q,3\zeta}|\,|E^c\cap \widetilde{Q}_{q,3\zeta}| \\
				& \geq C\zeta^{n-2s_1},
			\end{split}
		\end{equation}
		for some~$C>0$, depending only on~$n$, $s_1$, $\kappa_1$,
		and~$\kappa_2$.
				
		Set now~$\text{\Fontskrivan{R}}_\zeta$ to be the family of all cubes~$\widetilde{Q}_{q,\zeta}$ such that~$\widetilde{Q}_{q,3\zeta}\subseteq\Omega$, and notice that~$\text{\Fontauri{Q}}_{\ \zeta}\subseteq\text{\Fontskrivan{R}}_\zeta$. Thus, using~\eqref{eq::per_lower_bound2}, we see that
		\begin{equation*}
			\begin{split}
				&\mathcal{L}_K(E\cap \Omega,E^c\cap \Omega) 
				\geq \sum_{\substack{\widetilde{Q}_{q,\zeta}\in\text{\Fontskrivan{R}}_\zeta,\\ \widetilde{Q}_{q',\zeta}\in\text{\Fontskrivan{R}}_\zeta}} \mathcal{L}_K(E\cap \widetilde{Q}_{q,\zeta},E^c\cap \widetilde{Q}_{q',\zeta}) 
				\geq 3^{-2n}\sum_{\substack{\widetilde{Q}_{q,\zeta}\in\text{\Fontskrivan{R}}_\zeta,\\ \widetilde{Q}_{q',\zeta}\in\text{\Fontskrivan{R}}_\zeta}} \mathcal{L}_K(E\cap \widetilde{Q}_{q,3\zeta},E^c\cap \widetilde{Q}_{q',3\zeta}) \\
				&\qquad \geq 3^{-2n}\sum_{\widetilde{Q}_{q,\zeta}\in\text{\Fontskrivan{R}}_\zeta} \mathcal{L}_K(E\cap \widetilde{Q}_{q,3\zeta},E^c\cap \widetilde{Q}_{q,3\zeta}) 
				\geq 3^{-2n}\sum_{\widetilde{Q}_{q,\zeta}\in\text{\Fontauri{Q}}_{\ \zeta}} \mathcal{L}_K(E\cap \widetilde{Q}_{q,3\zeta},E^c\cap \widetilde{Q}_{q,3\zeta}) \\
				&\qquad \geq 3^{-2n}CN_\zeta\zeta^{n-2s_1}.
			\end{split}
		\end{equation*}
From this, we infer that
		\begin{equation*}
			\P_K(E,\Omega) \geq 3^{-2n}CN_\zeta\zeta^{n-2s_1},
		\end{equation*}
	which entails the desired estimate.
	\end{proof}
	
	Before proving Theorem~\ref{th::controlled_osc}, we observe that, since~$u_p$ is~$\Z^n$-periodic by assumption, we have that 
	$$ \osc_Q(u_p) = \osc(u_p) :=\osc_{\R^n}(u_p)= \esssup_{\R^n}  u_p - \essinf_{\R^n}  u_p.$$
	Thus, in what follows, we will simply write~$\osc(u_p)$.
	
	\begin{proof} [Proof of Theorem~\ref{th::controlled_osc}]
		Let us recall that~$E_{p,t}=\{v_p>t\}$ (see~\eqref{eq::def_lev_set}). In virtue of Proposition~\ref{prop::UDE}, we have that~$E_{p,t}$ satisfies the assumptions of Lemma~\ref{lemma::per_lower_bound} with~$\Omega:=Q$. Thus, integrating~\eqref{eq::per_lower_bound} over the set
		$$ \mathcal{T}_p := \{t\in\R \mbox{ s.t. }E_{p,t}\cap Q\neq\varnothing\},$$
		we obtain
		\begin{equation} \label{eq::osc_lower_bound}
			\int_{\mathcal{T}_p} \P_K(E_{p,t},Q)\,dt \geq C|\mathcal{T}_p| = C\osc_Q(v_p),
		\end{equation}
		where~$C$ is as in Lemma~\ref{lemma::per_lower_bound}.
		
		On the other hand, by the layer cake representation we have that
		\begin{equation}  \label{eq::osc_upper_bound1}
			\begin{split}
				&\int_{\mathcal{T}_p} \P_K(E_{p,t},Q)\,dt \leq
	 			\int_{-\infty}^{+\infty}\left(	 \iint_{Q_\sharp} \frac{1}{2}|\chi_{E_{p,t}}(x)-\chi_{E_{p,t}}(y)|K(x,y)\,dx\,dy\right)\,dt\\
	 			&\qquad\qquad	= \iint_{Q_\sharp} \left(\int_{-\infty}^{+\infty} \frac{1}{2}|\chi_{E_{p,t}}(x)-\chi_{E_{p,t}}(y)|K(x,y)\,dt\right)\,dx\,dy \\
				&\qquad\qquad = \iint_{Q_\sharp} \frac{1}{2}|v_p(x)-v_p(y)|K(x,y)\,dx\,dy \\
				&\qquad\qquad \leq 2 \iint_{Q\times\R^n} \frac{1}{2}|v_p(x)-v_p(y)|K(x,y)\,dx\,dy \\
				&\qquad\qquad = 2 \left(\cE_p(u_p)-\int_Q g(x)u_p(x)dx\right).
			\end{split}
		\end{equation}
		
Also, by the minimality of~$u_p$ and~\eqref{eq::K_integrable}, we deduce that
		\begin{equation}  \label{eq::osc_upper_bound2}
			\cE_p(u_p) \leq \cE_p(0) \leq c|p|,
		\end{equation}
		for some constant~$c>0$ independent of~$p$.
		
		Moreover, since~$\int_Q u_p(x)\,dx=0$, we have that 
		\begin{equation*}
			\essinf_{\R^n} u_p = \essinf_Q u_p \leq 0 \leq \esssup_Q u_p = \esssup_{\R^n} u_p,
		\end{equation*} 
		and hence~$|u_p|\leq\osc(u_p)$. It thereby follows that
		\begin{equation*} 
			\begin{split}
				&-\int_Q g(x)u_p(x)dx \leq \int_Q |g(x)||u_p(x)|dx \leq |Q|^{(n-s_1)/n}\norm{g}_{L^{n/s_1}(Q)}\osc(u_p) \\
				&\qquad\qquad\leq \norm{g}_{L^{n/s_1}(Q)}\left(\osc_Q(v_p)+\sqrt{n}|p|\right).
			\end{split}
		\end{equation*}
		Plugging this and~\eqref{eq::osc_upper_bound2} into~\eqref{eq::osc_upper_bound1}, we obtain
		\begin{equation*}
			\int_{\mathcal{T}_p} \P_K(E_{p,t},Q)dt \leq c|p| + 2\norm{g}_{L^{n/s_1}(Q)}\osc_Q(v_p),
		\end{equation*}
		up to renaming $c$.

		This and~\eqref{eq::osc_lower_bound} entail that 
		\begin{equation*}
			\osc_Q(v_p) \leq c|p|,
		\end{equation*}
		provided that~$\norm{g}_{L^{n/s_1}(Q)} \leq C/4$ (where~$C$ is as in Lemma~\ref{lemma::per_lower_bound}).
		
		Besides,
		\begin{equation*}
			\osc(u_p) \leq \osc_Q(v_p) + \sqrt{n}|p|\leq (c+\sqrt{n})|p|.\qedhere
		\end{equation*}
	\end{proof}
	
	\begin{proof}[Proof of Corollary~\ref{cor::K_planelike}]
For any~$t\in\mathcal{T}_p$ (i.e.~$t\in\R$ such that~$E_{p,t}\cap Q\neq\varnothing$), by Theorem~\ref{th::level_sets_minimality} we have that~$E_{p,t}$ is a class-$A$ minimizer for~$\J$.
		
		Moreover, notice that for any~$x_1\in E_{p,t}$, we have by definition that 
		\begin{equation} \label{eq::planelke_proof1}
			p\cdot x_1 > t- u_p(x_1).
		\end{equation}
		Now, since~$\partial E_{p,t}\cap Q\neq \varnothing$, there exist~$y_1\in Q$ such that
		\begin{equation*}
			u_p(y_1)+p\cdot y_1 \leq t.
		\end{equation*}
		In particular, we have that
		\begin{equation*}
			u_p(y_1)-\sqrt{n}|p| \leq t .
		\end{equation*}
		Hence, using also~\eqref{eq::planelke_proof1}, we infer that 
		\begin{equation*}
			p\cdot x_1 > t-u_p(x_1)\geq u_p(y_1)-u_p(x_1)-\sqrt{n}|p|\geq -\osc(u_p)-\sqrt{n}|p| .
		\end{equation*}
		From this and Theorem~\ref{th::controlled_osc}, it follows that 
		\begin{equation*}
			p\cdot x_1 >
			-\big(\osc(u_p)+\sqrt{n}|p|\big)\ge
			-(c+2\sqrt{n})|p|.
		\end{equation*}
			Thus, setting~$M:=c+2\sqrt{n}$, we obtain~\eqref{eq::planelike_claim_1}.
			
			Analogously, for any~$x_2\in E_{p,t}^c$, we have that 
			\begin{equation*} \label{eq::planelke_proof2}
				p\cdot x_2 \leq t- u_p(x_2).
			\end{equation*}
			Moreover, from~$\partial E_{p,t}\cap Q\neq \varnothing$, we deduce that there exist~$y_2\in Q$ such that
			\begin{equation*}
				u_p(y_2)+\sqrt{n}|p| \geq u_p(y_2)+p\cdot y_2 \geq t .
			\end{equation*}
			We infer from the last two displays that
			\begin{equation*}
				p\cdot x_2 \leq t-u_p(x_2) \leq u_p(y_2)-u_p(x_2)+\sqrt{n}|p| \leq\osc(u_p)+\sqrt{n}|p| \leq M.
			\end{equation*}
			This entails that
			\begin{equation*}
				E_{p,t}^c\subseteq \{x\in\R^n \mbox{ s.t. } x\cdot p\leq M|p|\}
			\end{equation*}
			which reads~\eqref{eq::planelike_claim_2}, concluding the proof.
	\end{proof}
	
	\begin{comment}[Proof of Corollary~\ref{cor::K_planelike}]
		For any~$t\in\mathcal{T}_p$ (i.e.~$t\in\R$ such that~$E_{p,t}\cap Q\neq\varnothing$), by Theorem~\ref{th::level_sets_minimality} we have that~$E_{p,t}$ is a class-$A$ minimizer for~$\J$.
		
		Moreover, notice that for any~$x\in\partial E_{p,t}$, we have by definition that 
		\begin{equation} %\label{eq::planelke_proof1}
			p\cdot x = t- u_p(x).
		\end{equation}
		Now, since~$\partial E_{p,t}\cap Q\neq \varnothing$, there exist~$y_1$, $y_2\in Q$ such that
		\begin{equation*}
			u_p(y_1)+p\cdot y_1 \leq t\qquad
			\mbox{and}\qquad u_p(y_2)+p\cdot y_2 \geq t.
		\end{equation*}
		In particular, we have that
		\begin{equation*}
			u_p(y_1)-\sqrt{n}|p| \leq t \leq u_p(y_2)+\sqrt{n}|p|.
		\end{equation*}
		Hence, using also~\eqref{eq::planelke_proof1}, we infer that 
		\begin{align*}
			&p\cdot x = t-u_p(x)\geq u_p(y_1)-u_p(x)-\sqrt{n}|p|\geq -\osc(u_p)-\sqrt{n}|p|\\
			\mbox{and}\qquad&p\cdot x = t-u_p(x) \leq u_p(y_2)-u_p(x)+\sqrt{n}|p| \leq\osc(u_p)+\sqrt{n}|p|.
		\end{align*}
		From this and Theorem~\ref{th::controlled_osc}, it follows that 
		\begin{equation*}
			|p\cdot x| \leq 
			\big(\osc(u_p)+\sqrt{n}|p|\big)\le
			(c+2\sqrt{n})|p|.
		\end{equation*}
		Thus,
		setting~$M:=c+2\sqrt{n}$, we obtain~\eqref{eq::planelike_claim}, as desired.
	\end{comment}
	
	% ----------------------------------------------------------------------------------------------------------------
	
	%\section{Preliminary results for Theorem~\ref{th::gamma_conv}} 
	\section{Preliminary results for the $\Gamma$-convergence} \label{sec::preliminary_results_gamma_conv}
	
	In this section, we state and prove some preliminary results which are
	useful for the proof of the~$\Gamma$-convergence stated in Theorem~\ref{th::gamma_conv}.
	
	In what follows, we will often use a stronger version of~\cite[Theorem~2.8]{MR1634336} to show that the energy contribution due to the interactions with the outside of the domain are negligible in the limit. We recall such result for the convenience of the reader.
	
	\begin{lemma}[{\cite[Theorem 2.8]{MR1634336}}] \label{lemma::alberti_bellettini}
		Let~$A$ and~$B$ be open subsets of~$\R^n$ and assume that~$A$ is bounded with Lipschitz boundary. Moreover, define~$\Sigma:=\partial A\cap \partial B$.
		
		Let also~$K$ satisfy~\eqref{eq::K_invariance}, \eqref{eq::K_integrable}, \eqref{eq::K_behavior}, and~\eqref{eq::K_lower_bound_Q} and let~$\{\epsilon_j\}_j$ be an infinitesimal sequence.
		
		Consider a sequence of functions~$\{u_j:A\cup B\to[-1,1]\}_j$ and assume that there exists an infinitesimal sequence~$\{\eta_j\}_j$ such that
		\begin{eqnarray*}
			&&\lim_{j\to+\infty}u_j(x+\eta_j h)= v(x),\quad\mbox{for a.e.~$x\in\Sigma$ and a.e~$h\in A$,}\\
			\mbox{and}\quad	&&\lim_{j\to+\infty}u_j(x+\eta_j h)= w(x),\quad\mbox{for a.e.~$x\in\Sigma$ and a.e.~$h\in B$,}
		\end{eqnarray*}
		for some functions~$v$, $w:\Sigma\to[-1,1].$
		
		Then, there exists a constant~$C_K>0$ depending only on~$K$ such that
		\begin{equation*} \label{eq::alberti_bellettini}
			\limsup_{j\to+\infty} \iint_{A\times B} |u_j(x)-u_j(y)| K_{\epsilon_j}(x,y)\,dx\,dy \leq C_K\int_{\Sigma} |v(x)-w(x)|\, d\haus{n-1}(x).
		\end{equation*}
	\end{lemma}
	
	Notice that, in our setting, a suitable rescaled sequence of planelike minimizers (for~$\J$) satisfies the assumptions of Lemma~\ref{lemma::alberti_bellettini}.
	
	For later convenience, for any~$p\in\S^{n-1}$ let us define the set
	\begin{equation} \label{eq::def_halfsapce}
		\mathcal{I}_p := \{x\cdot p <0\}.
	\end{equation}
	Then, we have the following:
	
	\begin{lemma}
		Let~$p\in\S^{n-1}$ and let~$E_p$ be a set such that
		\begin{equation}\label{RRRbvcnwiqour8392yt}
	\{x\cdot p \leq -M\} \subseteq E_p \subseteq \{x\cdot p \leq M\},
		\end{equation}
		for some~$M>0$.
		
		Let~$\{\epsilon_j\}_j$ be an infinitesimal sequence and define~$u_j:=\chi_{\epsilon_jE_p}$.
	
	Then, for every Lipschitz domain~$\Omega$, we have that,
	for a.e.~$x\in\partial\Omega$ and a.e.~$h\in \R^n$,
		\begin{equation*}
		\lim_{j\to+\infty}u_j(x+\epsilon_j^2h)= \chi_{\mathcal{I}_p}(x).
		\end{equation*}	
	\end{lemma}
	
	\begin{proof}
We observe that, thanks to~\eqref{RRRbvcnwiqour8392yt},
\begin{equation}\label{nvjekwth4ui657}
\{x\cdot p \leq -\epsilon_jM\} \subseteq \epsilon_j
E_p \subseteq \{x\cdot p \leq \epsilon_jM\}.\end{equation}
	
Now, let~$x\in\partial\Omega\setminus\partial\mathcal{I}_p$, then~$x\cdot p\neq0$ and therefore~$|x\cdot p|>\epsilon_j(M+1)$
	as soon as~$j$ is large enough.
	Moreover, if~$h\in (\partial\Omega\cap\partial\mathcal{I}_p)^c$,
	possibly taking~$j$ larger, we can assume that~$\epsilon_j|h|\leq 1$.
	
Suppose first that~$x\cdot p >\epsilon_j(M+1)$. In this case, we see that
		$$ (x+\epsilon_j^2h)\cdot p =
		x\cdot p+\epsilon_j^2 h\cdot p
		\geq x\cdot p - \epsilon_j^2|h| >
		\epsilon_j(M+1)- \epsilon_j=		
		 \epsilon_j M.$$
		It thereby follows from this and~\eqref{nvjekwth4ui657} that, for any~$j$ large enough,
		$$ u_j(x+\epsilon_j^2h)-\chi_{\mathcal{I}_p}(x)
		=\chi_{\epsilon_jE_p}(x+\epsilon_j^2h)-\chi_{\mathcal{I}_p}(x)=0.$$
		
		If instead~$x\cdot p <-\epsilon_j(M+1)$, then 
		$$ (x+\epsilon_j^2h)\cdot p =x\cdot p+\epsilon_j^2 h\cdot p\leq x\cdot p + \epsilon_j^2|h| <-\epsilon_j(M+1)+\epsilon_j
		=
		 -\epsilon_j M.$$
		Thus, recalling~\eqref{nvjekwth4ui657}, we get that,
		for any~$j$ large enough, 
		$$ u_j(x+\epsilon_j^2h)-\chi_{\mathcal{I}_p}(x) =
		\chi_{\epsilon_jE_p}(x+\epsilon_j^2h)-\chi_{\mathcal{I}_p}(x)
		= 1-1 =0.$$
		
These considerations establish the desired limit.
	\end{proof}
	
	Another crucial step to prove Theorem~\ref{th::gamma_conv} is showing that, for every set~$F\subseteq\R^n$, 
	$$ \E(F,B_R) := \iint_{B_R\times\R^n} \chi_F(x)\chi_{F^c}(y)K(x,y)\,dx\,dy + \int_{\mathcal{Q}(B_R)_1\cap F} g(x)\,dx$$
(namely~\eqref{eq::truncated_energy}
with~$\epsilon:=1$ and~$\Omega:=B_R$) is non-negative in any large ball. 
	
	\begin{proposition} \label{prop::energy_lower_bound_BR}
		Let~$K$ satisfy~\eqref{eq::K_invariance}, \eqref{eq::K_integrable}, \eqref{eq::K_behavior}, and~\eqref{eq::K_lower_bound_Q}.
		
		Then, if~$\norm{g}_{L^\infty(\R^n)}\leq \frac{\kappa_3}{2}$, we have that 
		\begin{equation*}
			\E(F,B_R)  \geq 0,
			%\iint_{B_R\times\R^n} \chi_F(x)\chi_{F^c}(y)K(x,y)\,dx\,dy + \int_{F\cap \mathcal{Q}(B_R)_1} g(x)\,dx \geq 0, %\left(\frac{\kappa_3}{2}-\norm{g}_{L^\infty}\right)\min{|F\cap \mathcal{Q}(B_R)_1|,|F^c\cap \mathcal{Q}(B_R)_1|},
		\end{equation*}
		for every set~$F\subseteq \R^n$ and for every~$R>1$.
	\end{proposition}
	
	This result is a consequence of the following lemma.
	
	\begin{lemma} \label{lemma::energy_lower_bound_Q}
		Let~$K$ satisfy~\eqref{eq::K_invariance}, \eqref{eq::K_integrable}, \eqref{eq::K_behavior}, and~\eqref{eq::K_lower_bound_Q}.
		
		Then, for every set~$F\subseteq\R^n$,
		\begin{equation*}
			\iint_{Q\times Q} \chi_F(x)\chi_{F^c}(y)K(x,y)\,dx\,dy + \int_{Q\cap F} g(x)\,dx \geq \left(\frac{\kappa_3}{2}-\norm{g}_{L^\infty(\R^n)}\right)\min\{|F\cap Q|,|F^c\cap Q|\} .
		\end{equation*}
	\end{lemma}
	
	\begin{proof}
		Notice that
		\begin{equation}\label{eq::double_vol_lower_bound_Q} 
			\begin{split}
				|F\cap Q|\,|F^c\cap Q| &= \min\{|F\cap Q|,|F^c\cap Q|\} \Big(1-\min\{|F\cap Q|,|F^c\cap Q|\}\Big) \\
				& \geq \frac{1}{2}\min\{|F\cap Q|,|F^c\cap Q|\}.
			\end{split}
		\end{equation}
		
		Moreover, recalling~\eqref{eq::g_zero_avg}, we have that
		\begin{equation*}
			\int_{Q\cap F}g(x)\,dx = - \int_{ Q\cap F^c}g(x)\,dx \geq -\norm{g}_{L^\infty(\R^n)} \min\{|F\cap Q|,|F^c\cap Q|\}.
		\end{equation*}
Thus, from this, \eqref{eq::K_lower_bound_Q}, and~\eqref{eq::double_vol_lower_bound_Q}, we obtain that
		\begin{equation*}
			\begin{split}
				&\iint_{Q\times Q} \chi_F(x)\chi_{F^c}(y)K(x,y)\,dx\,dy + \int_{Q\cap F} g(x)\,dx \\
				&\qquad\qquad\geq \kappa_3|F\cap Q|\,|F^c\cap Q|-\norm{g}_{L^\infty(\R^n)} \min\{|F\cap Q|,|F^c\cap Q|\}\\
				&\qquad\qquad \geq \left(\frac{\kappa_3}{2}-\norm{g}_{L^\infty(\R^n)}\right)\min\{|F\cap Q|,|F^c\cap Q|\},
			\end{split}
		\end{equation*}
		showing the result.
	\end{proof}
	
	\begin{proof}[Proof of Proposition~\ref{prop::energy_lower_bound_BR}]
		Notice that
		\begin{equation*}
			\begin{split}
				\E(F,B_R)&=\iint_{B_R\times\R^n} \chi_F(x)\chi_{F^c}(y)K(x,y)\,dx\,dy + \int_{ \mathcal{Q}(B_R)_1\cap F} g(x)\,dx \\
				&\geq \iint_{\mathcal{Q}(B_R)_1\times\R^n} \chi_F(x)\chi_{F^c}(y)K(x,y)\,dx\,dy + \int_{ \mathcal{Q}(B_R)_1\cap F} g(x)\,dx \\
				& \geq \sum_{\substack{ j\in\Z^n\\ j+Q\subseteq B_R}} \left(\iint_{(F\cap (j+Q))\times(F^c\cap (j+Q))} K(x,y)\,dx\,dy+ \int_{ F\cap(j+Q)} g(x)\,dx\right).
			\end{split}
		\end{equation*} 
In light if this, the claim in Proposition~\ref{prop::energy_lower_bound_BR} will follow if we show that
		\begin{equation} \label{eq::energy_lower_bound_cube}
			\iint_{(F\cap (j+Q))\times(F^c\cap (j+Q))} K(x,y)dxdy+ \int_{F\cap (j+Q)} g(x)dx \geq 0,
		\end{equation}
		for every~$j\in\Z^n$ such that~$j+Q\subseteq B_R$.
		
		Now, using Lemma~\ref{lemma::energy_lower_bound_Q} in combination with \eqref{eq::K_invariance} and the periodicity~$g$, we deduce that
		\begin{equation*}
			\begin{split}
				&\iint_{(F\cap (j+Q))\times(F^c\cap (j+Q))} K(x,y)\,dx\,dy+ \int_{F\cap (j+Q)} g(x)\,dx \\
				&\qquad\qquad= \iint_{((F-j)\cap Q)\times((F^c-j)\cap Q)} K(x,y)\,dx\,dy+ \int_{(F-j)\cap Q} g(x)\,dx\\
				&\qquad\qquad\geq \left(\frac{\kappa_3}{2}-\norm{g}_{L^\infty(\R^n)}\right)\min\{|(F-j)\cap Q|,|(F^c-j)\cap Q|\} \\&\qquad\qquad \geq 0,
			\end{split}
		\end{equation*}
		whenever~$\norm{g}_{L^\infty(\R^n)}\leq \frac{\kappa_3}{2}$.
		This shows~\eqref{eq::energy_lower_bound_cube}, as desired.
	\end{proof}
	
	% ------------------------------------------------------------------------------------------------------------
	
	\section{Well-definedness of the stable norm $\phi$} \label{sec::phi_well_defined}
	
	In this section, we prove that the stable norm~$\phi$
	given in~\eqref{eq::anisotropy}
	 is well-defined. The results presented below are the analogues in our setting of Lemma A.1 and Corollary A.2 in~\cite{chambolle_thouroude}.

	\begin{proposition} \label{prop::existence_phi}
		Let~$K$ satisfy~\eqref{eq::K_invariance}, \eqref{eq::K_integrable}, \eqref{eq::K_behavior}, and~\eqref{eq::K_lower_bound_Q}.
		Let~$g\in L^\infty(Q)$ be a~$\Z^n$-periodic function such that~$\int_Qg =0$.
		
		Assume that for any~$p\in \S^{n-1}$
		there exists a planelike minimizer~$E_p$ for~$\J$ in direction~$p$ such that
		\begin{equation} \label{eq::flatness}
			\partial E_p \subseteq \{|x\cdot p|\leq M\},
		\end{equation}
		for some~$M>0$ independent of~$p$.
		
		Then, there exists a function~$\phi: \S^{n-1}\to\R$ such that, for any infinitesimal sequence~$\{\epsilon_j\}_j$
		and any planelike minimizer~$E'_p$ for~$\J$ in direction~$p$ satisfying \eqref{eq::flatness}, it holds
		\begin{equation} \label{eq::existence_phi}
			\phi(p) = \lim_{j\to+\infty} \F_{\epsilon_j}(\epsilon_j E'_p,Q^p) = \lim_{j\to+\infty} \epsilon_j^{n-1}\F\left(
E'_p,\frac1{\epsilon_j}Q^p\right).
		\end{equation}
	\end{proposition}
	\begin{comment} %\label{prop::existence_phi}
		Let~$K$ satisfy~\eqref{eq::K_invariance}, \eqref{eq::K_integrable}, \eqref{eq::K_behavior}, and~\eqref{eq::K_lower_bound_Q}.
		Let~$g\in L^\infty(Q)$ be a~$\Z^n$-periodic function such that~$\int_Qg =0$.
		
		Assume also that for any~$p\in \S^{n-1}$
		there exists a planelike minimizer~$E_p$ for~$\J$ in direction~$p$ such that
		\begin{equation} %\label{eq::flatness}
			\partial E_p \subseteq \{|x\cdot p|\leq M\},
		\end{equation}
		for some~$M>0$ independent of~$p$.
		
		Then, there exists a function~$\phi: \S^{n-1}\to\R$ such that, for any infinitesimal sequence~$\{\epsilon_j\}_j$
		and any sequence of planelike minimizers~$E_p^j$ for~$\J$ in direction~$p$, with 
		$$ \partial E_p^j \subseteq \{|x\cdot p|\leq M\},\quad{\mbox{ for all }} j\in\N,$$
		it holds
		\begin{equation} %\label{eq::existence_phi}
			\phi(p) = \lim_{j\to+\infty} \F_{\epsilon_j}(\epsilon_j E_p^j,Q^p) = \lim_{j\to+\infty} \epsilon_j^{n-1}\F( E_p^j,Q^p/\epsilon_j).
		\end{equation}
	\end{comment}
		
	In order to prove Proposition~\ref{prop::existence_phi}, we need the auxiliary results in Lemmata~\ref{lemma::vanishing_energy}
	and~\ref{lemma::finite_limit} here below.
	%, which shows that the energy~$\F_\epsilon$ of a a set~$\epsilon E$, where $E$ is a class-$A$ planelike minimizer, is uniformly bounded in~$\epsilon$.
	
	\begin{lemma} \label{lemma::vanishing_energy}
		Let~$E\subseteq \R^n$ satisfy~\eqref{eq::flatness}. Let also~$\alpha\in(0,2s_2-1)$, where~$s_2$ is as in~\eqref{eq::K_behavior}. 
		
		Then,
		\begin{equation} \label{eq::vanishing_energy}
			\lim_{R\to+\infty} \frac{\F(E,A_{R,R^{1+\alpha},R})}{R^{n-1}}=0,
		\end{equation}
		where~$A_{L_1,L_2,R}:=\{L_1 < |x\cdot p|< L_2,\ |x\cdot p^\perp| < R\}$.
	\end{lemma}
	
	\begin{proof}
		Let us call for simplicity
		\begin{align*}
		A_{R,R^{1+\alpha},R}^+:= A_{R,R^{1+\alpha},R} \cap \{x\cdot p >0\}.
		\end{align*}
		Without loss of generality, we suppose that
		\begin{equation}\label{cbjekwfgiu876543}\{x\cdot p > M\}\subseteq E.\end{equation}
		Let us also take~$R$ so large that
		$$M\leq \frac{R}2.$$ 
		
We claim that 
\begin{equation}\label{bcnxwiqyr8tyuigfewi}
E\cap A_{R,R^{1+\alpha},R}=A_{R,R^{1+\alpha},R}^+.\end{equation}
To check this, we first show that
\begin{equation}\label{bewjkrt342345678}
%|x\cdot p|>R\qquad\mbox{and}\qquad x\cdot p>0.
x\cdot p>R\qquad{\mbox{for every~$x\in E\cap A_{R,R^{1+\alpha},R}$.}}
		\end{equation}
		Indeed, if~$x\in
	A_{R,R^{1+\alpha},R}$, we have that~$|x\cdot p|>R$, namely
	either~$x\cdot p<-R$ or~$x\cdot p> R$.

Let now~$x\in E\cap A_{R,R^{1+\alpha},R}$ and
suppose that~$x\cdot p<-R$.
Then thanks to~\eqref{eq::flatness} and~\eqref{cbjekwfgiu876543} we have that
$$ -R>x\cdot p>-M\ge-\frac{R}2,$$ which gives the desired contradiction
and establishes~\eqref{bewjkrt342345678}.
		
In light of~\eqref{bewjkrt342345678}, we see that
\begin{equation}\label{bcnxwiqyr8tyuigfewi2}
E\cap  A_{R,R^{1+\alpha},R} \subseteq
A_{R,R^{1+\alpha},R}^+.
\end{equation}

Furthermore, for every~$x\in A_{R,R^{1+\alpha},R}^+ $, we have that~$x\cdot p > R >M$, 
		which entails that~$x\in E$, thanks to~\eqref{cbjekwfgiu876543}.
This proves that
$$  A_{R,R^{1+\alpha},R}^+ \subseteq
E\cap A_{R,R^{1+\alpha},R}.$$
The claim in~\eqref{bcnxwiqyr8tyuigfewi}
then follows from this and~\eqref{bcnxwiqyr8tyuigfewi2}.

As a consequence of~\eqref{bcnxwiqyr8tyuigfewi} we also have that
$$ E\cap \mathcal{Q}(A_{R,R^{1+\alpha},R})_1
		=\mathcal{Q}\big(A_{R,R^{1+\alpha},R}^+\big)_1.$$
Hence, recalling also~\eqref{eq::g_zero_avg}, we infer that
\begin{equation} \label{eq::g_null_contribution}
\int_{E\cap \mathcal{Q}(A_{R,R^{1+\alpha},R})_1} g(x)\,dx=\int_{\mathcal{Q}\big(A_{R,R^{1+\alpha},R}^+\big)_1} g(x)\,dx=0.
\end{equation}
		
		Additionally, using~\eqref{cbjekwfgiu876543} and~\eqref{bcnxwiqyr8tyuigfewi}, and recalling~\eqref{eq::K_behavior}, we obtain
		\begin{equation}\label{mnbvc98765oiuytiuytgr}
			\begin{split}
				&\iint_{A_{R,R^{1+\alpha},R}\times \R^n} \chi_E(x)\chi_{E^c}(y)K(x,y)\,dx\,dy 
		\leq \kappa_2	\iint_{(E\cap A_{R,R^{1+\alpha},R})\times E^c}\frac{dx\,dy}{|x-y|^{n+2s_2}}\\
					&\quad
				\leq \kappa_2	\iint_{A_{R,R^{1+\alpha},R}^+\times \{y\cdot p \leq M\}}\frac{dx\,dy}{|x-y|^{n+2s_2}}\\
				&\quad = \kappa_2	\iint_{A_{R,R^{1+\alpha},R}^+\times (\{y\cdot p \leq M\}\cap B_{R})}\frac{dx\,dy}{|x-y|^{n+2s_2}} +\kappa_2 \iint_{A_{R,R^{1+\alpha},R}^+\times (\{y\cdot p \leq M\}\cap B_{R}^c)}\frac{dx\,dy}{|x-y|^{n+2s_2}}\\
				&\quad =: \mathscr{I}_1+\mathscr{I}_2,
			\end{split}
		\end{equation}
		and we aim to estimate~$\mathscr{I}_1$ and~$\mathscr{I}_2$.
		
		First, notice that for every~$x\in A_{R,R^{1+\alpha},R}^+$ and~$y\in\{y\cdot p \leq M\}$, we have that~$x\cdot p>R$ and
		\begin{equation*}
			|x-y| \geq (x-y)\cdot p=x\cdot p-y\cdot p
			% \mbox{dist}(\{x\cdot p>R\},\{y\cdot p \leq M\}) 
			> R-M\geq\frac{ R}2.
		\end{equation*}
		It thereby follows that
		\begin{equation}\label{mnbv87jtrsh}
		A_{R,R^{1+\alpha},R}^+\times\{y\cdot p \leq M\}\subseteq
		\left\{ (x,y)\in\R^n\times\R^n\;: \; |x-y|>\frac{R}2\right\}.
		\end{equation}
		
Using this, we see that
		\begin{equation} \label{eq::estimate_cylinder_short_range}
			\begin{split}
				\mathscr{I}_1 &\leq \kappa_2 \iint_{\substack{A_{R,R^{1+\alpha},R}^+\times (\{y\cdot p \leq M\}\cap B_R) \\ \{|x-y|>R/2\}}}\frac{dx\,dy}{|x-y|^{n+2s_2}} \\&\leq \frac{2^{n+2s_2}\kappa_2|A_{R,R^{1+\alpha},R}|\, |\{y\cdot p \leq M\}\cap B_R|}{R^{n+2s_2}}  \\
				&\leq \frac{c(R^{1+\alpha}-R)R^{n-1}\,R^{n}}{R^{n+2s_2}}\\&=cR^{n+\alpha-2s_2},
			\end{split}
		\end{equation}
		for some positive constant~$c=c(n,s_2,\kappa_2,M)$, possibly changing at every step of the calculation. 
		
		Besides, exploiting again~\eqref{mnbv87jtrsh} and changing variable~$h:=y-x$,
		\begin{equation*}
			\begin{split}
				&\mathscr{I}_2 \le \kappa_2 \iint_{\substack{A_{R,R^{1+\alpha},R}^+\times (\{y\cdot p \leq M\}\cap B_{R}^c) \\ \{|x-y|>R/2\}}}\frac{dx\,dy}{|x-y|^{n+2s_2}}  \leq \kappa_2 |A_{R,R^{1+\alpha},R}| \int_{B_{R/2}^c}\frac{dh}{|h|^{n+2s_2}}\\
				&\qquad \leq c(R^{1+\alpha}-R)R^{n-1}
				\int_{R/2}^{+\infty} \frac{d\rho}{\rho^{1+2s_2}}  =  cR^{n+\alpha-2s_2}.
			\end{split}
		\end{equation*}
Gathering this, \eqref{mnbvc98765oiuytiuytgr}, and~\eqref{eq::estimate_cylinder_short_range} together, we deduce that	$$\iint_{A_{R,R^{1+\alpha},R}\times \R^n} \chi_E(x)\chi_{E^c}(y)K(x,y)\,dx\,dy \le cR^{n+\alpha-2s_2}.$$

From this and~\eqref{eq::g_null_contribution}, we conclude that
		\begin{equation*}
			\left|\frac{\F(E,A_{R,R^{1+\alpha},R})}{R^{n-1}}\right|
			= \frac{1}{R^{n-1}}\left|
			\P_K\big(E,A_{R,R^{1+\alpha},R}\big) + \int_{E\cap\mathcal{Q}(A_{R,R^{1+\alpha},R})_1} g(x)\,dx\right|
			 \leq cR^{1+\alpha-2s_2} ,
		\end{equation*}
which entails~\eqref{eq::vanishing_energy}.
	\end{proof}
	
	\begin{lemma} \label{lemma::finite_limit}
		Let~$E$ be a planelike minimizer for~$\J$ satisfying~\eqref{eq::flatness}.
		
		Then, there exist~$C_1$, $C_2\in\R$ independent of~$\epsilon$ such that
		\begin{equation} \label{eq::finite_limit}
			C_1\epsilon^{1-n} \leq \F\left(E,\frac{1}{\epsilon} Q^p\right) \leq C_2\epsilon^{1-n}.
		\end{equation}
		
		Moreover, let~$\eta\in(0,1)$ such that~$\eta^{1+\alpha}<\epsilon<\eta$, for some parameter~$\alpha\in(0,2s_2-1)$. 
		
		Then, there exists a constant~$c$ independent of~$\epsilon$ and~$\eta$ such that
		\begin{equation} \label{eq::energy_tail}
			\F\left(E,\frac{1}{\epsilon}Q^p\setminus \frac{1}{\eta}Q^p\right)  \leq c\big(\epsilon^{1-n}-\eta^{1-n}+\eta^{2s_2-1-\alpha}\big),
		\end{equation}
		whenever~$\eta$ is small enough (and hence~$\epsilon$ is small enough).
	\end{lemma}
	
	\begin{proof}
	Since~$E$ satisfies~\eqref{eq::flatness}, we suppose, without loss
	of generality, that
	\begin{equation}\label{vbcnx4832345rwfyahduewtywoihl}
	\big\{x\cdot p > M\big\}\subseteq E.
	\end{equation}
	
We will use the notation~$\widetilde{Q}:=k+Q$, for some~$k\in\Z^n$. Notice that~$\mathcal{Q}(\widetilde{Q})_1=\widetilde{Q}$.
Thus, by the
		class-$A$ minimality of~$E$ (for~$\J$),
		%with the competitor~$E\setminus \widetilde{Q}$, 
		we find
		\begin{equation*}
			\F(E,\widetilde{Q}) = \J(E,\widetilde{Q}) \leq \J(E\setminus \widetilde{Q},\widetilde{Q}) =P_K(E\setminus \widetilde{Q},\widetilde{Q}) 
+\int_{(E\setminus \widetilde{Q})\cap\widetilde{Q}} g(x)\,dx
			\leq P_K(\widetilde{Q}) = P_K(Q). 
		\end{equation*}
		Hence,
		\begin{equation} \label{eq::PK_estimate_kQ}
			\P_K(E,\widetilde{Q}) = \F(E,\widetilde{Q})-\int_{E\cap \widetilde{Q}}g(x)\, dx \leq \P_K(Q)+\norm{g}_{L^1(Q)}.
		\end{equation}
		
		Now, 
		let~$S_M:= \{|x\cdot p|\leq M\}$ and
		take a (finite) covering of the strip~$S_M\cap Q^p/\epsilon$, namely
		$$ \mathcal{B} := \big\{k+Q \mbox{ s.t. } k\in\Z^n,\ (k+Q)\cap (S_M\cap Q^p/\epsilon)\neq\varnothing\big\}.$$
Observe that~$\sharp(\mathcal{B}) \leq 2(M+\sqrt{n})\epsilon^{1-n}$.
		
		Then, using~\eqref{eq::PK_estimate_kQ} together with the monotonicity and the subadditivity of~$\P_K$ with respect to the domain, we infer that
		\begin{equation} 
			\label{eq::P_K_upper_bound_close}
			\epsilon^{n-1}\P_K\left(E,S_M\cap \frac1{\epsilon}Q^p\right) \leq \epsilon^{n-1}\sum_{\widetilde{Q}\in\mathcal{B}} \P_K(E,\widetilde{Q}) \leq 2(M+\sqrt{n})\Big(\P_K(Q)+\norm{g}_{L^1(Q)}\Big)\end{equation}
and
\begin{equation}\label{eq::g_bound}\epsilon^{n-1}\left|\int_{E\cap\mathcal{Q}(Q^p/\epsilon)_1} g(x)\,dx\right| = \epsilon^{n-1}\sum_{\widetilde{Q}\in\mathcal{B}}\left|\int_{E\cap\mathcal{Q}(Q^p/\epsilon)_1\cap\widetilde{Q}} g(x)\,dx\right| \leq 2(M+\sqrt{n})\norm{g}_{L^1(Q)}.
		\end{equation}
		In particular, \eqref{eq::g_bound} yields 
		\begin{equation} \label{eq::lower_bound_limit}
			\epsilon^{n-1}\F\left(E,\frac1\epsilon Q^p\right)\geq \epsilon^{n-1}\int_{E\cap\mathcal{Q}(Q^p/\epsilon)_1} g(x)\,dx \geq -  2(M+\sqrt{n})\norm{g}_{L^1(Q)}.
		\end{equation}
		
Moreover, we claim that 
		\begin{equation} \label{eq::upper_bound_limit}
			\epsilon^{n-1}\P_K\left(E,\frac1{\epsilon}Q^p \setminus S_M\right) \leq c.
		\end{equation}
		% To achieve this, let~$m\in\N$ be such that~$2^{-m-1} \leq \epsilon\leq 2^{-m}$. We define, for every~$j\in\{0,\dots,m-1\}$,
				To achieve this, let~$m\in\N$ be such that~$2^{-m} \leq \epsilon\leq 2^{-m+1}$. We define, for every~$j\in\{0,\dots,m-1\}$,
		$$ C_j:=\frac1{\epsilon}Q^p  \cap \Big\{ M2^{j} \leq |x\cdot p | \leq M2^{j+1}\Big\}\subseteq\R^n.$$
		% $$ C_j:=\left(Q^p/\epsilon\right)^{n-1} \times \{ M2^{j} \leq |x\cdot p | \leq M2^{j+1}\},\quad \mbox{for every }j\in\{0,\dots,m-1\}.$$
We recall~\eqref{vbcnx4832345rwfyahduewtywoihl} and we also set
		\begin{align*}
			&S_{M}^+:=\big\{x\cdot p > M\big\}\subseteq E\\
\mbox{and}\qquad &S_{M}^-:=\big\{x\cdot p <- M\big\}\subseteq E^c.
		\end{align*}	
		Then, we have
		$$\frac1{\epsilon} Q^p \setminus S_{M} %\subseteq 2^mQ^p\setminus S_{M}
		\subseteq \bigcup_{j=0}^{m-1} C_j.$$
	
		Moreover, notice that every~$C_j$ can be covered with at most~$M2^{m(n-1)+j+1}$ cubes of side~$1$ centered at
		points in~$\Z^n$ (let us call~$\mathcal{B}_j$ such covering). 
		
		Let ~$\widetilde{Q}:=k+Q\in\mathcal{B}_j$, for some~$k\in\Z^n$ and some~$j\in\{0,\dots,m-1\}$. Without loss of generality, we can suppose that~$\widetilde{Q}\cap S_{M}^+\neq\varnothing$, and hence~$\widetilde{Q}\cap S_{M}^-=\varnothing$. 
Notice also that~$\widetilde{Q}^c\cap S_{M}^-\subseteq
{B}_{(2^j+1)M}^c(k)$. Thus,
		\begin{equation} \label{eq::estimate_annulus}
			\begin{split}
				&\iint_{\widetilde{Q}_\sharp} \chi_{S_{M}^+}(x)\chi_{S_{M}^-}(y) K(x,y)\,dx\,dy 
				=\iint_{\widetilde{Q}\times (\widetilde{Q}^c\cap S_{M}^-)} K(x,y)\,dx\,dy \\
				&\qquad\qquad\leq \kappa_2\iint_{\widetilde{Q}\times (\widetilde{Q}^c\cap S_{M}^-)} \frac{dx\,dy}{|x-y|^{n+2s_2}}
				\leq \kappa_2\iint_{\widetilde{Q}\times {B}_{(2^j+1)M}^c(k)} \frac{dx\,dy}{|x-y|^{n+2s_2}} \\
				&\qquad\qquad = \kappa_2\iint_{Q\times {B}_{(2^j+1)M}^c} \frac{dx\,dy}{|x-y|^{n+2s_2}}= \kappa_2\iint_{Q\times {B}_{(2^j+1)M}^c(x)} \frac{dx\,dh}{|h|^{n+2s_2}}\\
				&\qquad\qquad \leq \kappa_2 |Q|\int_{{B}_{2^jM}^c}\frac{dh}{|h|^{n+2s_2}} 
				\leq c_1\,2^{-2s_2j},
			\end{split}
		\end{equation}
		for some positive~$c_1=c_1(n,s_2,\kappa_2,M)$ independent of~$m$. %, possibly changing from line to line.

		Now, observe that, by~\eqref{vbcnx4832345rwfyahduewtywoihl} and the definitions of~$S_M^+$ and~$S_M^-$, we have that
		\begin{equation*}
			E\cap S_M^c = S_M^+ \qquad\mbox{and}\qquad E^c\cap S_M^c = S_M^-.
		\end{equation*}
		As a consequence,
		\begin{equation*}
			\begin{split}
				&\P_K\left(E,\frac1{\epsilon}Q^p\setminus S_M\right) = \iint_{(Q^p/\epsilon\setminus S_M)_\sharp} \chi_{E}(x)\chi_{E^c}(y) K(x,y)\,dx\,dy \\
				=& \iint_{(Q^p/\epsilon\setminus S_M)\times (Q^p/\epsilon\setminus S_M)} \chi_{E}(x)\chi_{E^c}(y) K(x,y)\,dx\,dy \\
				& +  \iint_{((Q^p/\epsilon)^c\setminus S_M)\times (Q^p/\epsilon\setminus S_M)} \chi_{E}(x)\chi_{E^c}(y) K(x,y)\,dx\,dy \\
				& +  \iint_{(Q^p/\epsilon\setminus S_M)\times ((Q^p/\epsilon)^c\setminus S_M)} \chi_{E}(x)\chi_{E^c}(y) K(x,y)\,dx\,dy\\
				& + \iint_{S_M\times(Q^p/\epsilon \setminus S_M)} \chi_{E}(x)\chi_{E^c}(y) K(x,y)\,dx\,dy+ \iint_{(Q^p/\epsilon\setminus S_M)\times S_M} \chi_{E}(x)\chi_{E^c}(y) K(x,y)\,dx\,dy\\
				=& \iint_{(Q^p/\epsilon\setminus S_M)_\sharp} \chi_{S_{M}^+}(x)\chi_{S_{M}^-}(y) K(x,y)\,dx\,dy\\
				& + \iint_{S_M\times(Q^p/\epsilon \setminus S_M)} \chi_{E}(x)\chi_{S_{M}^-}(y) K(x,y)\,dx\,dy+ \iint_{(Q^p/\epsilon\setminus S_M)\times S_M} \chi_{S_{M}^+}(x)\chi_{E^c}(y) K(x,y)\,dx\,dy\\
				\leq&  \iint_{(Q^p/\epsilon\setminus S_M)_\sharp} \chi_{S_{M}^+}(x)\chi_{S_{M}^-}(y) K(x,y)\,dx\,dy + \iint_{S_M\times(Q^p/\epsilon\setminus S_M)} K(x,y)\,dx\,dy.
			\end{split}
		\end{equation*}
		Moreover, by Proposition~\ref{prop::P_K_upper_bound_BR}, there exists a positive constant~$c_2$, independent of $\epsilon$, such that %depending only on $n$, $M$, and $K$ such that
		\begin{equation} \label{eq::estimate_remainder_S_M}
			\iint_{S_M\times(Q^p/\epsilon\setminus S_M)} K(x,y)\,dx\,dy
			\le 
			\P_K\left(\frac1{\epsilon}Q^p \cap S_M^+\right)+\P_K\left(\frac1{\epsilon}Q^p \cap S_M^-\right)\leq c_2\epsilon^{1-n}.
		\end{equation}
		
		Thus, the estimate in~\eqref{eq::estimate_annulus}, in combination with~\eqref{eq::estimate_remainder_S_M} and the subadditivity of~$P_K$ with respect to the domain, entails that
		\begin{equation*}
			\begin{split}
				&\epsilon^{n-1}\P_K\left(E,\frac1{\epsilon}Q^p\setminus S_M\right) \leq \epsilon^{n-1} \left(\iint_{(Q^p/\epsilon\setminus S_M)_\sharp} \chi_{S_{M}^+}(x)\chi_{S_{M}^-}(y) K(x,y)\,dx\,dy +c_2\epsilon^{1-n}\right)\\
				&\qquad \leq 2^{-(m-1)(n-1)} \iint_{(Q^p/\epsilon\setminus S_M)_\sharp} \chi_{S_{M}^+}(x)\chi_{S_{M}^-}(y) K(x,y)\,dx\,dy +c_2\\
				&\qquad\leq  2^{-(m-1)(n-1)}\sum_{j=0}^{m} \sum_{\widetilde{Q}\in\mathcal{B}_j} \iint_{\widetilde{Q}_\sharp} \chi_{S_{M}^+}(x)\chi_{S_{M}^-}(y) K(x,y)\,dx\,dy +c_2\\
				&\qquad \leq 2^{-(m-1)(n-1)}\sum_{j=0}^{m} c_12^{m(n-1)+j}2^{-2s_2j}\ +c_2 \leq c_1\sum_{j=0}^{+\infty} 2^{(1-2s_2)j}\ +c_2 <+\infty,
			\end{split}
		\end{equation*}
		showing~\eqref{eq::upper_bound_limit}.
		
		Therefore, gathering~\eqref{eq::P_K_upper_bound_close}, \eqref{eq::lower_bound_limit} and~\eqref{eq::upper_bound_limit}, we obtain
		\begin{equation*}
			-2(M+\sqrt{n})\norm{g}_{L^1(Q)} \leq \epsilon^{n-1}\F\left(E,\frac1{\epsilon}Q^p\right) \leq 2(M+\sqrt{n})\Big(\P_K(Q)+2\norm{g}_{L^1(Q)}\Big) + c,
		\end{equation*}
		showing~\eqref{eq::finite_limit}.
		
		Now, to prove~\eqref{eq::energy_tail}, let us consider~$\eta^{1+\alpha}<\epsilon<\eta$, for some~$\alpha\in(0,2s_2-1)$. We assume again that~$2^{-m}\leq\epsilon\leq 2^{-m+1}$, for some~$m\in\N$, and we
		decompose~$\frac{1}{\epsilon}Q^p\setminus \frac{1}{\eta}Q^p$ as
		\begin{equation*}
			\frac{1}{\epsilon}Q^p\setminus \frac{1}{\eta}Q^p = \left(\frac{1}{\epsilon}Q^p\setminus \frac{1}{\eta}Q^p\right)\cap S_M \cup \bigcup_{j=0}^{m-1}\widetilde{C}_j,
		\end{equation*}
		where, for all~$ j\in\{0,\dots,m-1\}$,
		$$\widetilde{C}_j:=\left(\frac{1}{\epsilon}Q^p\setminus \frac{1}{\eta}Q^p\right)\cap \Big\{M2^{j}\leq |x\cdot p| \leq M2^{j+1}\Big\}.$$ 
		Notice that~$\left(\frac{1}{\epsilon}Q^p\setminus \frac{1}{\eta}Q^p\right)\cap S_M$ can be covered with at most~$2M(\epsilon^{1-n}-\eta^{1-n})$ cubes of side~$1$. 
		Similarly, each cylinder~$\widetilde{C}_j$ can be covered with at most~$2^j(\epsilon^{1-n}-\eta^{1-n})$ translated unit cubes.
		
		Therefore, arguing as for~\eqref{eq::finite_limit} and recalling Lemma~\ref{lemma::vanishing_energy}, we deduce that, for every~$\eta$ small enough,
		\begin{equation*}
			\F\left(E,\frac{1}{\epsilon}Q^p\setminus\frac{1}{\eta}Q^p\right) \leq c\big(\epsilon^{1-n}-\eta^{1-n}+\eta^{2s_2-1-\alpha}\big),
		\end{equation*}
		for some positive constant~$c$ independent of~$\eta$ and~$\epsilon$, concluding the proof.
	\end{proof}
	
	\begin{proof}[Proof of Proposition~\ref{prop::existence_phi}]
		Let~$p\in\S^{n-1}$. In order to prove that the limit function~$\phi$ in~\eqref{eq::existence_phi} is well-defined, we show that if~$E_p$ and~$E'_p$ are two planelike minimizers for~$\J$ with respect to some~$p$ such that
		$$ \partial E_p\cap \partial E'_p \subseteq \{|x\cdot p| \leq M\},$$
		and~$\epsilon$, $\epsilon'>0$ are infinitesimal sequences
		such that the limits
$$
\ell:=\lim_{\epsilon\to0}\epsilon^{n-1}\F\left(E_p,\frac1{\epsilon}Q^p\right)\qquad\mbox{and}\qquad
\ell':=\lim_{\epsilon'\to0}(\epsilon')^{n-1}\F\left(E'_p,\frac1{\epsilon'}Q^p\right)$$	
exist, then it must be~$\ell=\ell'$. To this purpose, up to taking subsequences, let us assume that~$\epsilon'\ll\epsilon$ (e.g.~$\epsilon'\sim\epsilon^{1+\alpha}$, for some~$\alpha\in(0,2s_2-1)$). 
		
		Following the footsteps of~\cite{chambolle_thouroude}, 
		we recall the notation for~$\mathcal{I}_p$
		in~\eqref{eq::def_halfsapce} and
we consider a finite covering of~$\mathcal{I}_p\cap \mathcal{Q}(Q^p/\epsilon')_1$ made of~$N:=\left\lfloor \left(\frac{\epsilon}{\epsilon'(1+2\epsilon\sqrt{n})}\right)^{n-1}\right\rfloor$ cubes that are translations of~$(2\sqrt{n}+1/\epsilon)Q^p$ centered at~$\mathcal{I}_p$. Notice that each of this cube contains at least one translated cube~$Q^{(j)}:=j+Q^p/\epsilon$, with~$j\in\Z^n$. Let also 
		$$ E_j := j+\left(E_p\cap \frac1\epsilon Q^p\right) \subseteq Q^{(j)}.$$
Also, we define
		\begin{equation} \label{eq::competitor_phi_limit}
	R:= \bigcup_{j=1}^N Q^{(j)} \subseteq \mathcal{Q}(Q^p/\epsilon')_1\qquad
				\mbox{and}\qquad
		F':=\left(\bigcup_{j=1}^N E_j \right)\cup  (E'_p\setminus R).
		\end{equation}
		In particular, we have~$F'\setminus \mathcal{Q}(Q^p/\epsilon')_1 = E'_p\setminus\mathcal{Q}(Q^p/\epsilon')_1$ (see Figure \ref{fig::phi_well_defined}).
		
		\begin{figure}
			\centering
			\includegraphics[width=.85\linewidth]{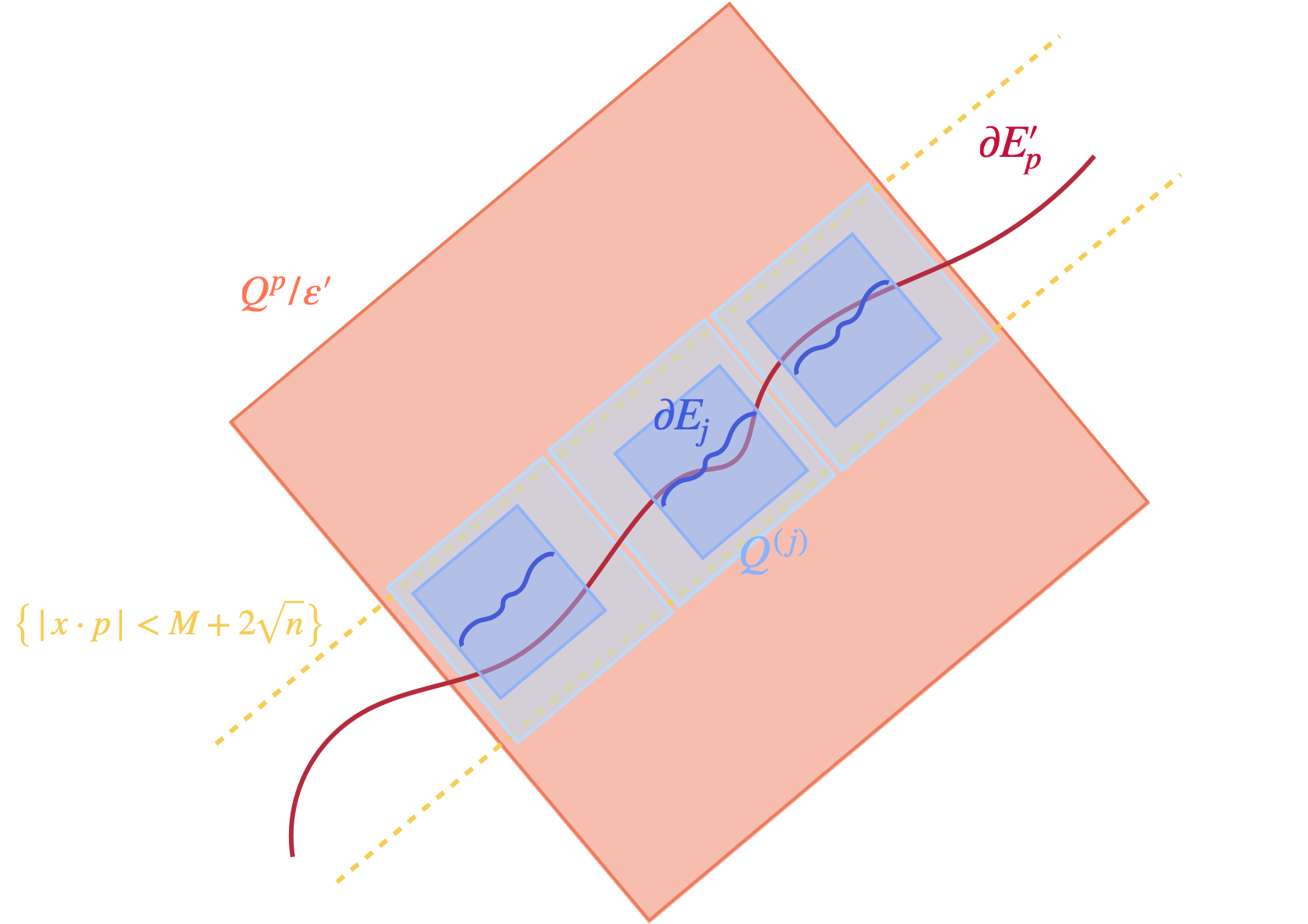}
			\caption{The construction of the set $F'$ as in \eqref{eq::competitor_phi_limit}.}\label{fig::phi_well_defined}
		\end{figure}
		
		Thus, by Remark~\ref{rem::class_A_F}, the subadditivity of~$\P_K$ with respect to the domain, and~\eqref{eq::K_invariance}, we infer that
		\begin{equation} \label{eq::estimate_phiexistence_1}
			\begin{split}
				&\F\left(E'_p,\frac1{\epsilon'}
				Q^p\right)\leq \F\left(F',\frac1{\epsilon'}Q^p\right) =  \F\left(F',R\cup \left(\frac1{\epsilon'}Q^p\setminus R\right)\right) \\
				&\qquad \leq \sum_{j=1}^{N}\F\left(E_j,Q^{(j)}\right) + \iint_{(Q^p/\epsilon'\setminus R)_\sharp} \chi_{F'}(x)\chi_{(F')^c}(y)K(x,y)\,dx\,dy \\
				&\qquad\qquad+ \int_{F'\cap (\mathcal{Q}(Q^p/\epsilon')_1\setminus \bigcup_{j=1}^N\mathcal{Q}(Q^{(j)})_1)}g(x)\,dx\\
				&\qquad \leq N\F\left(E_p,\frac1{\epsilon}Q^p\right) + \iint_{(Q^p/\epsilon'\setminus R)_\sharp} \chi_{F'}(x)\chi_{(F')^c}(y)K(x,y)\,dx\,dy \\
				&\qquad\qquad+ \int_{F'\cap (\mathcal{Q}(Q^p/\epsilon')_1\setminus \bigcup_{j=1}^N\mathcal{Q}(Q^{(j)})_1)}g(x)\,dx.
			\end{split}
		\end{equation}
		At this stage, observe that 
		$$ \partial F' \cap \left(\frac1{\epsilon'}Q^p\setminus R\right) \subseteq \left\{|x\cdot p| \leq M+2\sqrt{n}\right\}$$
		can be covered with~$\left\lfloor(M+2\sqrt{n})\left( (\epsilon')^{1-n}-N\epsilon^{1-n}\right)\right\rfloor$ cubes~$Q$ of side~$1$ and centered at~$\Z^n$. 
		Therefore, it follows from~\eqref{eq::energy_tail} that %Lemma~\ref{lemma::vanishing_energy} that%\eqref{eq::P_K_upper_bound_close} and~\eqref{eq::upper_bound_limit} that
		\begin{equation} \label{eq::term_xi_to_0}
			\begin{split}
				&\iint_{(Q^p/\epsilon'\setminus R)_\sharp} \chi_{F'}(x)\chi_{(F')^c}(y)K(x,y)\,dx\,dy \\
				\leq\;& c(M+2\sqrt{n})\Big(
				 (\epsilon')^{1-n}-N\epsilon^{1-n}+\epsilon^{2s_2-1-\alpha}\Big)  \\
				\leq\;& \frac{c(M+2\sqrt{n})}{(\epsilon')^{n-1}}\left(1-\left(\frac{\epsilon'}{\epsilon}\right)^{n-1}\left\lfloor \left(\frac{\epsilon}{\epsilon'(1+2\epsilon\sqrt{n})}\right)^{n-1}\right\rfloor\right) 
				+c(M+2\sqrt{n})\epsilon^{2s_2-1-\alpha} \\
				=:\;&\frac{\xi(\epsilon,\epsilon')}{(\epsilon')^{n-1}},
			\end{split}
		\end{equation}
		for some function~$\xi$ such that~$\xi(\epsilon,\epsilon')\to0$, as~$\epsilon$, $\epsilon'$, and~$\epsilon'/\epsilon\sim \epsilon^\alpha\to0$. 
		
		Furthermore, since
		$$ \partial F' \cap (\mathcal{Q}(Q^P/\epsilon')_1\setminus \bigcup_{j=1}^N\mathcal{Q}(Q^{(j)})_1)) \subseteq \{|x\cdot p| \leq M+2\sqrt{n}\}$$
		can be covered with~$\left\lfloor(M+2\sqrt{n})\left((\epsilon')^{1-n}-N\lfloor\epsilon^{1-n}\rfloor\right)\right\rfloor$ cubes~$Q^p$ of side~$1$ and centered at~$\Z^n$, we also have
		\begin{equation} \label{eq::term_zeta_to_0}
			\begin{split}
				&\int_{F' \cap (\mathcal{Q}(Q^P/\epsilon')_1\setminus \bigcup_{j=1}^N\mathcal{Q}(Q^{(j)})_1))}g(x)\,dx \leq \norm{g}_{L^1(Q)}(M+2\sqrt{n})\Big((\epsilon')^{1-n}-N\lfloor\epsilon^{1-n}\rfloor\Big)\\
				&\qquad= \frac{\norm{g}_{L^1(Q)}(M+2\sqrt{n})}{(\epsilon')^{n-1}}\left(1-(\epsilon')^{n-1}\lfloor\epsilon^{1-n}\rfloor\left\lfloor \left(\frac{\epsilon}{\epsilon'(1+2\epsilon\sqrt{n})}\right)^{n-1}\right\rfloor\right)=: \frac{\zeta(\epsilon,\epsilon')}{(\epsilon')^{n-1}},
			\end{split}
		\end{equation}
		for some function~$\zeta$ such that~$\zeta(\epsilon,\epsilon')\to0$, as~$\epsilon\to0$ (and also~$\epsilon',\epsilon'/\epsilon\to0$). Indeed, 
		\begin{equation*}
			\begin{split}
				&\lim_{\substack{\epsilon'\to0 \\ \epsilon\to0}}\ (\epsilon')^{n-1}\lfloor\epsilon^{1-n}\rfloor\left\lfloor \left(\frac{\epsilon}{\epsilon'(1+2\epsilon\sqrt{n})}\right)^{n-1}\right\rfloor \\
				= \;& \lim_{\substack{\epsilon'\to0 \\ \epsilon\to0}}\ (\epsilon')^{n-1}\big(\epsilon^{1-n}-\{\epsilon^{1-n}\}\big)\left( \left(\frac{\epsilon}{\epsilon'(1+2\epsilon\sqrt{n})}\right)^{n-1}-\left\{\left(\frac{\epsilon}{\epsilon'(1+2\epsilon\sqrt{n})}\right)^{n-1}\right\}\right)\\
				=\;& 1,
			\end{split}
		\end{equation*}
		where~$\{\cdot\}:=(\cdot)-\lfloor\cdot\rfloor$ denotes the fractional part of a real number.
		
		Now, recalling the definitions of~$\ell$ and~$\ell'$, and putting together~\eqref{eq::estimate_phiexistence_1}, \eqref{eq::term_xi_to_0} and~\eqref{eq::term_zeta_to_0}, we obtain that 
		\begin{equation*}
			(\epsilon')^{n-1}\F\left(E'_p,\frac1{\epsilon'}Q^p\right) \leq (\epsilon')^{n-1}\left\lfloor\left(\frac{\epsilon}{\epsilon'(1+2\epsilon\sqrt{n})}\right)^{n-1} \right\rfloor\F\left(
			E_p,\frac1{\epsilon}Q^p\right) +\xi(\epsilon,\epsilon')+\zeta(\epsilon,\epsilon').
		\end{equation*}
		Hence, taking the limits for~$\epsilon$, $\epsilon'$, and~$\epsilon'/\epsilon\sim \epsilon^\alpha\to0$, we deduce that~$\ell'\leq\ell$.
		
		Interchanging the roles of~$\epsilon$ and~$\epsilon'$, we also get that~$\ell\leq\ell'$, concluding the proof.
\end{proof}
	
	We stress that the construction of the cubes $Q^{(j)}$ in the proof of Proposition~\ref{prop::existence_phi} remains unaltered if we replace $Q^p/\epsilon'$ with one of its translated along $\partial \mathcal{I}_p$. Indeed, notice that when we consider translations of the cube $(2\sqrt{n}+1/\epsilon)Q^p$, we do not ask any requirement on those translations but the fact that they are centered at $\partial \mathcal{I}_p$.  Therefore, we infer that $\phi(p)$ is independent of translations of the domain in directions orthogonal to $p$. 
	
	The precise statement goes as follows: 
	
	%As a consequence of Proposition~\ref{prop::existence_phi}, we infer that the cube~$Q^p/\epsilon$ in the definition of~$\phi$ can be replaced by any of its translated along $\partial \mathcal{I}_p$.
	\begin{corollary} \label{cor::phi_translated_cube}
		Let~$K$ satisfy~\eqref{eq::K_invariance}, \eqref{eq::K_integrable}, \eqref{eq::K_behavior}, and~\eqref{eq::K_lower_bound_Q}.
		Let~$g\in L^\infty(Q)$ be a~$\Z^n$-periodic function such that~$\int_Qg =0$.
		
		Let~$p\in \S^{n-1}$	and let~$E_p$ be a planelike minimizer for~$\J$ in direction~$p$ that satisfies~\eqref{eq::flatness} .
		Let also~$v\in \partial \mathcal{I}_p$, i.e. $v\cdot p = 0$.
		
		Then,
		\begin{equation*}
			\phi(p) = \lim_{\epsilon\to0} \epsilon^{n-1}\F\left(E_{p},v+\frac1\epsilon Q^p\right).
		\end{equation*}
	\end{corollary}
	
		As a consequence, we deduce the following:
	\begin{corollary} \label{cor::tail_estimate_phi}
		Let~$K$ satisfy~\eqref{eq::K_invariance}, \eqref{eq::K_integrable}, \eqref{eq::K_behavior}, and~\eqref{eq::K_lower_bound_Q}.
		Let~$g\in L^\infty(Q)$ be a~$\Z^n$-periodic function such that~$\int_Qg =0$.
		
		Let~$p\in \S^{n-1}$ and $t\in\mathcal{T}_p$, so that $E_{p,t}$ is a planelike minimizer for~$\J$ in direction~$p$ that satisfies~\eqref{eq::flatness} (according to Corollary~\ref{cor::K_planelike}).
		
Then, for any Lipschitz domain~$\Omega\subseteq\R^n$,
		\begin{equation} \label{eq::tail_estimate_phi}
			\phi(p)\haus{n-1}(\Omega\cap \partial\mathcal{I}_p) = \lim_{\epsilon\to0} \F_\epsilon(\epsilon E_{p,t}, \Omega),
		\end{equation}
		where~$\mathcal{I}_p$ is defined as in~\eqref{eq::def_halfsapce}.
	\end{corollary}

	\begin{proof}
		Let us consider sequences~$\epsilon_j$, $\rho_j$ such that
		\begin{equation}\label{bvc47tyfuhkaeuo37948t42qioh}
			\lim_{j\to+\infty} \epsilon_j = \lim_{j\to+\infty} \rho_j = \lim_{j\to+\infty} \frac{\epsilon_j}{\rho_j} = 0. 
		\end{equation}
Also, we consider the collection%~$\mathcal{B}_j$ of all disjoint cubes that are translated of~$\rho_j Q^p$ (of side~$\rho_j$), centered at points of the form~$\epsilon_j v_j$, for some~$v_j\in\partial \mathcal{I}_p$, and contained in~$\Omega$. %, namely
		$$\mathcal{B}_j := \big\{\rho_j\widetilde{Q}:=\epsilon_j v_j+\rho_jQ^p \mbox{ s.t. } v_j\in\partial\mathcal{I}_p \mbox{, and }\rho_j\widetilde{Q}\subseteq\Omega\big\}. $$
		In particular, $N_j:=\sharp(\mathcal{B}_j) = \lfloor \rho_j^{1-n}\haus{n-1}(\Omega\cap\partial\mathcal{I}_p)\rfloor$.  Let us also denote 
		$$R_j:=\bigcup_{\rho_j\widetilde{Q}\in \mathcal{B}_j}\rho_j\widetilde{Q}.$$
		
%		Since~$\epsilon_j/\rho_j\to0$ and~$\int_Q g=0$, i
It is convenient to look at the family of cubes 
		\begin{equation*}
			\mathcal{Q}_{E_p}(\Omega/\epsilon_j)_1 := \bigcup
			\big(\ell+Q\big)\end{equation*}
where the union is taken over all~$\ell\in\Z^n$ such that~$\ell+Q\subset \Omega/\epsilon_j $ and~$(\ell+Q)\cap\partial E_p\neq\varnothing$.

Thanks to~\eqref{bvc47tyfuhkaeuo37948t42qioh}, we can take~$j$
sufficiently large such that~$\rho_j/\epsilon_j>M$, and we have that
		\begin{equation} \label{eq::g_reminder}
			\begin{split}
				&\epsilon_j^{n-1}\int_{E_p\cap \mathcal{Q}(\Omega/\epsilon)_1}g = \epsilon_j^{n-1}\int_{E_p\cap \mathcal{Q}_{E_p}(\Omega/\epsilon_j)_1}g\\
				&\qquad= \sum_{\rho_j\widetilde{Q}\in \mathcal{B}_j}\left(\epsilon_j^{n-1}\int_{E_p\cap \mathcal{Q}_{E_p}(\rho_j\widetilde{Q}/\epsilon_j)_1}g\right) + \epsilon_j^{n-1}\int_{E_p\cap \left( \mathcal{Q}_{E_p}(\Omega/\epsilon_j)\setminus \bigcup_{\rho_j\widetilde{Q}\in \mathcal{B}_j}\mathcal{Q}_{E_p}(\rho_j\widetilde{Q}/\epsilon_j)_1\right)}g\\
				&\qquad = \sum_{\rho_j\widetilde{Q}\in \mathcal{B}_j}\left(\epsilon_j^{n-1}\int_{E_p\cap \mathcal{Q}(\rho_j\widetilde{Q}/\epsilon_j)_1}g\right)  +\psi(j),
			\end{split}
		\end{equation}
		where
		\begin{equation*} 
			\psi(j):=\epsilon_j^{n-1}\int_{E_p\cap \left( \mathcal{Q}_{E_p}(\Omega/\epsilon_j)_1\setminus \bigcup_{\rho_j\widetilde{Q}\in \mathcal{B}_j}\mathcal{Q}_{E_p}(\rho_j\widetilde{Q}/\epsilon_j)_1\right)}g.
		\end{equation*}
		
		Now, let $\eta_j:=\epsilon_j/\rho_j$, and define, for every~$\rho_j\widetilde{Q}\in\mathcal{B}_j$,
		\begin{align*}
			&M_{\Omega,j} := \sharp \left\{\ell+Q\subset \Omega/\epsilon_j \mbox{ s.t. } \ell\in\Z^n \mbox{ and }(\ell+Q)\cap\partial E_p\neq\varnothing\right\}\\
			\mbox{and}\quad& M_{\widetilde{Q},j} := \sharp\left\{\ell+Q\subset \widetilde{Q}/\eta_j \mbox{ s.t. } \ell\in\Z^n \mbox{ and }(\ell+Q)\cap\partial E_p\neq\varnothing\right\}.
		\end{align*}
		Since $E_p$ is a planelike in direction $p$, we have that
		\begin{equation} \label{eq::estimate_M_Omega}
			\begin{split}
				\lim_{j\to+\infty}\epsilon^{n-1}M_{\Omega,j} &= \lim_{j\to+\infty} \sharp \left\{\epsilon_j(\ell+Q)\subset \Omega \mbox{ s.t. } \ell\in\Z^n \mbox{ and }\epsilon_j(\ell+Q)\cap\partial \epsilon_jE_p\neq\varnothing\right\}\\
				&= \haus{n-1}(\Omega\cap\partial\mathcal{I}_p) 
			\end{split}
		\end{equation}
		and, similarly, for every $\rho_j\widetilde{Q}\in\mathcal{B}_j$,
		\begin{equation} \label{eq::estimate_M_Q}
			\begin{split}
				&\lim_{j\to+\infty} \eta^{n-1}M_{\widetilde{Q},j} = \lim_{j\to+\infty} \sharp\left\{\eta_j(\ell+Q)\subset \widetilde{Q} \mbox{ s.t. } \ell\in\Z^n \mbox{ and }\eta_j(\ell+Q)\cap\partial (\eta_jE_p)\neq\varnothing\right\}\\
				&\qquad = \haus{n-1}(\widetilde{Q}\cap\partial\mathcal{I}_p) = 1.
			\end{split}
		\end{equation}
		
		Moreover, recalling that the fractional part of a real number is a bounded function, we deduce that
		\begin{equation} \label{eq::reminder1}\begin{split}
			\lim_{j\to+\infty} \rho_j^{n-1} N_j &= \lim_{j\to+\infty} \rho_j^{n-1} \left(\rho_j^{1-n}\haus{n-1}(\Omega\cap\partial\mathcal{I}_p)-\{\rho_j^{1-n}\haus{n-1}(\Omega\cap\partial\mathcal{I}_p\}\right)\\& = \haus{n-1}(\Omega\cap\partial\mathcal{I}_p).
			\end{split}
		\end{equation}
		
		Thus, from~\eqref{eq::estimate_M_Omega},~\eqref{eq::estimate_M_Q}, and~\eqref{eq::reminder1} together, we infer that
		\begin{equation} \label{eq::small_g_reminder}
			\begin{split}
				&|\psi(j)| \leq \epsilon_j^{n-1}\norm{g}_{L^1(Q)}\left[M_{\Omega,j}-\sum_{\rho_j\widetilde{Q}\in \mathcal{B}_j} M_{\tilde{Q},j}\right]\\
				&\qquad = \norm{g}_{L^1(Q)}\left[\haus{n-1}(\Omega\cap\partial\mathcal{I}_p) + o(1)-\sum_{\rho_j\widetilde{Q}\in \mathcal{B}_j} \left(\rho^{n-1}+o(\rho^{n-1})\right)\right]\\
				&\qquad =  \norm{g}_{L^1(Q)}\left[\haus{n-1}(\Omega\cap\partial\mathcal{I}_p) + o(1)- N_j\left(\rho^{n-1}+o(\rho^{n-1})\right)\right],
			\end{split}
		\end{equation} which is infinitesimal as~$j\to+\infty$.
	
		Therefore, from~\eqref{eq::g_reminder} and recalling the definition of~$\E$ from~\eqref{eq::truncated_energy}, we obtain that
		\begin{equation*}
			\begin{split}
			\epsilon_j^{n-1}\F\left(E_p,\frac1{\epsilon_j}\Omega\right) &= \epsilon_j^{n-1}\E\left(E_p,\frac1{\epsilon_j}\Omega\right) + \epsilon_j^{n-1}\mathcal{L}_K\left(E_p\cap\frac1{\epsilon_j}\Omega,E_p^c\cap\frac1{\epsilon_j}\Omega^c\right)\\
				&  =\sum_{\rho_j\widetilde{Q}\in \mathcal{B}_j} \left[\epsilon_j^{n-1}\E\left(E_p,\frac{\rho_j}{\epsilon_j}\widetilde{Q}\right)\right] 
				+ \psi(j)+ \epsilon_j^{n-1}\P_K\left(E_p,\frac1{\epsilon_j}(\Omega\setminus R_j)\right)\\
				&\quad-\epsilon_j^{n-1}\mathcal{L}_K\left(E_p\cap\frac1{\epsilon_j}(\Omega\setminus R_j)^c,E_p^c\cap\frac1{\epsilon_j}(\Omega\setminus R_j)\right)\\
				&\quad+ \epsilon_j^{n-1}\mathcal{L}_K\left(E_p\cap\frac1{\epsilon_j}\Omega,E_p^c\cap\frac1{\epsilon_j}\Omega^c\right) \\
				&  =\sum_{\rho_j\widetilde{Q}\in \mathcal{B}_j} \left[\epsilon_j^{n-1}\F\left(E_p,\frac{\rho_j}{\epsilon_j}\widetilde{Q}\right) -\epsilon_j^{n-1}\mathcal{L}_K\left(E_p\cap\frac{\rho_j}{\epsilon_j}\widetilde{Q}, E_p^c\cap \frac{\rho_j}{\epsilon_j}\widetilde{Q}^c \right)\right] \\
				&\quad + \psi(j)+ \epsilon_j^{n-1}\P_K\left(E_p,\frac1{\epsilon_j}(\Omega\setminus R_j)\right)\\
				&\quad-\epsilon_j^{n-1}\mathcal{L}_K\left(E_p\cap\frac1{\epsilon_j}(\Omega\setminus R_j)^c,E_p^c\cap\frac1{\epsilon_j}(\Omega\setminus R_j)\right)\\
				&\quad + \epsilon_j^{n-1}\mathcal{L}_K\left(E_p\cap\frac1{\epsilon_j}\Omega,E_p^c\cap\frac1{\epsilon_j}\Omega^c\right) .
			\end{split}
		\end{equation*}
		
		At this point, observe that it follows from Corollary~\ref{cor::phi_translated_cube}, Lemma~\ref{lemma::alberti_bellettini}, and our choice of~$\rho_j\widetilde{Q}$ that
			\begin{equation*}
				\epsilon_j^{n-1}\F\left(E_p,\frac{\rho_j}{\epsilon_j}\widetilde{Q}\right) -\epsilon_j^{n-1}\mathcal{L}_K\left(E_p\cap\frac{\rho_j}{\epsilon_j}\widetilde{Q}, E_p^c\cap \frac{\rho_j}{\epsilon_j}\widetilde{Q}^c \right) = \rho_j^{n-1}\phi(p)+ o(\rho_j^{n-1}).
			\end{equation*}
		Thus, we infer that
		\begin{equation} \label{eq::estimate_energy_Omega}
			\begin{split}
				&\epsilon_j^{n-1}\F\left(E_p,\frac1{\epsilon_j}\Omega\right) \\
				=&N_j\left(\rho_j^{n-1}\phi(p)+ o(\rho_j^{n-1})\right) 
				+ \psi(j)+ \epsilon_j^{n-1}\P_K\left(E_p,\frac1{\epsilon_j}(\Omega\setminus R_j)\right) \\
				&\quad-\epsilon_j^{n-1}\mathcal{L}_K\left(E_p\cap\frac1{\epsilon_j}(\Omega\setminus R_j)^c,E_p^c\cap\frac1{\epsilon_j}(\Omega\setminus R_j)\right)\\
				&\quad + \epsilon_j^{n-1}\mathcal{L}_K\left(E_p\cap\frac1{\epsilon_j}\Omega,E_p^c\cap\frac1{\epsilon_j}\Omega^c\right).
			\end{split}
		\end{equation}
		
		Furthermore, by inspection of the proof of~\eqref{eq::energy_tail} in Lemma~\ref{lemma::finite_limit}, we infer that 
		\begin{equation}  \label{eq::reminder2}\begin{split}&
				\lim_{j\to+\infty} \epsilon_j^{n-1}\P_K\left(E_p,\Omega/\epsilon_j\setminus\left(\bigcup_{\rho_j\widetilde{Q}\in \mathcal{B}_j} \frac{\rho_j}{\epsilon_j}\widetilde{Q}\right)\right) \\&\qquad\leq \lim_{j\to+\infty} c\Big(\haus{n-1}(\Omega\cap\partial \mathcal{I}_p)-N_j\rho^{n-1}+\epsilon_j^{n-1}(N_j\rho_j\epsilon_j^{-1})^{2s_2-1-\alpha}\Big)=0.\end{split}
		\end{equation}
		
		Additionally, thanks to Lemma~\ref{lemma::alberti_bellettini}, we also have
		\begin{equation*}
				\limsup_{j\to+\infty}\, \epsilon_j^{n-1}\mathcal{L}_K\left(E_p\cap\frac1{\epsilon_j}(\Omega\setminus R_j)^c,E_p^c\cap\frac1{\epsilon_j}(\Omega\setminus R_j)\right) + \epsilon_j^{n-1}\mathcal{L}_K\left(E_p\cap\frac1{\epsilon_j}\Omega,E_p^c\cap\frac1{\epsilon_j}\Omega^c\right)=0.
		\end{equation*}
		Using this, \eqref{eq::reminder1},~\eqref{eq::small_g_reminder}, and~\eqref{eq::reminder2} in~\eqref{eq::estimate_energy_Omega}, we conclude that
		\begin{equation*}
			\lim_{j\to+\infty} \epsilon_j^{n-1}\F\left(E_p,\frac1{\epsilon_j}\Omega\right) = \haus{n-1}(\Omega\cap\partial\mathcal{I}_p)\phi(p).\qedhere
		\end{equation*}
	\end{proof}

	\section{Proof of the $\Gamma$-$\liminf$ inequality} \label{sec::gamma_liminf}

	Here, we present a proof of the~$\Gamma-\liminf$ inequality in Theorem~\ref{th::gamma_conv}-\eqref{item::gamma_conv_inf}.
	To this purpose, let~$E\subseteq \R^n$, and let~$\{E_\epsilon\}_\epsilon$ be a sequence of sets such that~$E_\epsilon\to E$ 
	in~$L^1_{\loc}(\R^n)$.
	
	Moreover, up to extracting a subsequence, we suppose that
	$$\liminf_{\epsilon\to0} \F_\epsilon(E_\epsilon,\Omega) = \lim_{\epsilon\to0} \F_\epsilon(E_\epsilon,\Omega) <+\infty,$$ 
	otherwise the result is trivial. In particular, without loss of generality, we assume that 
	\begin{equation*}
		\sup_{\epsilon} \F_\epsilon(E_\epsilon,\Omega)<+\infty,
	\end{equation*}
	
	Now, following the ideas of~\cite{MR1286918} (see also~\cite{MR1656477, MR1634336, chambolle_thouroude}), we consider the Radon measure
	\begin{equation}
		\lambda_\epsilon(\cdot) := \iiint_{(\cdot)\cap\Omega\times\R^n} \chi_{E_\epsilon}(x)\chi_{E_\epsilon^c}(y)K_\epsilon(x,y)\,dx\,dy  + \sum_{\substack{\ell\in\Z^n\\ \epsilon(\ell+Q)\subseteq \Omega}}  \delta_{\epsilon \ell}(\cdot) \iint_{E_\epsilon\cap \epsilon(\ell+Q)}g_\epsilon(x)\,dx,
	\end{equation}
	where, for any set $A\subseteq\R^n$, 
	\begin{equation*}
		\delta_{\epsilon \ell}(A):=
		\begin{cases}
			0,\quad&\mbox{if }\epsilon\ell\not\in A,\\
			1,\quad&\mbox{if }\epsilon\ell\in A.
		\end{cases}
	\end{equation*}
	
	By Proposition~\ref{prop::energy_lower_bound_BR}, we have that~$\lambda_\epsilon\geq0$. Moreover, the total mass of~$\lambda_\epsilon$ is bounded uniformly in~$\epsilon$. Indeed, by construction, we have that
	\begin{align*}
		&\lambda_\epsilon(\R^n) = \lambda_\epsilon(\Omega) = \iint_{\Omega\times\R^n} \chi_{E_\epsilon}(x)\chi_{E^c_\epsilon}(y)K_\epsilon(x,y)\,dx\,dy + \int_{E_\epsilon\cap\mathcal{Q}(\Omega)_\epsilon} g_\epsilon(x)\,dx \\
		&\qquad\qquad=\E_\epsilon(E_\epsilon,\Omega) \leq \F_\epsilon(E_\epsilon,\Omega) \leq \sup_{\epsilon} \F_\epsilon(E_\epsilon,\Omega)<+\infty,
	\end{align*}
	where $\E_\epsilon$ si defined as in~\eqref{eq::truncated_energy}.
	
	Therefore, there exists a positive measure~$\lambda$ such that
	\begin{align*}
		&\lambda_\epsilon \rightharpoonup^* \lambda,\mbox{ i.e. $\lambda_\epsilon$ converges weakly$*$ to $\lambda$ in the sense of measure}\\
		&\mbox{and}\quad \lambda(A) \leq \liminf_{\epsilon\to0} \lambda_\epsilon(A),\quad\mbox{for every set~$A$ such that~$\lambda(\partial A)=0$.}
	\end{align*}
	In particular, we have that
	\begin{equation} \label{eq::lambda_lsc}
		\lambda(\Omega) \leq \liminf_{\epsilon\to0} \lambda_\epsilon(\Omega) \leq \liminf_{\epsilon\to0} \F_\epsilon(E_\epsilon,\Omega).
	\end{equation}
	\begin{comment}
		Up to taking a subsequence, we assume that 
		\begin{equation*}
			\lambda(A) = \lim_{\epsilon\to0} \lambda_\epsilon(A),
		\end{equation*}
		for every set~$A$ such that~$\lambda(\partial A)=0$.
	\end{comment}
	
Now, let~$\overline{x}$ be a regular point of~$E$, namely a point~$\overline{x}\in\partial E$ with~$(n-1)$-density~$1$ and such that the blow-up of~$E$ around~$\overline{x}$ converges to~$\{(x-\overline{x})\cdot\nu_E(\overline{x})>0\}$. More precisely,
	\begin{align*}
		&\lim_{r\to0}\frac{\haus{n-1}(\partial E\cap B_r(\overline{x}))}{\omega_{n-1}r^{n-1}}=1,\\
		\mbox{and}\quad& \lim_{r\to0^+}\frac1{r^{n}}\int_{B_{2r}(\overline{x})} |\chi_E-\chi_{\{(x-\overline{x})\cdot\nu_E(\overline{x})>0\}}(x)|\,dx = 0,
	\end{align*}
	where $\omega_{n-1}:=\haus{n-1}(B^{n-1}_1)$.
	
	We consider the Radon-Nikodyn derivative of~$\lambda$ with respect to the~$(n-1)$-dimensional Hausdorff measure at~$\overline{x}$, namely
	\begin{equation*}
		\frac{d\lambda}{d\haus{n-1}}(\overline{x}) := \lim_{r\to0} \frac{\lambda(B_r(\overline{x}))}{\haus{n-1}(\partial E\cap B_r(\overline{x}))} =\lim_{r\to0} \frac{\lambda(B_r(\overline{x}))}{\omega_{n-1}r^{n-1}},
	\end{equation*}
	and observe that, in light of~\eqref{eq::lambda_lsc} and the Besicovitch Derivation Theorem (see~\cite[Theorem~5.52]{MR1857292}), Theorem~\ref{th::gamma_conv}-\eqref{item::gamma_conv_inf} will follow if we show that 
	\begin{equation} \label{eq::radon_nikodym}
		\frac{d\lambda}{d\haus{n-1}}(\overline{x}) \geq \phi (\nu_E(\overline{x})),
	\end{equation} where~$\phi$ is as in~\eqref{eq::anisotropy}.
	
Let us consider subsequences~$\{\epsilon_j\}_j$, $\{r_j\}_j$, and~$\{\eta_j:=\epsilon_j/r_j\}_j$ such that
	\begin{equation} \label{eq::taking_subseq}
		\begin{split}
			&\lim_{j\to+\infty} \frac{\lambda_{\epsilon_j}(B_{r_j}(\overline{x}))}{\omega_{n-1}r_j^{n-1}} =\frac{d\lambda}{d\haus{n-1}}(\overline{x}),\\
			&\lim_{j\to+\infty}\frac1{r_j^{n}}\int_{B_{2r_j}(\overline{x})} |\chi_E(x)-\chi_{\{(x-\overline{x})\cdot\nu_E(\overline{x})>0\}}(x)|\,dx = 0,\\
			\mbox{and}\quad&\lim_{j\to+\infty} \eta_j  = 0.
		\end{split}
	\end{equation}
	In what follows, in order to lighten the notation, we will write~$E_j:=E_{\epsilon_j}$, $\lambda_{\epsilon_j}:=\lambda_j$, and~$\nu:=\nu_E(\overline{x})$. %drop the subindex $j$.
	We now focus on proving~\eqref{eq::radon_nikodym}, that in this notation reads
	\begin{equation*}
		\lim_{j\to+\infty} \frac{\lambda_{j}(B_{r_j}(\overline{x}))}{\omega_{n-1}r_j^{n-1}} \geq\phi(\nu).
	\end{equation*}
	
	To achieve this, let us introduce, for every set~$A$, the notation~$A':=(A-\overline{x})/r_j$. Thus, taking advantage of the scaling properties of~$\lambda_{j}$, we obtain
	\begin{equation} \label{eq::rn_estimates1}
		\begin{split}
			&\frac{\lambda_{j}(B_{r_j}(\overline{x}))}{r_j^{n-1}} \\		
			=& r_j^{1-n}\left(\iint_{B_{r_j}(\overline{x})\times\R^n} \chi_{E_j}(x)\chi_{E^c_j}(y)K_{\epsilon_j}(x,y)\,dx\,dy + \int_{E_j\cap\mathcal{Q}(B_{r_j}(\overline{x}))_{\epsilon_j}} g_{\epsilon_j}(x)\,dx\right)\\
			=& \iint_{B_{1}\times\R^n} \chi_{E'_j}(x)\chi_{E'^c_j}(y)K_{\eta_j}(x+\overline{x}/r_j,y+\overline{x}/r_j)\,dx\,dy \\
			&+ \int_{E'_j\cap\mathcal{Q}'(B_{r_j}(\overline{x}))_{\epsilon_j}} g_{\eta_j}(x+\overline{x}/r_j)\,dx,
		\end{split}
	\end{equation}
	where
	\begin{equation*}
		\mathcal{Q}'(B_{r_j}(\overline{x}))_{\epsilon_j} := \left(\mathcal{Q}(B_{r_j}(\overline{x}))_{\epsilon_j}-\overline{x}\right)/r_j = \bigcup\left\{\frac{\epsilon_j(k+Q)-\overline{x}}{r_j} \mbox{s.t. }k\in\Z^n,\ \epsilon_j(k+Q)\subseteq B_{r_j}(\overline{x})\right\}.
	\end{equation*}
	
	We also set~$A'':=A'+\eta_j\{\overline{x}/\epsilon_j\}$, where~$\{x\}_j:=\{x_j\} = x_j-\lfloor x_j \rfloor$ is the fractional part of the~$j$-th component of a point~$x$. Then, \eqref{eq::rn_estimates1} boils down to
	\begin{equation*}
		\begin{split}
			\frac{\lambda_{j}(B_{r_j}(\overline{x}))}{r_j^{n-1}} &=  \iint_{B_{1}(\eta_j\{\overline{x}/\epsilon_j\})\times\R^n} \chi_{E''_j}(x)\chi_{E''^c_j}(y)K_{\eta_j}(x,y)\,dx\,dy + \int_{E''_j\cap\mathcal{Q}''(B_{r_j}(\overline{x}))_{\epsilon_j}} g_{\eta_j}(x)\,dx\\
			& =\iint_{B''_j\times\R^n} \chi_{E''_j}(x)\chi_{E''^c_j}(y)K_{\eta_j}(x,y)\,dx\,dy + \int_{E''_j\cap\mathcal{Q}''(B_{r_j}(\overline{x}))_{\epsilon_j}} g_{\eta_j}(x)\,dx,
		\end{split}
	\end{equation*}
	where, to be consistent with our notation, we wrote~$B_j'':=B_{1}(\eta_j\{\overline{x}/\epsilon_j\})$.
	
	Moreover, observe that
	\begin{equation*}
		\begin{split}
			\mathcal{Q}''(B_{r_j}(\overline{x}))_{\epsilon_j} &= \bigcup\left\{\frac{\epsilon_j(k+Q)-\overline{x}}{r_j}+\eta_j\left\{\frac{\overline{x}}{\epsilon_j}\right\} \mbox{s.t. }k\in\Z^n,\ \epsilon_j(k+Q)\subseteq B_{r_j}(\overline{x})\right\}\\
			& = \bigcup\left\{\eta_j\left(k-\left\lfloor\frac{\overline{x}}{\epsilon_j}\right\rfloor + Q\right) \mbox{s.t. }k\in\Z^n,\ \eta_j\left(k-\left\lfloor\frac{\overline{x}}{\epsilon_j}\right\rfloor + Q\right)\subseteq B''_j(\overline{x})\right\}\\
			& = \mathcal{Q}(B''_j)_{\eta_j}.
		\end{split}
	\end{equation*}
	Hence, it follows from the last two equations in display that
	\begin{equation} \label{eq::rn_estimates2}
		\frac{\lambda_{j}(B_{r_j}(\overline{x}))}{\omega_{n-1}r_j^{n-1}} =  \frac{\E_{\eta_j}(E''_j, B''_j)}{\omega_{n-1}}.
	\end{equation}
	
	Furthermore, since~$\eta_j\to0$ and the fractional part is a bounded function, we deduce from~\eqref{eq::taking_subseq} that
	\begin{equation}\label{eq::conv_halfspace}
		\begin{split}
			&\lim_{j\to+\infty} \int_{B_{3/2}}|\chi_{E''_j}(x)-\chi_{\mathcal{I}_\nu}(x)|\,dx \\
			\leq\;& \lim_{j\to+\infty} \int_{B_{3/2}}|\chi_{E''_j}(x)-\chi_{\{(x-\eta_j\{\overline{x}/\epsilon_j\})\cdot\nu>0\}}(x)|\,dx \\&\qquad+ \lim_{j\to+\infty} \int_{B_{3/2}}|\chi_{\{(x-\eta_j\{\overline{x}/\epsilon_j\})\cdot\nu>0\}}(x)-\chi_{\mathcal{I}_\nu}(x)|\,dx\\
			=\;& 0.
		\end{split}
	\end{equation}
	In particular, $\chi_{E''_j}\to\chi_{\mathcal{I}_\nu}$ a.e. in~$B_{3/2}$ (up to a subsequence).
	
	Now, let~$E_\nu$ be a planelike minimizer for~$\J$ in direction~$\nu$ constructed combining Theorem~\ref{th::existence_minimizer}, Theorem~\ref{th::level_sets_minimality}, and Corollary~\ref{cor::K_planelike}. Then, let us define sets
	\begin{equation} \label{eq::competitor_liminf}
		F_j := (\eta_j E_\nu\setminus \mathcal{Q}(B''_j)_1) \cup (E''_j\cap  \mathcal{Q}(B''_j)_1),
	\end{equation} see Figure~\ref{fig::G_liminf}.
	
	\begin{figure}
		\centering
		\includegraphics[width=.8\linewidth]{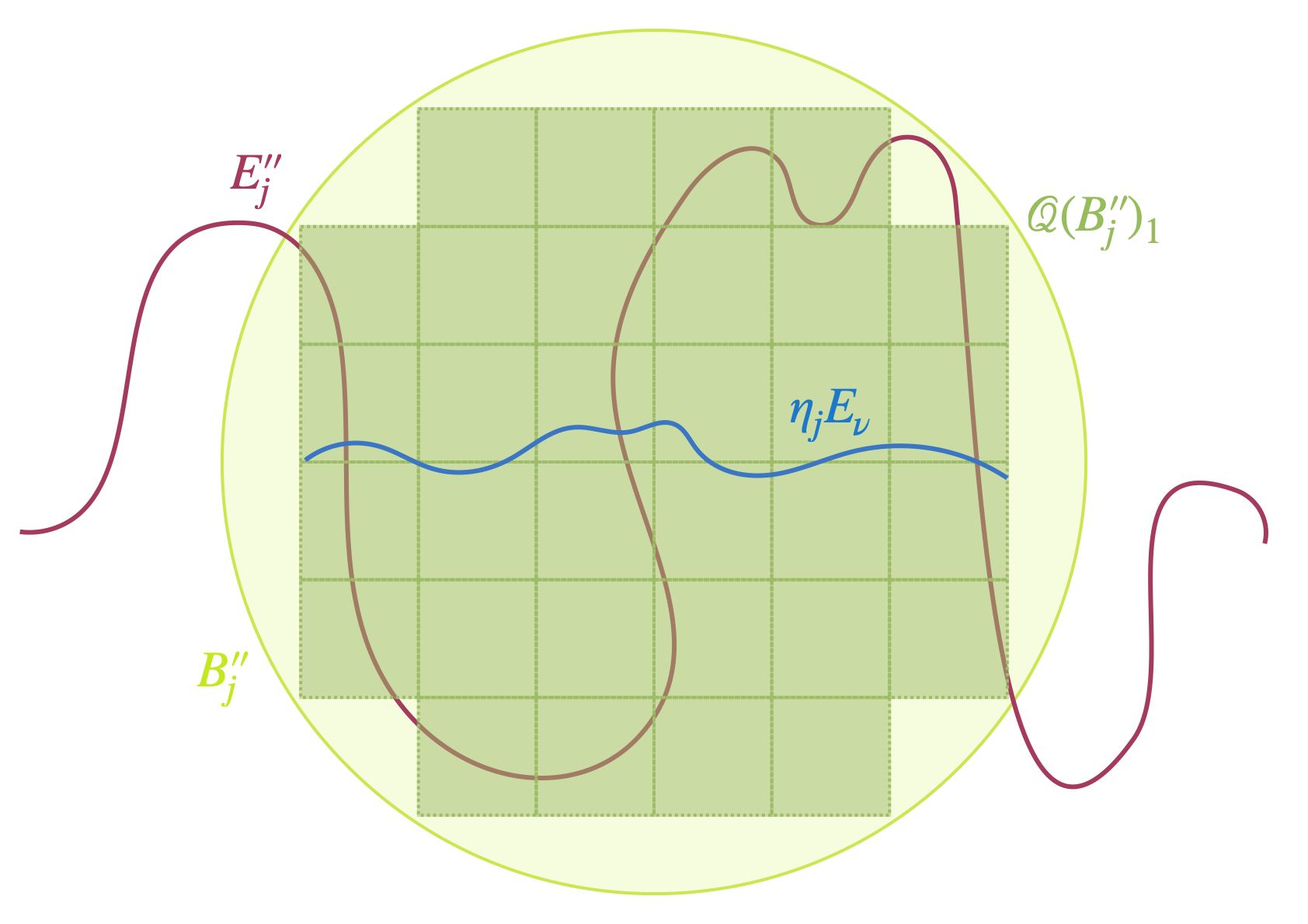}
		\caption{The competitor $F_j$ as in \eqref{eq::competitor_liminf} that produces the correct $\Gamma-\liminf$ inequality.}\label{fig::G_liminf}
	\end{figure}
	
	Since~$F_j\setminus\mathcal{Q}(B''_j)_1 = \eta_j E_\nu \setminus\mathcal{Q}(B''_j)_1$, by Theorem~\ref{th::level_sets_minimality}, Remark~\ref{rem::class_A_F} and Corollary~\ref{cor::tail_estimate_phi}, we infer that
	\begin{equation} \label{eq::conv_phi}
	\lim_{j\to+\infty}	\F_{\eta_j}(F_j, B''_j) \geq\lim_{j\to+\infty} \F_{\eta_j}(\eta_jE_\nu, B''_j) =\haus{n-1}(\partial \mathcal{I}_\nu\cap B_1)\phi(\nu) = \omega_{n-1}\phi(\nu).
	\end{equation}
	
	Also, observe that
	\begin{equation} \label{eq::measure_energy_lower_bound}
		\E_{\eta_j}(E''_j, B''_j) - \F_{\eta_j}(F_j, B''_j) \geq -\mathcal{L}_{K_{\eta_j}}(F_j\cap B''_j,F_j^c\cap (B''_j)^c) .
	\end{equation}
	Since both~$E_j''\to \mathcal{I}_\nu$ and~$\eta_j E_\nu\to \mathcal{I}_\nu$ a.e. in~$\R^n$, thanks to Lemma~\ref{lemma::alberti_bellettini}, we have that
	\begin{equation} \label{eq::fast_decay}
		\lim_{j\to\infty}  \mathcal{L}_{K_{\eta_j}}(F_j\cap B''_j,F_j^c\cap (B''_j)^c) =  0.
	\end{equation}
	
	Therefore, putting together~\eqref{eq::conv_phi}, \eqref{eq::measure_energy_lower_bound}, and~\eqref{eq::fast_decay}, we find that
	\begin{equation*}
		\liminf_{j\to+\infty} \E_{\eta_j}(E''_j, B''_j) \geq \lim_{j\to+\infty} \F_{\eta_j}(\eta_jE_\nu, B''_j)-\mathcal{L}_{K_{\eta_j}}(F_j\cap B''_j,F_j^c\cap (B''_j)^c) = \omega_{n-1}\phi(\nu).
	\end{equation*}
	
	This, together with~\eqref{eq::rn_estimates2}, entails the desired result.
	
	% ------------------------------------------------------------------------------------------------
	
	\section{Continuity of the stable norm $\phi$} \label{sec::phi_continuous}
	
	In this section, we discuss the continuity of the stable norm~$\phi$. This result is a version of~\cite[Corollary A.4]{chambolle_thouroude}
	covering our setting, which will play a crucial role in proving the~$\Gamma-\limsup$ inequality (see Theorem~\ref{th::gamma_conv}-\eqref{item::gamma_conv_sup}) in Section~\ref{sec::gamma_limsup}.
	
	\begin{remark}
		Before presenting our result, we point out that, in~\cite[Corollary~A.4]{chambolle_thouroude}, the authors claim that the same argument of~\cite[Lemma 10.2]{MR1852978} suffices to prove the convexity of the stable norm~$\phi$. However, because of—plausibly—a typo in the proof of the latter, such argument cannot be directly applied in our context.
	\end{remark}
	
	\begin{proposition} \label{prop::phi_continuous} %\label{lemma::phi_convex}
		$\phi:\S^{n-1}\to\R$ is a continuous function.
	\end{proposition}
	
	\begin{proof}
		Let~$p$, $q\in\S^{n-1}$ and define~$\theta$ as the angle between~$p$ and~$q$ (i.e.~$\theta$ is defined through~$p\cdot q = \cos\theta$), so that if~$\mathcal{R}_\theta\in SO(n)$ denotes a rotation of angle~$\theta$, then~$p=\mathcal{R}_\theta q$.
		
		The continuity of~$\phi$ will follow if we show the existence of a function~$\psi$ such that
		\begin{equation*}
|\phi(q)-\phi(p)| \leq \psi(\theta)\qquad
\mbox{and}\qquad\lim_{\theta\to0} \psi(\theta) = 0.
		\end{equation*}
		
		To this purpose, we use Theorem~\ref{th::gamma_conv}-\eqref{item::gamma_conv_inf} (notice indeed that~$\epsilon \mathcal{R}_\theta E_q$ converges to~$\mathcal{I}_p$ as~$\epsilon\to0$)
		and Proposition~\ref{prop::existence_phi} to obtain
		\begin{equation*}
			\liminf_{\epsilon\to0} \F_\epsilon(\epsilon \mathcal{R}_\theta E_q,Q^p) \geq \F_\phi(\mathcal{I}_p,Q^p)=\phi(p)=\lim_{\epsilon\to0} \F_\epsilon(\epsilon E_p,Q^p).
		\end{equation*}
		From this and the fact that~$K$ is rotation invariant (see~\eqref{eq::K_invariance}), it follows that
		\begin{equation} \label{eq::phi_variation1}
			\begin{split}
				\phi(q)-\phi(p) &= \lim_{\epsilon\to0}\Big( \F_\epsilon(\epsilon E_q,Q^q)- \F_\epsilon(\epsilon E_p,Q^p)\Big)\\
				&\geq \liminf_{\epsilon\to0}\Big( \F_\epsilon(\epsilon E_q,Q^q)- \F_\epsilon(\epsilon \mathcal{R}_\theta E_q,Q^p)\Big) \\
				&= \liminf_{\epsilon\to0} \epsilon^{n-1}\left(\int_{E_q\cap \mathcal{Q}(Q^q/\epsilon)_1} g - \int_{\mathcal{R}_\theta E_q\cap \mathcal{Q}(Q^p/\epsilon)_1} g\right).
			\end{split}
		\end{equation}
		
We now recall that
		$$ \partial E_q \subseteq \{|x\cdot q| \le M\} =: S^q_{M} $$
		and hence
		\begin{equation*}
			\bigcup\big\{k+Q \mbox{ s.t. }k\in\Z^n \mbox{ and }(k+Q)\cap\partial E_q \neq\varnothing\big\} \subseteq S_{M+\sqrt{n}}^q.
		\end{equation*}
		Thus, in light of~\eqref{eq::g_zero_avg}, we find that
		\begin{equation*}
			\int_{E_q\cap \mathcal{Q}(Q^q/\epsilon)_1} g - \int_{\mathcal{R}_\theta E_q\cap \mathcal{Q}(Q^p/\epsilon)_1} g 
= \int_{E_q\cap \mathcal{Q}(Q^q/\epsilon)_1 \cap S^q_{M+\sqrt{n}}} g - \int_{\mathcal{R}_\theta E_q\cap \mathcal{Q}(Q^p/\epsilon)_1 \cap \mathcal{R}_\theta S^q_{M+\sqrt{n}}} g.
		\end{equation*}
		Therefore, using also a H\"older inequality in~\eqref{eq::phi_variation1} and observing that $\mathcal{R}_\theta S^q_{M+\sqrt{n}} = S^p_{M+\sqrt{n}}$, we obtain
		\begin{equation*}
			\begin{split}
				&\phi(q)-\phi(p) \\&\quad
				\geq \liminf_{\epsilon\to0} \left(-\epsilon^{n-1}\norm{g}_{L^\infty(\R^n)} \left| 
				\Big(E_q\cap \mathcal{Q}(Q^q/\epsilon)_1\cap S^q_{M+\sqrt{n}}\Big)\Delta \Big(\mathcal{R}_\theta E_q\cap \mathcal{Q}(Q^p/\epsilon)_1\cap \mathcal{R}_\theta S^q_{M+\sqrt{n}}\Big)\right|\right)\\
				&\quad \geq \liminf_{\epsilon\to0} \left(-\epsilon^{n-1}\norm{g}_{L^\infty(\R^n)} \left| \Big(E_q\cap \mathcal{Q}(Q^q/\epsilon)_1\cap S^q_{M+\sqrt{n}}\Big)\Delta \Big(\mathcal{R}_\theta E_q\cap \mathcal{Q}(Q^p/\epsilon)_1\cap  S^p_{M+\sqrt{n}}\Big)\right|\right).
			\end{split}
		\end{equation*}
		
		Now, observe that, arguing as in~\eqref{eq::term_zeta_to_0}, we have that
		\begin{equation*}
			\begin{split}
				&\left| (E_q\cap \mathcal{Q}(Q^q/\epsilon)_1\cap S^q_{M+\sqrt{n}})\Delta (\mathcal{R}_\theta E_q\cap \mathcal{Q}(Q^p/\epsilon)_1\cap S^p_{M+\sqrt{n}})\right|\\
				&\qquad\leq \left| [E_q\cap S^q_{M+\sqrt{n}}\cap \frac{1}{\epsilon}C^q] \Delta [ \mathcal{R}_\theta E_q\cap S^p_{M+\sqrt{n}}\cap \frac{1}{\epsilon}C^p]\right| + \epsilon^{1-n}\zeta(\epsilon),
			\end{split}
		\end{equation*}
		where $C^q$ is the (unitary) square-based cylinder with axis $q$
		% $$ C^q := \{x-x\cdot q\}$$
		and
		\begin{equation*}
			\begin{split}
				\epsilon^{1-n}\zeta(\epsilon) := &\left|S^q_{M+\sqrt{n}}\cap S^p_{M+\sqrt{n}}\cap\mathcal{R}_\theta E^q\cap \big(
				C^p/\epsilon \setminus \mathcal{Q}(Q^p/\epsilon)_1\big)\right|\\
				&+\left|S^q_{M+\sqrt{n}}\cap S^p_{M+\sqrt{n}}\cap E^q \cap \big(
				C^q/\epsilon \setminus \mathcal{Q}(Q^q/\epsilon)_1\big)\right|,
			\end{split}
		\end{equation*}
		with~$\zeta(\epsilon)\to0$, as~$\epsilon\to0$. 
		
		Thus,
		\begin{equation*}
			\begin{split}
				&\phi(q)-\phi(p) \\
				\geq&  \liminf_{\epsilon\to0} \left(-\epsilon^{n-1}\norm{g}_{L^\infty(\R^n)} \left| \left(
				E_q\cap S^q_{M+\sqrt{n}}\cap \frac{1}{\epsilon}C^q\right) \Delta \left(  \mathcal{R}_\theta E_q\cap S^p_{M+\sqrt{n}}\cap \frac{1}{\epsilon}C^p\right)\right|\right).
				%\\ =& -\norm{g}_\infty \left| [\mathcal{I}_q\cap S^q_{M+\sqrt{n}}\cap C^q] \Delta [\mathcal{I}_p\cap \mathcal{R}_\theta S^q_{M+\sqrt{n}}\cap C^p]\right|.
			\end{split}
		\end{equation*}
		
		Let us define 
		\begin{equation*}
			%\begin{split}
			\psi(\theta) := \limsup_{\epsilon\to0} \epsilon^{n-1} \left| \left(
			E_q\cap S^q_{M+\sqrt{n}}\cap \frac{1}{\epsilon}C^q\right) \Delta \left(  \mathcal{R}_\theta E_q\cap S^p_{M+\sqrt{n}}\cap \frac{1}{\epsilon}C^p\right)\right|,
			% &:= \norm{g}_\infty \left| [\mathcal{I}_q\cap S^q_{M+\sqrt{n}}\cap C^q] \Delta [\mathcal{I}_p\cap \mathcal{R}_\theta S^q_{M+\sqrt{n}}\cap C^p]\right|\\
			% &\qquad =  \norm{g}_\infty \left| [\mathcal{I}_q\cap S^q_{M+\sqrt{n}}\cap C^q] \Delta \mathcal{R}_\theta [\mathcal{I}_q\cap S^q_{M+\sqrt{n}}\cap C^q]\right|,
			%\end{split}
		\end{equation*}
		so that 
		\begin{equation*}
			\phi(q)-\phi(p) \geq -\norm{g}_{L^\infty(\R^n)}\psi(\theta).
		\end{equation*}
		Moreover, notice that~$\psi(\theta)\to0$, as~$\theta\to0$ (namely as~$p\to q$).
		
		Interchanging the roles of~$p$ and~$q$, we obtain that
		\begin{equation*}
			\begin{split}
				&\phi(p)-\phi(q) \geq \liminf_{\epsilon\to0}\left( -\norm{g}_{L^\infty(\R^n)} \left| \big(
				E_q\cap S^p_{M+\sqrt{n}}\cap C^p\big) \Delta \big(\mathcal{R}_{-\theta}\cap S^q_{M+\sqrt{n}}\cap C^q\big)\right|\right) \\
				&\qquad\qquad= -\norm{g}_{L^\infty(\R^n)}|J \mathcal{R}_\theta| \psi(\theta)= -\norm{g}_{L^\infty(\R^n)} \psi(\theta),
			\end{split}
		\end{equation*}
		and hence 
		\begin{equation*}
			|\phi(q)-\phi(p)| \leq \norm{g}_\infty\psi(\theta),
		\end{equation*}
		as desired.
	\end{proof}
	
	Now, we define the~$1$-homogeneous extension of~$\phi$ as 
	\begin{equation} \label{eq::homog_anisotropy}
		\widetilde{\phi}(p) :=
		\begin{cases}
			|p|\phi\left(\frac{p}{|p|}\right),\quad&{\mbox{ if }}p\neq0,\\
			0,\quad&{\mbox{ if }} p=0.
		\end{cases}
	\end{equation}	
		As a byproduct of Proposition~\ref{prop::phi_continuous}, we obtain that also~$\widetilde{\phi}$ is continuous (in the whole~$\R^n$).
	
	\begin{corollary} \label{cor::phi_homog_continuous}
		$\widetilde{\phi}:\R^{n}\to\R$ is a continuous function.
	\end{corollary}
	
	\begin{proof}
		Let~$q\in\R^n$. Then, it follows from the homogeneity of~$\widetilde{\phi}$ and Proposition~\ref{prop::phi_continuous} that
		\begin{equation*}
			\begin{split}
				&\lim_{p\to q}\widetilde{\phi}(p) 
				= \lim_{p\to q}|p|{\phi}\left(\frac{p}{|p|}\right) 
				=  \lim_{p\to q}\left((|p|-|q|){\phi}\left(\frac{p}{|p|}\right) + |q|{\phi}\left(\frac{p}{|p|}\right)\right)\\
			&\qquad = \lim_{p\to q}\left(
(|p|-|q|){\phi}\left(\frac{p}{|p|}\right)  
	+ |q|{\phi}\left(\frac{q}{|q|}\right)
	\right) 
				=  \lim_{p\to q}\left( (|p|-|q|){\phi}
				\left(\frac{p}{|p|}\right)\right) 
				 + \widetilde{\phi}(q).
			\end{split}
		\end{equation*}
		Since~${\phi}(p/|p|)$ is bounded, we conclude that
		\begin{equation*}
			\lim_{p\to q}\widetilde{\phi}(p) = \widetilde{\phi}(q),
		\end{equation*}
		showing the continuity of~$\widetilde{\phi}$ at~$q$.
		\end{proof}
	
	% ------------------------------------------------------------------------------------------------------
	
	\section{Proof of the $\Gamma$-$\limsup$ inequality} \label{sec::gamma_limsup}

	This section is devoted to proving~$\Gamma$-$\limsup$ inequality in Theorem~\ref{th::gamma_conv}-\eqref{item::gamma_conv_sup}. Our argument adapts to the setting under consideration here the ideas of~\cite{MR1852978, chambolle_thouroude}, in which the result follows from a standard polyhedral approximation taking advantage of the existence of planelike minimizers. For this, we need the following technical result.
	
	\begin{lemma}\label{lemma::limsup_polyhedron}
		Let~$E\subseteq \R^n$ be such that~$E\cap \Omega$ is a polyhedron. Then, there exists a sequence of sets~$\{E_\epsilon\}_\epsilon$ such that
		\begin{equation*}
			\limsup_{\epsilon\to0} \F_{\epsilon}(E_\epsilon,\Omega) \leq \F_\phi(E,\Omega).
		\end{equation*}
	\end{lemma}
	\begin{proof}
		Since~$E\cap \Omega$ is a polyhedron, there exist points~$x_j\in \partial E\cap \Omega$ and directions~$p_j\in\S^{n-1}$, with~$j=1,\dots,N$, for some~$N\in\N$, such that
		$$ \partial E\cap \Omega = \bigcup_{j=1}^N( x_j+\mathcal{I}_{p_j})\cap\Omega.$$
Here above and in the rest of the proof, the notation in~\eqref{eq::def_halfsapce} for~$\mathcal{I}_{p_j}$ is used.

Also, there are~$N$ disjoint Lipschitz domains~$\Omega_j$ such that
$$\Omega=\bigcup_{j=1}^N\Omega_j\qquad {\mbox{and}}\qquad
(x_j+\mathcal{I}_{p_j})\cap\Omega\subseteq \Omega_j.$$
		
		Now, for any~$\epsilon>0$, we define 
		\begin{equation} \label{eq::poly_approx}
			E_\epsilon := (E\setminus \Omega) \cap \bigcup_{j=1}^N \Big(\Omega_j \cap \big(x_j+\epsilon E_{p_j}\big)\Big),
		\end{equation}
		where, for every~$j=1,\dots,N$, $E_{p_j}$ is a planelike minimizer constructed as in Theorem~\ref{th::existence_minimizer} (see Figure~\ref{fig::G_limsup}). 
		
		\begin{figure}
			\centering
			\includegraphics[width=.85\linewidth]{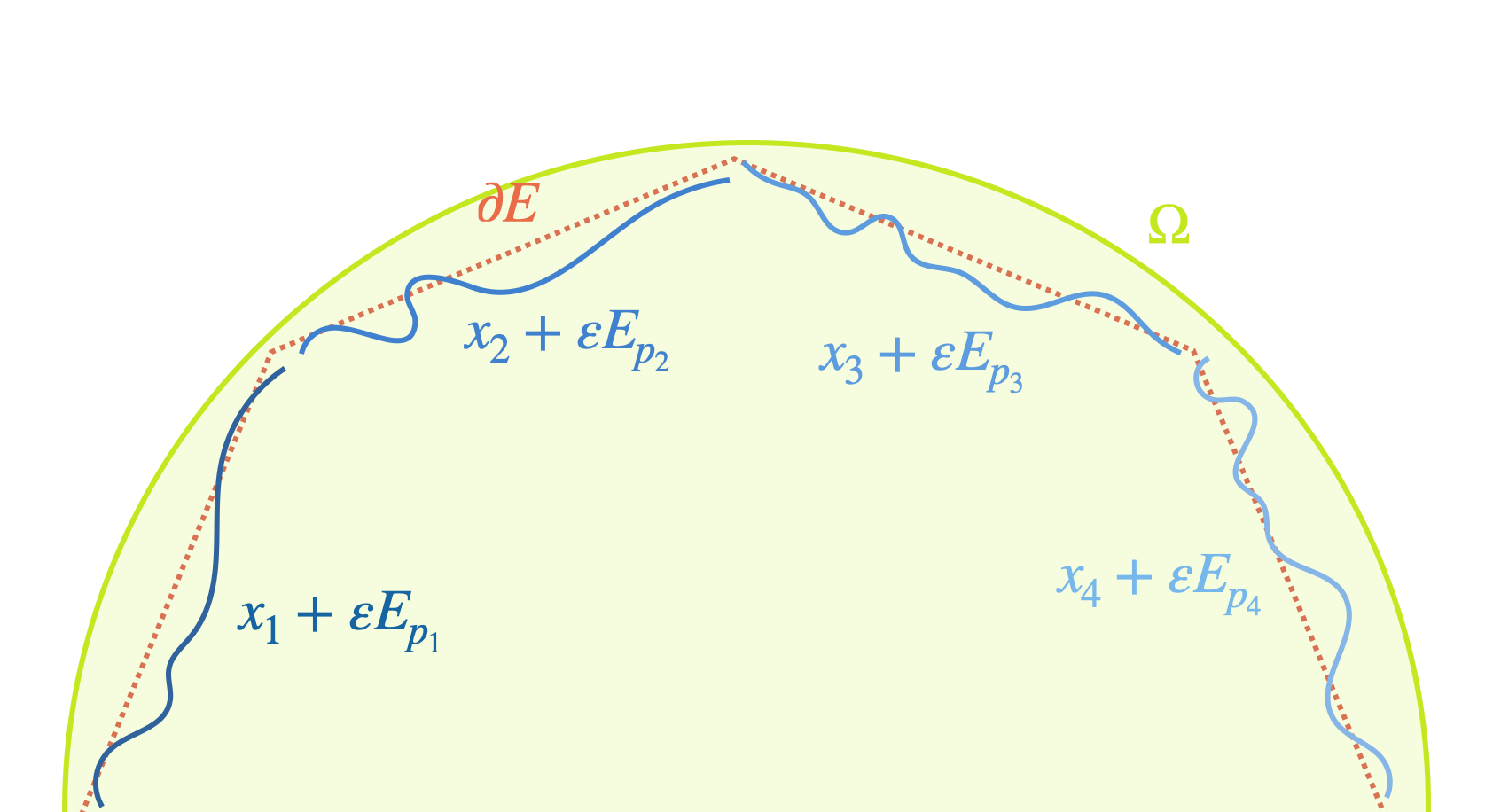}
			\caption{The planelike approximation of the polygonal set $E$ as in \eqref{eq::poly_approx}.}\label{fig::G_limsup}
		\end{figure}
		
		Therefore, we obtain
		\begin{equation} \label{eq::fragmentation}
			\begin{split}
				\F_\epsilon(E_\epsilon,\Omega) &\leq \sum_{j=1}^{N}\F_\epsilon(x_j+\epsilon E_{p_j},\Omega_j) + \int_{E_\epsilon\cap (\mathcal{Q}(\Omega)_\epsilon\setminus \bigcup_{j=1}^N\mathcal{Q}(\Omega_j)_\epsilon)} g_\epsilon (x)\,dx\\
				&\quad+ \sum_{j=1}^{N} \mathcal{L}_{K_\epsilon}(\epsilon E_{p_j}\cap\Omega_j, E_\epsilon^c\cap\Omega_j^c)+ \sum_{j=1}^{N} \mathcal{L}_{K_\epsilon}(\epsilon E_{p_j}^c\cap\Omega_j, E_\epsilon\cap\Omega_j^c) \\
				&\quad- \sum_{j=1}^{N} \mathcal{L}_{K_\epsilon}(\epsilon E_{p_j}\cap\Omega_j, \epsilon E_{p_j}^c\cap\Omega_j^c)- \sum_{j=1}^{N} \mathcal{L}_{K_\epsilon}(\epsilon E_{p_j}^c\cap\Omega_j, \epsilon E_{p_j}\cap\Omega_j^c) .
			\end{split}
		\end{equation}
		
		By Corollary~\ref{cor::tail_estimate_phi}, for every~$j\in\{1,\dots,N\}$, we have that
		\begin{equation} \label{eq::conv_length}
			\lim_{\epsilon\to0} \F_\epsilon(x_j+\epsilon E_{p_j},\Omega_j) = \phi(p_j)\haus{n-1}\big(\Omega_j\cap(x_j+\partial\mathcal{I}_{p_j})\big).
		\end{equation}
		
		Furthermore,
		\begin{equation*}
			\begin{split}
				&\int_{E_\epsilon\cap (\mathcal{Q}(\Omega)_\epsilon\setminus \bigcup_{j=1}^N\mathcal{Q}(\Omega_j)_\epsilon)} g_\epsilon(x)\, dx \leq \epsilon^{n-1} \int_{(E_\epsilon/\epsilon)\cap (\mathcal{Q}(\Omega/\epsilon)_1\setminus \bigcup_{j=1}^N\mathcal{Q}(\Omega_j/\epsilon)_1)} |g(x)|\, dx\\
				&\qquad = \epsilon^{n-1} \sum_{j=1}^{N} \int_{(x_j/\epsilon+E_{p_j})\cap (\Omega_j/\epsilon) \cap (\mathcal{Q}(\Omega/\epsilon)_1\setminus \mathcal{Q}(\Omega_j/\epsilon)_1} |g(x)|\, dx\\
				&\qquad \leq 2(M+\sqrt{n})\norm{g}_{L^1(Q)}\epsilon^{n-1}\sum_{j=1}^{N}\Big(\epsilon^{1-n}\haus{n-1}(\Omega_j\cap \partial\mathcal{I}_{p_j})-\lfloor \epsilon^{1-n}\haus{n-1}(\Omega_j\cap \partial\mathcal{I}_{p_j}) \rfloor \Big).
			\end{split}
		\end{equation*}
		Thus,
		\begin{equation}  \label{eq::vanishing_g}
			\limsup_{\epsilon\to0} \int_{E_\epsilon\cap (\mathcal{Q}(\Omega)_\epsilon\setminus \bigcup_{j=1}^N\mathcal{Q}(\Omega_j)_\epsilon)} g_\epsilon(x)\, dx \leq 0.
		\end{equation}
		
		Moreover, thanks to Lemma~\ref{lemma::alberti_bellettini}, we infer that
		\begin{equation} \label{eq::vanishing_dim_n-2}
			\begin{split}
				\lim_{\epsilon\to0} &\sum_{j=1}^{N} \mathcal{L}_{K_\epsilon}(\epsilon E_{p_j}\cap\Omega_j, E_\epsilon^c\cap\Omega_j^c)+ \sum_{j=1}^{N} \mathcal{L}_{K_\epsilon}(\epsilon E_{p_j}^c\cap\Omega_j, E_\epsilon\cap\Omega_j^c) \\
				&- \sum_{j=1}^{N} \mathcal{L}_{K_\epsilon}(\epsilon E_{p_j}\cap\Omega_j, \epsilon E_{p_j}^c\cap\Omega_j^c)- \sum_{j=1}^{N} \mathcal{L}_{K_\epsilon}(\epsilon E_{p_j}^c\cap\Omega_j, \epsilon E_{p_j}\cap\Omega_j^c) =0,
			\end{split}
		\end{equation}
		for all~$j\in\{1,\dots,N\}$. 
		
		Hence, gathering~\eqref{eq::fragmentation}, \eqref{eq::conv_length}, \eqref{eq::vanishing_g},
		and~\eqref{eq::vanishing_dim_n-2}, we conclude that
		\begin{equation*}
			\limsup_{\epsilon\to0} \F_\epsilon(E_\epsilon,\Omega) \leq \sum_{j=1}^{N} \phi(p_j)\haus{n-1}\big(\Omega_j\cap(x_j+\partial\mathcal{I}_{p_j})\big) = 
			\int_{\partial E \cap \Omega} \phi(\nu_E)\,d\haus{n-1}.
		\end{equation*}
		as desired. 
	\end{proof}
	
	\begin{proof}[Proof of Theorem~\ref{th::gamma_conv}-\eqref{item::gamma_conv_sup}]
		By a standard approximation argument, there exists a sequence of sets~$\{E_j\}_j$ such that~$E_j\cap \Omega$ is a polyhedron, for every~$j$, and 
		$$ \lim_{j\to+\infty} \haus{n-1}(E_j\cap\Omega) = \mbox{Per}(E,\Omega).$$
		Therefore, recalling that (the~$1$-homogeneous extension of)~$\phi$ is continuous by Corollary~\ref{cor::phi_homog_continuous}, we deduce from~\cite[Theorem~2.39]{MR1857292} that 
		$$ \lim_{j\to+\infty} \F_\phi(E_j,\Omega) = \F_\phi(E,\Omega) .$$
		
		Using Lemma~\ref{lemma::limsup_polyhedron} and taking a diagonal sequence, we construct a sequence of sets~$\{E_{\epsilon_j}\}_j$ such that~$E_{\epsilon_j}\to E$ in~$L^1_{\loc}(\R^n)$, and
		$$ \limsup_{j\to+\infty} \F_{\epsilon_j}(E_{\epsilon_j},\Omega) \leq \F_\phi(E,\Omega) , $$
		showing Theorem~\ref{th::gamma_conv}-\eqref{item::gamma_conv_sup}.
	\end{proof}
	
	\begin{remark}
		We point out that, in light of Lemma~\ref{lemma::alberti_bellettini}, 
		one can equivalently
		consider energies~$\F_\epsilon((\cdot),\Omega)$, $\E_\epsilon((\cdot),\Omega)$, or even~$\mathscr{L}_\epsilon((\cdot),\Omega):=\mathcal{L}_{K_\epsilon}((\cdot)\cap\Omega,(\cdot)^c\cap\Omega)$ in Theorem~\ref{th::gamma_conv}. Indeed, since the rescaled 
		planelike~$\epsilon E_p$ converges to the straight line~$\mathcal{I}_p$, as~$\epsilon\to0$, such limit crosses transversely any Lipschitz domain. Therefore, the energy contribution due to the interactions with the exterior of~$\Omega$ vanishes in the limit.
	\end{remark}
	\begin{comment}	
		On the other hand, choosing one energy over another comes with its own advantages and drawbacks. For instance, the energy~$\F_\epsilon^\Omega$ is symmetric with respect to the domain, in the fashion of the non-local Allen-Cahn energy, where, however, the contribution coming from outside~$\Omega$ cannot be omitted (see, for instance, \cite{MR2948285}). The energy~$\E_\epsilon^\Omega$ is additive, and hence it is more suitable to prove Theorem~\ref{th::gamma_conv}-\eqref{item::gamma_conv_inf}.
		%thanks to Lemma~\ref{lemma::isoperimetric_estimate} and following the argument in~\cite[Proposition~7.3]{MR1852978}, we infer that its minimizers must be connected. 
		Lastly, the energy~$\mathscr{L}_\epsilon^\Omega$, which is similar to the one considered in~\cite{MR1634336}, is easier to deal with in the proof of Theorem~\ref{th::gamma_conv} since it does not require Lemma~\ref{lemma::alberti_bellettini} to estimate the contribution due to~$\Omega^c$. 
		%in the proof of Theorem~\ref{th::gamma_conv}-\eqref{item::gamma_conv_inf}.
	\end{comment}
	
	As a corollary of the~$\Gamma$-convergence, we improve Proposition~\ref{prop::phi_continuous}. Indeed, since a~$\Gamma$-limit is lower semi-continuous by construction, we infer that the stable norm must be convex.
	
	Recall the definitions of the stable norm~$\phi:\S^{n-1}\to\R$ in~\eqref{eq::anisotropy}, and of its homogenization~$\widetilde{\phi}$ in~\eqref{eq::homog_anisotropy}. Then, the following holds true:
	\begin{corollary}\label{cor::phi_convex}
		The function~$\widetilde{\phi}:\R^n\to\R$ is convex.
	\end{corollary}
		
	% ----------------------------------------------------------------------------------------------------------------
	
	\begin{appendix}
		
		\section{Norms subject to short-range interactions}
		
		We point out a useful equivalence of norms when a kernel is restricted to close-by interactions (see e.g.~\cite[Lemma~2.1]{2025arXiv250409976A} for related results). For the rest of this section, let~$s\in(0,1/2)$.
		
		\begin{lemma}\label{lemma::short_range_norm} There exists~$C>0$, depending only on~$\delta$, $n$, and~$s$, such that,
			for all~$u:\R^n\to\R$ such that
			\begin{equation*}\int_Q u(x)\,dx=0,\end{equation*} we have that
			$$ \iint_{Q\times Q} \frac{|u(x)-u(y)|}{|x-y|^{n+2s}}\,dx\,dy\le
			C\iint_{(Q\times Q)\cap\{|x-y|<\delta\}} \frac{|u(x)-u(y)|}{|x-y|^{n+2s}}\,dx\,dy.$$
		\end{lemma}
		
		\begin{proof} Up to renaming constants, the desired result follows if we prove that
			$$ \iint_{(Q\times Q)\cap\{|x-y|\ge \delta\}} \frac{|u(x)-u(y)|}{|x-y|^{n+2s}}\,dx\,dy\le
			C\iint_{(Q\times Q)\cap\{|x-y|<\delta\}} \frac{|u(x)-u(y)|}{|x-y|^{n+2s}}\,dx\,dy.$$
			
			Moreover, since
			\begin{equation*}\begin{split}
					\iint_{(Q\times Q)\cap\{|x-y|\ge \delta\}} \frac{|u(x)-u(y)|}{|x-y|^{n+2s}}\,dx\,dy&
					\le\frac{1}{\delta^{n+2s}}
					\iint_{Q\times Q}\big( |u(x)|+|u(y)|\big)\,dx\,dy\\&=\frac{2}{\delta^{n+2s}}\int_Q|u(x)|\,dx,
			\end{split}\end{equation*}
			it suffices to show that
			$$\int_Q|u(x)|\,dx\le
			C\iint_{(Q\times Q)\cap\{|x-y|<\delta\}} \frac{|u(x)-u(y)|}{|x-y|^{n+2s}}\,dx\,dy.$$
			
			We argue for the sake of contradiction and suppose that this is not true.
Namely, we suppose that
there exists  a sequence of functions~$u_j:\R^n\to\R$ such that
\begin{equation}\label{f8u49thgeoigeg98765}\int_Q u_j(x)\,dx=0\end{equation} and
			\begin{equation}\label{0122}
				1=\int_Q|u_j(x)|\,dx>j\iint_{(Q\times Q)\cap\{|x-y|<\delta\}} \frac{|u_j(x)-u_j(y)|}{|x-y|^{n+2s}}\,dx\,dy.
			\end{equation}
			We decompose~$Q$ as the disjoint union of cubes~$Q_1,\dots,Q_N$ of side less than~$\frac{\delta}2$ and we deduce from~\eqref{0122} that, for all~$\ell\in\{1,\dots,N\}$,
			\begin{equation}\label{0123} 
				\frac1j>\iint_{\substack{Q_\ell\times (Q_\ell)_\delta \\ (Q\times Q)\cap\{|x-y|<\delta\} }} \frac{|u_j(x)-u_j(y)|}{|x-y|^{n+2s}}\,dx\,dy
				\geq\iint_{Q_\ell\times Q_\ell} \frac{|u_j(x)-u_j(y)|}{|x-y|^{n+2s}}\,dx\,dy,
			\end{equation}
			where
			$$ (Q_\ell)_\delta := \{x\in\R^n \mbox{ s.t. } \mbox{dist}(x,Q_\ell)<\delta\}.$$
			
			Then (see e.g.~\cite[Theorem~7.1]{MR2944369}), the sequence~$u_j$ is precompact in~$L^1(Q_\ell)$
			and consequently, up to a subsequence, $u_j$ converges to some function~$u_0$ in~$L^1(Q_\ell)$ for all~$\ell\in\{1,\dots,N\}$, and thus in~$L^1(Q)$ and a.e. in~$Q$.
			
Thus, from~\eqref{0123} and Fatou Lemma,
$$		\iint_{Q_\ell\times Q_\ell} \frac{|u_0(x)-u_0(y)|}{|x-y|^{n+2s}}\,dx\,dy=0,$$
which implies that~$u_0$ is constant.	As a result, since
we have from~\eqref{f8u49thgeoigeg98765} that
$$ \int_Q u_0(x)\,dx=0,$$ we conclude that~$u_0$ is the null function.

 However, according to~\eqref{0122},
			$$1=\lim_{j\to+\infty}\int_Q|u_j(x)|\,dx=\int_Q|u_0(x)|\,dx,$$
			providing the desired  contradiction.
		\end{proof}
		
		% ------------------------------------------------------------------------------------------------
		
		\section{Uniform density estimates for~$\left(\Lambda,\frac{2s_1}{n}\right)$-minimizers of~$\P_K$} \label{appendix::ude}
	
		In this section, we revisit the proof of~\cite[Theorem 2.2]{sequoia} adapting it to the setting under consideration here. The precise statement goes as follows:
	
		\begin{proposition} \label{prop::ude_almost_min}
			Let~$K$ satisfy~\eqref{eq::K_invariance}, \eqref{eq::K_integrable}, \eqref{eq::K_behavior}, and~\eqref{eq::K_lower_bound_Q} and let~$\Omega\subseteq \R^n$ be a Lipschitz domain.
			
Then, there exists~$\Lambda_0>0$ such that for all~$\Lambda\in(0,\Lambda_0]$ the following statement holds true.
		
			Let~$E$ be a~$\left(\Lambda,\frac{2s_1}{n}\right)$-minimal set for~$\P_K$ in~$\Omega$ (in the sense of Definition~\ref{def::almost_minimality}). 
			
			Then, there exists a constant~$c_0\in(0,1)$, depending only on~$n$, $s_1$, and~$\kappa_2$, such that, for any~$x_0\in(\partial E)\cap\Omega$ and~$r\in(0,\min\{\delta/4,\mbox{dist}(x_0,\partial \Omega)\})$, 
			\begin{equation} \label{eq::unif_dens_estimates}
				c_0r^n \leq |E\cap B_r(x_0)| \leq (1-c_0)r^n.
			\end{equation}
		\end{proposition}
		
		For the proof of Proposition~\ref{prop::ude_almost_min}, we employ the following auxiliary results
		in Lemmata~\ref{lemma::iteration} and~\ref{lemma::isoperimetric_estimate}.
		
		\begin{lemma}[Lemma 7.1, \cite{MR1707291}] \label{lemma::iteration}
			Let~$\beta\in(0,1)$, $N>1$, and~$M>0$. Let~$\{x_k\}_k$ be a decreasing sequence in~$\R$ such that
			$$ x_{k+1}^{1-\beta}\leq N^kMx_k.$$
			If~$x_0\leq N^{\frac{1}{\beta}-\frac{1}{\beta^2}}M^{-\frac{1}{\beta}}$, then~$x_k\to0$ as~$k\to+\infty$.
		\end{lemma}
		
		\begin{lemma}[Isoperimetric lower bound for the kernel~$K$] \label{lemma::isoperimetric_estimate}
			Suppose that~$K$ satisfy~\eqref{eq::K_invariance}, \eqref{eq::K_integrable}, \eqref{eq::K_behavior}, and~\eqref{eq::K_lower_bound_Q}. %, \eqref{eq::K_invariance}, and~\eqref{eq::K_integrable}.
			
			Then, there exists a positive constant~$C$, depending on~$n$, $s_1$, $\kappa_1$, and~$\delta$, such that
			\begin{equation} \label{eq::isoperimetric_estimate}
				\P_K(B_r) \geq C
				\begin{cases}
					r^{n-2s_1},\quad &{\mbox{if }}\;0<r\leq\delta/4,\\
					r^{n-1},\quad &{\mbox{if }}\;r>\delta/4.
				\end{cases}
				\quad
			\end{equation}
		\end{lemma}
		
		\begin{proof}
			To show~\eqref{eq::isoperimetric_estimate}, let us start by
			assuming~$0<r\leq\delta/4$. In this case, by~\eqref{eq::K_behavior} and the scaling properties of the kernel~$|x|^{-n-2s_1}$, we have that
			\begin{equation} \label{eq::isoperimetric_small_scale}
				\begin{split}
					&P_K(B_r)=\iint_{B_r\times B_r^c} K(x,y)\,dx\,dy \geq \iint_{B_r\times (B_{2r}\setminus B_{r})} K(x,y)\,dx\,dy \\
					&\qquad\geq \kappa_1 \iint_{B_r\times (B_{2r}\setminus B_{r})} \frac{dx\,dy}{|x-y|^{n+2s_1}}\\
					&\qquad =\kappa_1r^{n-2s_1} \iint_{B_1\times (B_2\setminus B_1)} \frac{dx\,dy}{|x-y|^{n+2s_1}}=: C_1(n,s_1,\kappa_1)r^{n-2s_1}.
				\end{split}
			\end{equation}
			
			Also, if~$r>\delta/4$, let us consider the covering~$\{B_{\delta/8}(x_0)\}_{x_0\in\partial B_r}$ of~$\partial B_r$. By compactness, we extract a finite sub-covering~$\mathcal{B}$. In particular, \begin{equation}\label{ciserve}
			\sharp(\mathcal{B})=c\left(\frac{r}{\delta}\right)^{n-1},\end{equation} for some constant~$c>0$ depending only on~$n$. Thus, we have that
			\begin{equation}\label{y854ghsdjkgu4y}\begin{split}
				&P_K(B_r)=\iint_{B_r\times B_r^c} K(x,y)\,dx\,dy \geq \sum_{B_{\delta/8}(x_0)\in\mathcal{B}} \iint_{(B_r\cap B_{\delta/8}(x_0))\times (B_r^c\cap B_{\delta/8}(x_0))} K(x,y)\,dx\,dy \\
				&\qquad \geq \kappa_1 \sum_{B_{\delta/8}(x_0)\in\mathcal{B}} \iint_{(B_r\cap B_{\delta/8}(x_0))\times (B_r^c\cap B_{\delta/8}(x_0))} \frac{dx\,dy}{|x-y|^{n+2s_1}}\\
				&\qquad \geq \kappa_1\delta^{n-2s_1} \sum_{B_{\delta/8}(x_0)\in\mathcal{B}} \iint_{(B_{8r/\delta}\cap B_1(8x_0/\delta))\times (B_{8r/\delta}^c\cap B_1(8x_0/\delta))} \frac{dx\,dy}{|x-y|^{n+2s_1}} .
			\end{split}\end{equation}
			
We now use the notation
$$x_\delta:=\frac{8x_0}\delta$$
and we claim that
\begin{equation}\label{aggbis23}
(\partial B_{8r/\delta})\cap B_1(x_\delta) \subseteq \left\{ -\frac14<(x-x_\delta)\cdot\frac{x_0}{|x_0|}\le 0 \right\}.
\end{equation}
To check this, we pick~$x\in \partial B_{8r/\delta}\cap B_1(x_\delta)$ and we observe that
\begin{equation*}
1>|x-x_\delta|^2=|x|^2+|x_\delta|^2-2x\cdot x_\delta=
\left(\frac{8r}\delta\right)^2+\left(\frac{8}\delta\right)^2|x_0|^2
-\frac{16 x\cdot x_0}\delta=
2\left(\frac{8r}\delta\right)^2
-\frac{16 x\cdot x_0}\delta
\end{equation*}
and accordingly
\begin{equation}\label{wqkdh82hiunvtymbioSq19wsdu.0}
x\cdot\frac{ x_0}{|x_0|}=
\frac{\delta}{16r}\;
\frac{16 x\cdot x_0}\delta
>\frac{\delta}{16r}\left[
2\left(\frac{8r}\delta\right)^2-1
\right]
=\frac{8r}\delta-\frac{\delta}{16r}.
\end{equation}

Moreover,
\begin{equation}\label{wqkdh82hiunvtymbioSq19wsdu.1}
(x-x_\delta)\cdot\frac{x_0}{|x_0|}=
x\cdot\frac{x_0}{|x_0|}-|x_\delta|=
x\cdot\frac{x_0}{|x_0|}-\frac{8|x_0|}\delta=
x\cdot\frac{x_0}{|x_0|}-\frac{8r}\delta,
\end{equation}
from which we arrive at
\begin{equation}\label{wqkdh82hiunvtymbioSq19wsdu}
(x-x_\delta)\cdot\frac{x_0}{|x_0|}\le
|x|-\frac{8r}\delta=0.
\end{equation}

Also, from~\eqref{wqkdh82hiunvtymbioSq19wsdu.0} and~\eqref{wqkdh82hiunvtymbioSq19wsdu.1},
$$ (x-x_\delta)\cdot\frac{x_0}{|x_0|}>-\frac{\delta}{16r}>-\frac14.$$
Combining this and~\eqref{wqkdh82hiunvtymbioSq19wsdu}, we obtain~\eqref{aggbis23}, as desired.	
		
			It thereby follows from~\eqref{aggbis23} that
			\begin{align*}
				&\iint_{(B_{8r/\delta}\cap B_1(x_\delta))\times (B_{8r/\delta}^c\cap B_1(x_\delta))} \frac{dx\,dy}{|x-y|^{n+2s_1}} \\
				&\qquad\qquad \geq \iint_{\left( \left\{ (x-x_\delta)\cdot\frac{x_0}{|x_0|}>0\right\}\cap B_1(x_\delta)\right)\times \left( \left\{ (y-x_\delta)\cdot\frac{x_0}{|x_0|}<-\frac{1}{4}\right\}\cap B_1(x_\delta)\right)} \frac{dx\,dy}{|x-y|^{n+2s_1}}.
			\end{align*}
			Plugging this information into~\eqref{y854ghsdjkgu4y}, we obtain that
			\begin{equation*}\begin{split}
				&P_K(B_r) \geq \kappa_1\delta^{n-2s_1} 
				\sum_{B_{\delta/8}(x_0)\in\mathcal{B}} \iint_{\left( \left\{ (x-x_\delta)\cdot\frac{x_0}{|x_0|}>0\right\}\cap B_1(x_\delta)\right)\times \left( \left\{ (y-x_\delta)\cdot\frac{x_0}{|x_0|}<-\frac{1}{4}\right\}\cap B_1(x_\delta)\right)} \frac{dx\,dy}{|x-y|^{n+2s_1}}.
			\end{split}\end{equation*}			
			By the translation and rotation invariance of the fractional kernel~$|x|^{-n-2s_1}$ and~\eqref{ciserve}, we infer that
			\begin{equation}\label{eq::isoperimetric_large_scale}
				\begin{split}
					P_K(B_r)&\ge\kappa_1\delta^{n-2s_1} 
				\sum_{B_{\delta/8}(x_0)\in\mathcal{B}}\iint_{\left(B_1\cap\{x_n>0\}\right)\times \left(B_1\cap \left\{ y_n<-\frac14\right\}\right)} \frac{dx\,dy}{|x-y|^{n+2s_1}}\\&
				 \geq c\kappa_1\delta^{1-2s_1}r^{n-1}\iint_{(B_1\cap\{x_n>0\})\times (B_1\cap\{y_n<-1/4\})} \frac{dx\,dy}{|x-y|^{n+2s_1}}\\
					&=:C_2(n,\kappa_1,s_1,\delta)r^{n-1}.
				\end{split}
			\end{equation}
			
We conclude the proof from~\eqref{eq::isoperimetric_small_scale} and~\eqref{eq::isoperimetric_large_scale} by setting~$C:=\min\{C_1,C_2\}$.
		\end{proof}
		
		\begin{proof} [Proof of Proposition~\ref{prop::ude_almost_min}]
			Let us recall that, by definition, if~$E$ is a~$\left(\Lambda,\frac{2s_1}{n}\right)$-minimal set, for some~$\Lambda\geq0$, then
			\begin{equation} \label{eq::almost_min2}
				\P_K(E, \Omega)\leq\P_K(F, \Omega)+\Lambda|E\Delta F|^{1-\frac{2s_1}{n}},
			\end{equation}
			for every~$F$ such that~$ F\setminus \Omega=E\setminus \Omega$.
			
			Let~$x_0\in( \partial E)\cap \Omega$. Up to a translation, we suppose that~$x_0$ coincides with the origin. 
			
			Define~$A_r:=E\cap B_r$, with~$r\in(0,\delta/4)$, where~$\delta$ is as in~\eqref{eq::K_behavior}, and observe
			that~$A_r\subseteq \Omega$. Also, let~$\mu(r):=|A_r|$ and notice that, by the co-area formula, $\mu'(r)=\haus{n-1}(E\cap\partial B_r)$.
			
			Our strategy now is to provide an estimate for~$\mu^{1-\frac{2s_1}{n}}(r)$ in terms of~$\mu(r)$.
			% For this, we set $q=\frac{n-1}{n}$%$q:=\frac{2n}{n-2s_1}$ 
			For this, let us consider~$\rho_r\in(0,r)$ such that~$|A_r|=\omega_n\rho_r^n$. Then, observe that, thanks to Lemma~\ref{lemma::isoperimetric_estimate} and~\cite[Proposition~3.1]{MR3732175},  
			\begin{equation} \label{eq::ude1}
				%\norm{\chi_{A_r}}_{L^q(\R^n)}\leq C\norm{\chi_{A_r}}_{H^{\frac{2s_1}{2}}(\R^n)} = C\left(\mathcal{L}(A_r,A_r^c)\right)^{\frac12},
				\norm{\chi_{A_r}}_{L^{\frac{n}{n-2s_1}}(\R^n)} = |A_r|^{1-\frac{2s_1}{n}} \leq C\P_K(B_{\rho_r}) \leq C\P_K(A_r) = C \mathcal{L}_K(A_r,A_r^c),
			\end{equation}
			up to renaming~$C$, and we stress that~$C$ depends only\footnote{Notice indeed that, once we assume that~$\rho_r<\delta/4$, we can select~$C:=C_1$ given by~\eqref{eq::isoperimetric_small_scale}			
			in Lemma~\ref{lemma::isoperimetric_estimate}. Thus, the constant~$C$ in~\eqref{eq::ude1} is independent of~$\delta$.}
			on~$n$, $s_1$, and~$\kappa_1$.
			
			Now, since
			\begin{equation*}
				\mathcal{L}_K(A_r,A_r^c) = \mathcal{L}_K(A_r,E\cap A_r^c)+\mathcal{L}_K(A_r,E^c),
			\end{equation*}
			it follows from the subsolution property in~\cite[Definition~1.2]{sequoia} that
			\begin{equation} \label{eq::ude2}
				\mathcal{L}_K(A_r,A_r^c)\leq 2\mathcal{L}_K(A_r,E\cap A_r^c) + \Lambda\mu^{1-\frac{2s_1}{n}}(r)\leq 2\mathcal{L}_K(A_r,B_r^c) + \Lambda\mu^{1-\frac{2s_1}{n}}(r).
			\end{equation}
			Moreover, by the co-area formula and~\eqref{eq::K_behavior}, we have that
			\begin{equation*}
				\begin{split}
					&	\mathcal{L}_K(A_r,B_r^c)
					= \iint_{A_r\times B_r^c} K(x,y)dx\,dy \leq \kappa_2 \iint_{A_r\times B_r^c} \frac{dy\,dx}{|x-y|^{n+2s_1}} 
					\leq C\int_{A_r}\left(\int_{r-|x|}^{+\infty}\frac{dz}{z^{2s_1+1}}\right)\,dx \\
					&\qquad\quad\leq C\int_{A_r} \frac{dx}{(r-|x|)^{2s_1}}
					\leq C\int_0^r\frac{\mu'(\rho)}{(r-\rho)^{2s_1}}\, d\rho ,
				\end{split}
			\end{equation*}
			for some~$C>0$, depending on~$n$, $s_1$, and~$\kappa_2$ and possibly changing from line to line.
			
			Plugging this into~\eqref{eq::ude2}, we obtain that
			\begin{equation*} 
				\mathcal{L}_K(A_r,A_r^c)\le C\int_0^r\frac{\mu'(\rho)}{(r-\rho)^{2s_1}}\, d\rho + \Lambda\mu^{1-\frac{2s_1}{n}}(r).
			\end{equation*}
			From the last inequality and~\eqref{eq::ude1} we deduce that
			\begin{equation}\label{eq::ude3}
				\mu^{1-\frac{2s_1}{n}}(r)=\norm{\chi_{A_r}}_{L^{\frac{n}{n-2s_1}}(\R^n)} \leq 
				C \left(\int_0^r\frac{\mu'(\rho)}{(r-\rho)^{2s_1}}\, d\rho + \Lambda\mu^{1-\frac{2s_1}{n}}(r)\right) ,
			\end{equation} up to renaming~$C$.
			
			Furthermore, we assume that~$\Lambda$ is so small that 
			\begin{equation*}
				C\Lambda \leq \frac{1}{2}.
			\end{equation*}
			
			Using this into~\eqref{eq::ude3}, we thus obtain that
			\begin{equation*}
				\mu^{1-\frac{2s_1}{n}} (r)\leq C\int_0^r \frac{\mu'(\rho)}{(r-\rho)^{2s_1}}\, d\rho .
			\end{equation*}
			Hence,	integrating the latter inequality in~$r\in(0,t)$, we deduce that, for all~$t\in(0, \delta/4]$,
			\begin{equation} \label{eq::ude7}
				\begin{split}
					\int_0^t \mu^{1-\frac{2s_1}{n}}(r)\, dr
					&\leq \int_0^t \left(C \int_0^r \frac{\mu'(\rho)}{(r-\rho)^{2s_1}}\, d\rho  \right)\,dr\\
					&= C\int_0^t \left(\mu'(\rho)\int_\rho^t \frac{dr}{(r-\rho)^{2s_1}}\right)\, d\rho \\
					&= C\int_0^t \mu'(\rho) (t-\rho)^{1-2s_1}\,d\rho \\
					&\leq C \mu(t) \,t^{1-2s_1}.
				\end{split}
			\end{equation}
			
			Now, we set
			$$ c_0:=2^{\frac{n}{2s_1}-\frac{n^2}{(2s_1)^2}}(4C)^{-\frac{n}{2s_1}}$$
			and we claim that
			\begin{equation}\label{y9564rzxcvbnasdfghjkqwertyui1234567}
				\mu(t)\ge c_0t^n \quad {\mbox{for all }} t\in(0,\delta/4].
			\end{equation}
			To prove this,
			we argue by contradiction and assume that
			there exists~$t_0\in(0, \delta/4]$ such that
			\begin{equation}\label{wt436y95687rfdcghvdsj}
				\mu(t_0)< c_0t_0^n.\end{equation}
			
			We define~$t_k:=\frac{t_0}{2}+\frac{t_0}{2^{k+1}}$. Then, using~\eqref{eq::ude7}, we have that
			$$ \frac{t_0}{2^{k+2}} \mu^{1-\frac{2s_1}{n}}(t_{k+1}) = (t_k-t_{k+1})  \mu^{1-\frac{2s_1}{n}}(t_{k+1})  \leq \int_{t_{k+1}}^{t_k}  \mu^{1-\frac{2s_1}{n}}(r)\,dr\leq C \mu(t_k) t_k^{1-2s_1}\leq C \mu(t_k) t_0^{1-2s_1} .$$
			Notice that, by continuity, 
			$$\lim_{k\to+\infty}\mu(t_{k})=\mu\left(\frac{t_0}2\right)=|E\cap B_{t_0/2}|>0.$$
			Therefore, using Lemma~\ref{lemma::iteration}
			with~$x_k:=\mu(t_k)$,
			$\beta:=s/n$, $M:=4Ct_0^{-2s_1}$, and~$N:=2$, we find that 
			$$\mu(t_0)>N^{\frac{1}{\beta}-\frac{1}{\beta^2}}M^{-\frac{1}{\beta}}=2^{\frac{n}{2s_1}-\frac{n^2}{(2s_1)^2}}\big(4Ct_0^{-2s_1}\big)^{-\frac{n}{2s_1}}.$$ 
			Thus, thanks to~\eqref{wt436y95687rfdcghvdsj} we deduce that
			$$ c_0t_0^n> \mu(t_0)>
			2^{\frac{n}{2s_1}-\frac{n^2}{(2s_1)^2}}\big(4Ct_0^{-2s_1}\big)^{-\frac{n}{2s_1}}=
			2^{\frac{n}{2s_1}-\frac{n^2}{(2s_1)^2}}(4C)^{-\frac{n}{2s_1}} t_0^n
			= c_0t_0^n ,$$
			which gives the desired contradiction
			and proves~\eqref{y9564rzxcvbnasdfghjkqwertyui1234567}. 
			
			So, \eqref{y9564rzxcvbnasdfghjkqwertyui1234567} yields
			$$ |E\cap B_r|\geq c_0r^n \quad \text{ for all }r\in(0,\delta/4],$$
			proving the first inequality in~\eqref{eq::unif_dens_estimates}.
			
			Moreover, if~$E$ is an almost minimal set, then also~$E^c$ is almost minimal . %(recall Definition~\ref{def::super_sub_sol} and Lemma~\ref{lemma::almost_minimal_subsupersol}),
			Thus, we exploit~\eqref{y9564rzxcvbnasdfghjkqwertyui1234567}
			replacing~$E$ with~$E^c$, obtaining that, for~$r$ sufficiently small,
			$$ |E^c\cap B_r|\geq c_0r^n,$$
			from which we infer the second inequality in~\eqref{eq::unif_dens_estimates}.
		\end{proof}
		
		% ------------------------------------------------------------------------------------------------
		
		\section{Behavior of~$\P_K$ in large domains}
		
		Here, we show that if $\Omega$ is a Lipschitz domain and
		$$ \Omega_R :=\left\{x\in\R^n \mbox{ s.t. }\frac{x}{R}\in\Omega\right\},$$
		then~$\P_K(\Omega_R)$ is bounded by~$R^{n-1}$. 
		
		First, let us consider the space of Lipschitz function in~$B_\rho^{n-1}$, for some~$\rho>0$,
		denoted by~$\mbox{Lip}(B_\rho^{n-1})$ and
		endowed with the norm
		\begin{equation*}
			\norm{f}_{Lip(B_\rho^{n-1})} := \norm{f}_{L^\infty(B_\rho^{n-1})} +  [f]_{Lip(B_\rho^{n-1})} ,
		\end{equation*}
		where
		\begin{equation*}
			[f]_{Lip(B_\rho^{n-1})} := \sup_{\substack{x',y'\in B_\rho^{n-1} \\ x'\neq y'}} \frac{|f(x')-f(y')|}{|x-y'|}.
		\end{equation*}
		
We also recall the following notion of Lipschitz set (see \cite[Definition 2.4.5]{MR3791463}).
		\begin{definition} \label{def::unif_lipschitz}
			We say that an open set $\Omega\subset \R^n$ has Lipschitz boundary (or that $\Omega$ is a Lipschitz set) if there exist $r_\Omega>0$, $L_\Omega>0$, and $\alpha_\Omega>0$ such that for every $x_0\in\partial\Omega$ there exist a rotation $\mathcal{R}_{\theta_{x_0}}$ (of an angle $\theta_{x_0}\in[0,2\pi)$) and a Lipschitz function $\phi_{x_0}\in\mbox{Lip}(B_{r_\Omega}^{n-1})$ such that:
			\begin{enumerate}[i)] 
				\item $[\phi_{x_0}]_{Lip(B_{r_\Omega}^{n-1})} \leq L_\Omega$; 
				\item $\mathcal{R}_{\theta_{x_0}}(\Omega-x_0)\cap\left(B_{r_\Omega}^{n-1}\times(-\alpha_\Omega,\alpha_\Omega)\right) = \left\{(x',x_n) \mbox{ s.t. }x'\in B_{r_\Omega}^{n-1}\mbox{ and }x_n\in(-\alpha_\Omega,\phi_{x_0}(x'))\right\}$.
		%		\item $\partial \left(\mathcal{R}_{\theta_{x_0}}\Omega(-x_0)\right)\cap\left(B_{r_\Omega}^{n-1}\times(-\alpha_\Omega,\alpha_\Omega)\right) = \left\{(x',\phi_{x_0}(x')) \mbox{ s.t. }x'\in B_{r_\Omega}^{n-1}\right\}$.
				%\mbox{for every }x\in\partial\Omega\cap B_{r_\Omega}(x_0) , \mbox{we have } x = \widetilde{x}+\phi_{x_0}(\widetilde{x})\nu_{x_0} \mbox{ with }\widetilde{x}\in\{(x-x_0)\cdot\nu_{x_0}=0\}\cap
			\end{enumerate}
		\end{definition}
		
		The main result of this section is the following:
		\begin{proposition} \label{prop::P_K_upper_bound_BR} 
			There exist~$\eta\in(0,1)$ and~$C=C(n,s_1,s_2,\kappa_2,\delta,\eta,\Omega)>0$ such that
			\begin{equation*} \label{eq::P_K_upper_bound_BR}
				\P_K(\Omega_R) \leq CR^{n-1}, 
			\end{equation*}
		for every $R\geq1$ large enough.
		\end{proposition}
		
		The proof relies on the following one-dimensional observation:
		
		\begin{lemma} \label{lemma::estimate_mid_scale_1D}
			Let~$s\in(1/2,1)$ and~$\delta\in(0,1)$. Then,
			\begin{equation*}
				\int_{R}^{+\infty}\int_{0}^{R}\frac{\chi_{(\delta,+\infty)}(|x-y|)}{|x-y|^{1+2s}}\,dx\,dy \leq \frac{\delta^{1-2s}}{2s-1}.
			\end{equation*}
		\end{lemma}
		
		\begin{proof}
			The claim follows from a direct computation that we provide here for the facility of the reader.
			
		We observe that
			\begin{align*}
\int_{R}^{+\infty}\int_{0}^{R}\frac{\chi_{(\delta,+\infty)}(|x-y|)}{|x-y|^{1+2s}}\,dx\,dy &= \int_{R}^{+\infty}\int_{0}^{\min\{R,x-\delta\}}\frac{dx\,dy}{(x-y)^{1+2s}} \\
				&= \frac{1}{2s}\int_{R}^{+\infty} \left[\frac{1}{(x-y)^{2s}}\right]_0^{\min\{R,x-\delta\}}\,dx.
			\end{align*}
			Since~$x-\delta=\min\{R,x-\delta\}$ if and only if~$x\in(R,R+\delta]$, we obtain that
			\begin{align*}
				&\int_{R}^{+\infty}\int_{0}^{R}\frac{\chi_{(\delta,+\infty)}(|x-y|)}{|x-y|^{1+2s}}\,dx\,dy = \frac{1}{2s}\left(\int_{R}^{R+\delta} \frac{dx}{\delta^{2s}} +
				\int_{R+\delta}^{+\infty}\frac{dx}{(x-R)^{2s}}- \int_R^{+\infty} \frac{dx}{x^{2s}}\right)\\&\qquad\le  \frac{1}{2s}\left(\int_{R}^{R+\delta} \frac{dx}{\delta^{2s}} +
				\int_{R+\delta}^{+\infty}\frac{dx}{(x-R)^{2s}}\right)
			=\frac1{2s}\left(\delta^{1-2s}+\frac{\delta^{1-2s}}{2s-1}
				\right)
	= \frac{\delta^{1-2s}}{2s-1}.\qedhere
			\end{align*}
		\end{proof}
		
		Another important tool for proving Proposition~\ref{prop::P_K_upper_bound_BR} is an adaptation of~\cite[Lemma~10]{MR2782803} for Lipschitz domains. For this, let us define, for any $\rho>0$ and $\lambda\geq1$, the sets
		\begin{equation}\label{C2BIS}
			\begin{split}
				&C_{\rho,\lambda} := \{(x',x_n)\in\R^n \mbox{ s.t. }|x'|<\rho,\ |x_n|<\lambda\rho\} = B_\rho^{n-1}\times(-\lambda\rho,\lambda\rho),\\
				&C_{\rho,\lambda}^+ := \{(x',x_n)\in\R^n \mbox{ s.t. }|x'|<\rho,\ 0<x_n<\lambda\rho\} = B_\rho^{n-1}\times(0,\lambda\rho),\\
				&C_{\rho,\lambda}^- := \{(x',x_n)\in\R^n \mbox{ s.t. }|x'|<\rho,\ -\lambda\rho<x_n<0\} = B_\rho^{n-1}\times(-\lambda\rho,0),\\
				\mbox{and}\quad& C_{\rho}^0 := C_{\rho,\lambda}\cap\{x_n=0\}= B_\rho^{n-1}\times\{0\}.
			\end{split}
		\end{equation}
For any $\phi\in \mbox{Lip}(B_\rho^{n-1})$, let also
			\begin{align*}
				&C_{\rho\lambda}^+(\phi) := \{(x',x_n)\in\R^n \mbox{ s.t. }|x'|<\rho,\ \phi(x')<x_n<\lambda\rho\} ,\\
				\mbox{and}\quad& C_{\rho,\lambda}^-(\phi) := \{(x',x_n)\in\R^n \mbox{ s.t. }|x'|<\rho,\ -\rho\lambda<x_n<\phi(x')\} .
			\end{align*}
		
		\begin{lemma} \label{lemma::lip_transform}
			Let $s\in(0,1)$, $\rho>0$, and $\lambda\geq1$, and let $\phi\in \mbox{Lip}(B_\rho^{n-1})$ be such that $\phi(0)=0$ and~$\nabla\phi(0)=0$.
			
				If 
			\begin{equation} \label{eq::norm_bound}
				 [\phi]_{Lip(B_\rho^{n-1})} \leq \frac{\lambda}{10},
			\end{equation}
			then there exists a transformation $\Psi:\R^n\to\R^n$ such that $\Psi(C_\rho^\pm(\phi))=C_\rho^\pm$ and
			\begin{equation*}
				\iint_{C_{\rho,\lambda}^+(\phi) \times C_{\rho,\lambda}^-(\phi)} \frac{dx\,dy}{|x-y|^{n+2s}} \leq c \iint_{C_{\rho,\lambda}^+ \times C_{\rho,\lambda}^-} \frac{dx\,dy}{|x-y|^{n+2s}},
			\end{equation*}
			for some positive constant $c$ independent of $\rho$.
		\end{lemma}
		
		\begin{comment}
			\begin{equation*}
				\eta_\rho(\phi) := [\phi]_{Lip(B_\rho^{n-1})}
			\end{equation*}
		\end{comment}
		\begin{proof}
			Since $\phi$ is Lipschitz and $\phi(0)=0$, we have that
			\begin{equation}\label{C3BIS}
				\sup_{|x'|<\rho}|\phi(x')| \leq [\phi]_{Lip(B_\rho^{n-1})}\rho.
			\end{equation}
			
We recall~\eqref{eq::norm_bound} and we define the function $\psi\in\cont^\infty(\R^n,[0,1])$ such that
			\begin{equation*}
				\psi(x) :=
				\begin{cases}
					1,\quad&\mbox{if }\displaystyle|x_n-\phi(x')|\leq \frac{\lambda-[\phi]_{Lip(B_\rho^{n-1})}}{8}\rho,\\
					0,\quad&\mbox{if }\displaystyle|x_n-\phi(x')|\geq \big(\lambda-[\phi]_{Lip(B_\rho^{n-1})}\big) \rho,
				\end{cases}
			\end{equation*}
			and
			\begin{equation}\label{C3TER}
				|\nabla\psi|\leq \frac{2}{(\lambda-[\phi]_{Lip(B_\rho^{n-1})}) \rho}.
			\end{equation}
			
			Notice that, for any $|x'|\leq\rho$, by~\eqref{C3BIS} we have that
			\begin{equation*}
				|\phi(x')\pm \lambda\rho| \geq \lambda\rho - |\phi(x')| \geq (\lambda-[\phi]_{Lip(B_\rho^{n-1})}) \rho
			\end{equation*}
			and hence
			\begin{equation} \label{eq::psi_flattening1}
				\psi(x',\pm \lambda\rho) = 0.
			\end{equation}
			On the other hand, for any $|x'|\leq\rho$, we have that
			\begin{equation} \label{eq::psi_flattening2}
				\psi(x',\phi(x')) = 1.
			\end{equation}
			
			Now, we consider the transformation $\Psi:\R^n\to\R^n$ defined as
			\begin{equation*}
				\Psi(x) := (x',x_n-\psi(x)\phi(x')) = x-\psi(x)\phi(x')e_n,
			\end{equation*}
			where $e_n:=(0,\dots,0,1)$ is the usual unit vector in the $n^{th}$-direction.
			
			By Rademacher's Theorem, $\phi$ is differentiable almost everywhere, and, if $\widetilde{x}:=\Psi(x)$ and $\widetilde{y}:=\Psi(y)$, we have (see also~\cite[Theorem 2, Sec. 3.3.3]{MR3409135})
			\begin{equation*}
				\begin{split}
					d\widetilde{x}\,d\widetilde{y} &= |1-\partial_n\psi(x)\phi(x')|\,|1-\partial_n\psi(y)\phi(y')|dx\,dy \\
					& \geq \left(1- \left| \partial_n\psi(x)\phi(x')+ \partial_n\psi(y)\phi(y') - \partial_n\psi(x)\phi(x')\partial_n\psi(y)\phi(y')\right|\right) dx\,dy \\
					&\geq (1-c_1) dx\,dy,
				\end{split}
			\end{equation*}
			with
			\begin{equation} \label{eq::c_1_estimate}
						c_1:= \frac{4[\phi]_{Lip(B_\rho^{n-1})}^2}{(\lambda-[\phi]_{Lip(B_\rho^{n-1})})^2}+\frac{4[\phi]_{Lip(B_\rho^{n-1})}}{\lambda-[\phi]_{Lip(B_\rho^{n-1})}} \leq \frac{1}{16}+\frac{1}{2}< 1,
			\end{equation}
			where we have employed~\eqref{eq::norm_bound}, \eqref{C3BIS} and~\eqref{C3TER}.
			It thereby follows that
			\begin{equation} \label{eq::diff_bound}
				dx\,dy \leq \frac{d\widetilde{x}\,d\widetilde{y}}{1-c_1} .
			\end{equation}
			Besides, recalling~\eqref{eq::psi_flattening1} and~\eqref{eq::psi_flattening2}, we deduce that 
			\begin{equation} \label{eq::Psi_transform}
				\Psi(C_{\rho,\lambda}^\pm(\phi))=C_{\rho,\lambda}^\pm .
			\end{equation}
			
			Moreover, for almost every $x$, $y\in C_\rho$, using again~\eqref{C3BIS} and~\eqref{C3TER} we have that
			\begin{equation}\label{eq::s_kernel_bound} 
				\begin{split}
					|\widetilde{x}-\widetilde{y}| &= |x-\psi(x)\phi(x')e_n-y+\psi(y)\phi(y')e_n| \\
					& \leq |x-y| + |\psi(x)|\,|\phi(x')-\phi(y')|+|\phi(y')|\,|\psi(x)-\psi(y)|\\
					& \leq |x-y| + [\phi]_{Lip(B_\rho^{n-1})}|x'-y'| + [\phi]_{Lip(B_\rho^{n-1})}\rho \frac{2}{(\lambda-[\phi]_{Lip(B_\rho^{n-1})})\rho} |x-y|\\
					&\leq c_2 |x-y|
				\end{split}
			\end{equation}
			for some constant $c_2>0$ depending on~$\lambda$ but independent of~$\rho$.
			
Employing the changes of variable $\widetilde{x}:=\Psi(x)$ and $\widetilde{y}:=\Psi(y)$ together with~\eqref{eq::diff_bound}, \eqref{eq::Psi_transform}, and~\eqref{eq::s_kernel_bound}, we conclude that
			\begin{equation*}
				\iint_{C_{\rho,\lambda}^+(\phi) \times C_{\rho,\lambda}^-(\phi)} \frac{dx\,dy}{|x-y|^{n+2s}} \leq \frac{c_3}{1-c_1} \iint_{C_{\rho,\lambda}^+ \times C_{\rho,\lambda}^-} \frac{d\widetilde{x}\,d\widetilde{y}}{|\widetilde{x}-\widetilde{y}|^{n+2s}},
			\end{equation*}
			% =\iint_{C_\rho^+ \times C_\rho^-} \frac{|\nabla\Psi^{-1}(\widetilde{x})||\nabla\Psi^{-1}(\widetilde{y})|}{|\Psi^{-1}(\widetilde{x})-\Psi^{-1}(\widetilde{y})|^{n+2s}}d\widetilde{x}d\widetilde{y}
		for some~$c_3>0$ independent of~$\rho$, showing the desired result.
		\end{proof}

\begin{proof}[Proof of Proposition~\ref{prop::P_K_upper_bound_BR}]	
	Let~$R>0$ and take~$\eta\in(0,1)$ such that 
	$$\delta<\eta R< 6\sqrt{n}\eta R < R.$$
	Since $\Omega$ is bounded, there exists~$\zeta>0$ such that $\Omega\subseteq B_\zeta$. Up to scaling, we can suppose that $\zeta=1$.
	%, otherwise we consider $\widetilde{\eta}:=\eta\zeta\in(0,1)$. 
	
	Moreover, let $r_\Omega$, $L_\Omega$, and $\alpha_\Omega$ be as in Definition~\ref{def::unif_lipschitz} and define 
	\begin{equation*}
		\lambda := \max\{10 L_\Omega,1\}.
	\end{equation*}
	Wa assume that $R$ is so large that
	\begin{equation*}
		3\sqrt{n}\delta < R r_{\Omega}\quad\mbox{and}\quad 3\sqrt{n}\delta\lambda< R \alpha_{\Omega}
	\end{equation*}
	and that $\eta$ is so small that 
	\begin{equation*}
		3\sqrt{n}\eta <r_\Omega\quad\mbox{and}\quad 3\sqrt{n}\eta\lambda< \alpha_{\Omega}.
	\end{equation*}

	\begin{comment}
		Moreover, since $\Omega$ has Lipschitz boundary, for every $x_0\in\partial \Omega$, there exist $\nu_{x_0}\in \S^{n-1}$, $\phi_{x_0}\in\mbox{Lip}(\{(x-x_0)\cdot\nu_{x_0}=0\})$, and $r_{x_0}>0$ such that
		\begin{equation*}
			\mbox{for every }x\in\partial\Omega\cap B_{r_{x_0}}(x_0) , \mbox{we have } x = \widetilde{x}+\phi(\widetilde{x})\nu_{x_0} \mbox{ with }\widetilde{x}\in\{(x-x_0)\cdot\nu_{x_0}=0\}\cap B_{r_{x_0}}(x_0).
		\end{equation*}
		Assume also that
		\begin{equation*}
			3\sqrt{n}\delta < R\inf_{x_0\in\partial\Omega}r_{x_0}.
		\end{equation*}
		
		Observe that
		\begin{equation*}
			\sup_{x_0\in\partial\Omega} \norm{\phi_{x_0}}_{Lip} <+\infty
		\end{equation*} 
		by compactness, and let 
		\begin{equation*}
			\lambda := \max\{10 \sup_{x_0\in\partial\Omega} \norm{\phi_{x_0}}_{Lip},1\}.
		\end{equation*}
	\end{comment}

			Then, we have
			\begin{equation} \label{eq::estimate_P_K_B_R}
				\begin{split}
					&\P_K(\Omega_R) \\
					=& \iint_{\substack{\Omega_R\times \Omega_R^c \\ \{|x-y|<\delta\}}} K(x,y)\,dx\,dy +  \iint_{\substack{\Omega_R\times \Omega_R^c \\ \{\delta\leq |x-y|<\eta R\}}} K(x,y)\,dx\,dy +  \iint_{\substack{\Omega_R\times \Omega_R^c \\ \{|x-y|\ge\eta R\}}} K(x,y)\,dx\,dy\\
					=&:\mathcal{I}_1+\mathcal{I}_2+\mathcal{I}_3.
				\end{split}
			\end{equation}
We aim to provide estimates for each~$\mathcal{I}_j$, with~$j=1,2,3$.
			
			First, notice that, by compactness, there exists a finite covering of~$\partial \Omega_R$ made of balls of radius~$\delta$,
			$$ \mathcal{B}_\delta := \{B_\delta(x_j) \mbox{ s.t. }x_j\in\partial \Omega_R\}_{j=1}^{N_\delta},$$
			where~$N_\delta:=\sharp(\mathcal{B}_\delta) \leq c_n\delta^{1-n}R^{n-1}$, for some~$c_n>0$ depending only on the dimension~$n$. 
			
			Moreover, 
			\begin{align*}
				(\Omega_R\times \Omega_R^c) \cap \{|x-y|<\delta\} &\subseteq \bigcup_{i,j=0}^{N_\delta}
				\Big((\Omega_R\cap
				B_{\sqrt{n}\delta}(x_i))\times (\Omega_R^c\cap B_{\sqrt{n}\delta}(x_j))\Big) \cap \{|x-y|<\delta\}\\
				& \subseteq \bigcup_{i=0}^{N_\delta}\Big((\Omega_R\cap
				B_{3\sqrt{n}\delta}(x_i))\times(\Omega_R^c\cap B_{3\sqrt{n}\delta}(x_i))\Big).
			\end{align*}
			Besides, since $\Omega_R$ has Lipschitz boundary, for every $i$, let $\mathcal{R}_{\theta_i}:=\mathcal{R}_{\theta_{x_i}}$ be a rotation and let $\phi_i:=\phi_{x_i}$ be a Lipschitz function as in Definition~\ref{def::unif_lipschitz}, so that
			\begin{equation*}
				\partial(\mathcal{R}_{\theta_i}(\Omega_R-x_i))\cap C_{3\sqrt{n}\delta,\lambda} = \left\{(x',\phi_i(x')) \mbox{ s.t. } x'\in B_{3\sqrt{n}\delta}^{n-1}\right\}.
				%\mbox{graph}(\phi_i)\cap C_{3\sqrt{n}\delta,\lambda}
			\end{equation*}
			
			Thus, recalling also~\eqref{eq::K_invariance} and~\eqref{eq::K_behavior}, and taking advantage of the rotational symmetries, we infer that 
			\begin{equation} \label{eq::estimate_I1}
				\begin{split}
					\mathcal{I}_1 &= \iint_{\substack{\Omega_R\times \Omega_R^c \\ \{|x-y|<\delta\}}} K(x,y)\,dx\,dy \leq \sum_{i=0}^{N_\delta}\kappa_2 \iint_{(\Omega_R\cap B_{3\sqrt{n}\delta}(x_i))\times (\Omega_R^c\cap B_{3\sqrt{n}\delta}(x_i))} \frac{dx\,dy}{|x-y|^{n+2s_1}}\\
					& \leq \sum_{i=0}^{N_\delta}\kappa_2 \iint_{(\mathcal{R}_{\theta_i}(\Omega_R-x_i)\cap B_{3\sqrt{n}\delta})\times (\mathcal{R}_{\theta_i}(\Omega_R^c-x_i)\cap B_{3\sqrt{n}\delta})} \frac{dx\,dy}{|x-y|^{n+2s_1}}\\
					& \leq \sum_{i=0}^{N_\delta}\kappa_2 \iint_{(\mathcal{R}_{\theta_i}(\Omega_R-x_i)\cap C_{3\sqrt{n}\delta,\lambda})\times (\mathcal{R}_{\theta_i}(\Omega_R^c-x_i)\cap C_{3\sqrt{n}\delta,\lambda})} \frac{dx\,dy}{|x-y|^{n+2s_1}}.
				\end{split}
			\end{equation}
			
			\begin{comment}
			&\kappa_2N_\delta \iint_{(B_R(-Re_n)\cap B_{3\sqrt{n}\delta})\times (B_R^c(-Re_n)\cap B_{3\sqrt{n}\delta})} \frac{dx\,dy}{|x-y|^{n+2s_1}}\\
			&\leq \kappa_2N_\delta \iint_{(B_R(-Re_n)\cap C_{3\sqrt{n}\delta})\times (B_R^c(-Re_n)\cap C_{3\sqrt{n}\delta})} \frac{dx\,dy}{|x-y|^{n+2s_1}},
			where, for any~$\rho>0$, 
			\begin{equation}%\label{C2BIS}
				C_\rho := \{(x',x_n)\in\R^n \mbox{ s.t. }|x'|<\rho,\ |x_n|<\rho\} = B_\rho^{n-1}\times(-\rho,\rho).
			\end{equation}
			
			Now, observe that the angle at the center that insists on the arch~$\partial B_R(-Re_n)\cap C_{3\sqrt{n}\delta}$ is of order~$1/R$. Hence, for any~$R$ big enough, there exists a transformation~$\Phi$ close to the identity such that
			\end{comment}
			
			Now, thanks to Lemma~\ref{lemma::lip_transform}, for every $i$, there exists a transformation~${\Phi_i:\R^n\to\R^n}$ such that
			\begin{align*}
				&\Phi_i(\mathcal{R}_{\theta_i}(\Omega_R-x_i)\cap C_{3\sqrt{n}\delta,\lambda}) = \{x_n<0\}\cap C_{3\sqrt{n}\delta,\lambda}\\
				\mbox{and}\quad&\Phi_i(\mathcal{R}_{\theta_i}(\Omega_R^c-x_i)\cap C_{3\sqrt{n}\delta,\lambda}) = \{x_n>0\}\cap C_{3\sqrt{n}\delta,\lambda}.
			\end{align*}
			and
			\begin{equation} \label{eq::apply_flattening}
				\begin{split}
					&\iint_{(\mathcal{R}_{\theta_i}(\Omega_R-x_i)\cap C_{3\sqrt{n}\delta,\lambda})\times (\mathcal{R}_{\theta_i}(\Omega_R^c-x_i)\cap C_{3\sqrt{n}\delta,\lambda})} \frac{dx\,dy}{|x-y|^{n+2s_1}} \\
					&\qquad\qquad\leq c_{1,i} \iint_{(\{x_n<0\}\cap C_{3\sqrt{n}\delta})\times (\{x_n>0\}\cap C_{3\sqrt{n}\delta})} \frac{dx\,dy}{|x-y|^{n+2s_1}} ,
				\end{split}
			\end{equation}
			for some positive constant~$c_{1,i}$ depending only on~$n$, $s_1$, and~$\Phi_i$. %, and possibly changing at every step. 
			
			Moreover, recalling~\eqref{eq::c_1_estimate} and~\eqref{eq::s_kernel_bound} from the proof of Lemma~\ref{lemma::lip_transform}, we infer that, for every $i$,
			\begin{equation*} 
			c_{1,i} \leq \left(1+[\phi_i]_{Lip(B_{r_\Omega}^{n-1})} +\frac{2[\phi_i]_{Lip(B_{r_\Omega}^{n-1})} }{\lambda-[\phi_i]_{Lip(B_{r_\Omega}^{n-1})} }\right)^{-(n+2s_1)}\frac{1}{1-\frac{9}{16}} \leq \frac{16}{7}.
			\end{equation*}
Let us set
			\begin{equation} \label{eq::estimate_c_1_i}
				c_1 := \frac{16}{7} \iint_{(\{x_n<0\}\cap C_{3\sqrt{n}\delta})\times (\{x_n>0\}\cap C_{3\sqrt{n}\delta})} \frac{dx\,dy}{|x-y|^{n+2s_1}}<+\infty,
			\end{equation}
			which is independent of $i$.
			  
			 Thus, plugging~\eqref{eq::apply_flattening} and~\eqref{eq::estimate_c_1_i} into~\eqref{eq::estimate_I1}, we arrive at
			\begin{equation}
				\mathcal{I}_1 \leq \sum_{i=0}^{N_\delta}\kappa_2 c_1^i \leq \kappa_2c_1 N_\delta \leq c_1R^{n-1},
			\end{equation}
			up to renaming~$c_1>0$, now also depending on~$\kappa_2$ and~$\delta$.
			
			Let us now focus on estimating~$\mathcal{I}_2$. To do this, we cover~$\partial \Omega_R$ with balls of radius~$\eta R$. By compactness, we find a finite covering
			$$ \mathcal{B}_{\eta R} := \{B_{\eta R}(x_i) \mbox{ s.t. }x_i\in \partial \Omega_R\}_{i=0}^{N_{\eta R}},$$
			with~$N_{\eta R} \leq c_{n,\Omega} \eta^{1-n}$ (independent of~$R$), and
			$$ (\Omega_R\times \Omega_R^c) \cap \{|x-y|<\eta R\} \subseteq \bigcup_{i=0}^{N_{\eta R}}\left( (\Omega_R\cap B_{3\sqrt{n}\eta R})\times( \Omega_R^c \cap B_{3\sqrt{n}\eta R})\right)
			\cap \{|x-y|<\eta R\} .$$
			Therefore,
			\begin{equation} \label{eq::estimate_I2_1}
				\begin{split}
					&\mathcal{I}_2 \leq \sum_{i=0}^{N_{\eta R}} \kappa_2\iint_{\substack{(\Omega_R\cap B_{3\sqrt{n}\eta R}(x_i))\times (\Omega_R^c\cap B_{3\sqrt{n}\eta R}(x_i)) \\ \{\delta\leq|x-y|<\eta R\}}} \frac{dx\,dy}{|x-y|^{n+2s_2}}\\
					&\qquad\leq \sum_{i=0}^{N_{\eta R}} \kappa_2\iint_{\substack{(\mathcal{R}_{\theta_i}(\Omega_R-x_i)\cap B_{3\sqrt{n}\eta R})\times (\mathcal{R}_{\theta_i}(\Omega_R^c-x_i)\cap B_{3\sqrt{n}\eta R}) \\ \{\delta\leq|x-y|<\eta R\}}} \frac{dx\,dy}{|x-y|^{n+2s_2}}\\
					&\qquad\leq \sum_{i=0}^{N_{\eta R}} \kappa_2\iint_{\substack{(\mathcal{R}_{\theta_i}(\Omega_R-x_i)\cap C_{3\sqrt{n}\eta R,\lambda})\times (\mathcal{R}_{\theta_i}(\Omega_R^c-x_i)\cap C_{3\sqrt{n}\eta R,\lambda}) \\ \{\delta\leq|x-y|<\eta R\}}} \frac{dx\,dy}{|x-y|^{n+2s_2}}\\
				\end{split}
			\end{equation}
			
			Thus, in the same spirit as above, recalling the notation in~\eqref{C2BIS},	we find a transformation~${\Psi_i:\R^n\to\R^n}$ such that
			\begin{align*}
				&\Psi_i(\mathcal{R}_{\theta_i}(\Omega_R-x_i)\cap C_{3\sqrt{n}\eta R,\lambda}) = \{x_n<0\}\cap C_{3\sqrt{n}\eta R,\lambda} = B_{3\sqrt{n}\eta R}^{n-1}\times (-3\sqrt{n}\eta \lambda R,0)\\
				\mbox{and}\quad& \Psi_i(\mathcal{R}_{\theta_i}(\Omega_R^c-x_i)\cap C_{3\sqrt{n}\eta R,\lambda}) = \{x_n>0\}\cap C_{3\sqrt{n}\eta R,\lambda} = B_{3\sqrt{n}\eta R}^{n-1}\times (0,3\sqrt{n}\eta \lambda R).
			\end{align*}
			
			For the sake of simplicity, let us call 
			\begin{equation}\label{vbcnmxwieoy8t48ouegjk}
				r:=3\sqrt{n}\eta R.
			\end{equation} 
			Then, using again Lemma~\ref{lemma::lip_transform} and setting $c_2:=\max_i c_{2,i}$ (which does not depend on $R$), we have
			\begin{equation}\label{C4BIS}\begin{split}
				&\iint_{\substack{(\mathcal{R}_{\theta_i}(\Omega_R-x_i)\cap B_{r})\times (\mathcal{R}_{\theta_i}(\Omega_R^c-x_i)\cap B_{r}) \\ \{\delta\leq|x-y|<\eta R\} }} \frac{dx\,dy}{|x-y|^{n+2s_2}} \\
				\leq &c_{2,i} \iint_{\substack{(B_r^{n-1}\times (-\lambda r,0))\times (B_r^{n-1}\times(0,\lambda r)) \\ \{\delta\leq|x-y|<\eta R\} }} \frac{dx\,dy}{|x-y|^{n+2s_2}}\\
				\leq &c_2 \iint_{\substack{(B_r^{n-1}\times (-\lambda r,0))\times (B_r^{n-1}\times(0,\lambda r)) \\ \{\delta\leq|x-y|<\eta R\} }} \frac{dx\,dy}{(|x'-y'|^2+|x_n-y_n|^2)^{\frac{n+2s_2}{2}}},
			\end{split}\end{equation}
where we have used the notation~$x=(x',x_n)\in\R^{n-1}\times\R$.
			
Adopting the change of variable~$z':=\frac{y'-x'}{|x_n-z_n|}$ and~$z_n:=y_n$, we obtain
		\begin{eqnarray*}
			&& \iint_{\substack{ (B_r^{n-1}\times (-\lambda r,0))\times (B_r^{n-1}\times(0,\lambda r)) \\ \{\delta\leq|x-y|<\eta R\} }} \frac{dx\,dy}{(|x'-y'|^2+|x_n-y_n|^2)^{\frac{n+2s_2}{2}}} \\
			&&\qquad\le \int_{B_r^{n-1}}\int_{ \R^{n-1} } \int_{-\lambda r}^0 \int_0^{\lambda r}  \chi_{(\delta,\eta R)} \left(|x_n-z_n|\sqrt{1+|z'|^2} \right) \frac{|x_n-z_n|^{-1-2s_2}}{ (1+|z'|^2)^{\frac{n+2s_2}{2}}} \,dx'\,dz'\,dx_n\,dz_n\\
			&&\qquad\le \int_{B_r^{n-1}}\int_{ \R^{n-1} } \int_{-\lambda r}^0 \int_0^{\lambda r}  \chi_{(\delta,+\infty)} \left(|x_n-z_n|\sqrt{1+|z'|^2} \right) \frac{|x_n-z_n|^{-1-2s_2}}{ (1+|z'|^2)^{\frac{n+2s_2}{2}}} \,dx'\,dz'\,dx_n\,dz_n\\
			&&\qquad =\int_{B_r^{n-1}}\int_{ \R^{n-1} } \int_{-\lambda r}^0 \int_0^{\lambda r}  \chi_{\left(\delta/\sqrt{1+|z'|^2},+\infty\right)} (|x_n-z_n|) \frac{|x_n-z_n|^{-1-2s_2}}{ (1+|z'|^2)^{\frac{n+2s_2}{2}}} \,dx'\,dz'\,dx_n\,dz_n,
			\end{eqnarray*}
		where the notation~$B_r^{n-1}$ stands for the ball of radius~$r$ in~$\R^{n-1}$.
		
Now, thanks to Lemma~\ref{lemma::estimate_mid_scale_1D}, we have that
$$	\int_{-\lambda r}^0 \int_0^{\lambda r}  \frac{\chi_{\left(\delta/\sqrt{1+|z'|^2},+\infty\right)} (|x_n-z_n|)}{|x_n-z_n|^{1+2s_2}}\,dx_n\,dz_n
	\le \frac{1}{2s_2-1} \left(\frac\delta{\sqrt{1+|z'|^2}}\right)^{1-2s_2}. $$
Therefore,
\begin{eqnarray*}
&& \iint_{\substack{ (B_r^{n-1}\times (-\lambda r,0))\times (B_r^{n-1}\times(0,\lambda r)) \\ \{\delta\leq|x-y|<\eta R\} }}
\frac{dx\,dy}{(|x'-y'|^2+|x_n-y_n|^2)^{\frac{n+2s_2}{2}}} \\
&&\qquad \le\frac{\delta^{1-2s_2}}{2s_2-1}
\int_{B_r^{n-1}}\int_{ \R^{n-1} }  
\frac{dx'\,dz'}{ (1+|z'|^2)^{\frac{n+1}{2}}}\le c_2 r^{n-1}\le c_2 R^{n-1},
\end{eqnarray*} up to changing~$c_2>0$ at every step,
where we used~\eqref{vbcnmxwieoy8t48ouegjk} in the last inequality.
	
Plugging this information into~\eqref{C4BIS}, we conclude that
 		\begin{equation*}
			\iint_{\substack{(\mathcal{R}_{\theta_i}(\Omega_R-x_i)\cap B_{r})\times (\mathcal{R}_{\theta_i}(\Omega_R^c-x_i)\cap B_{r}) \\ \{\delta\leq|x-y|<\eta R\} }} \frac{dx\,dy}{|x-y|^{n+2s_2}} 
			\leq c_2 R^{n-1}.
		\end{equation*}
It follows from this and~\eqref{eq::estimate_I2_1} that
			\begin{equation} \label{eq::estimate_I2_3}
				\mathcal{I}_2 \leq \sum_{i=0}^{N_{\eta R}} \kappa_2 c_2 R^{n-1} = N_{\eta R} \kappa_2 c_2 R^{n-1} \leq c_{n,\Omega} \eta^{1-n} \kappa_2 c_2 R^{n-1} \leq c_2R^{n-1} ,
			\end{equation}
			up to renaming $c_2$.
			
			At last, we exhibit an estimate for~$\mathcal{I}_3$. % Since $\Omega$ is bounded, there exists a constant $\zeta>0$ such that $\Omega\subseteq B_\zeta$
			By virtue of~\eqref{eq::K_behavior}, we have that
			\begin{equation}\label{C6BIS}\begin{split}
				\mathcal{I}_3 &\leq \kappa_2\iint_{\substack{\Omega_R\times \Omega_R^c \\ \{|x-y|\geq\eta R\}}} \frac{dx\,dy}{|x-y|^{n+2s_2}} \\
				&= \kappa_2\iint_{\substack{\Omega_R\times (B_{(1+\eta)R}\setminus \Omega_R) \\ \{|x-y|\geq\eta R\}}} \frac{dx\,dy}{|x-y|^{n+2s_2}} + \kappa_2\iint_{\substack{\Omega_R\times B_{(1+\eta)R}^c \\ \{|x-y|\geq\eta R\}}} \frac{dx\,dy}{|x-y|^{n+2s_2}}.
			\end{split}\end{equation}
Notice that
			\begin{equation} \label{eq::estimate_I3_1}
				\iint_{\substack{\Omega_R\times (B_{(1+\eta)R}\setminus \Omega_R) \\ \{|x-y|\geq\eta R\}}} \frac{dx\,dy}{|x-y|^{n+2s_2}} \leq \frac{|\Omega_R||B_{(1+\eta)R}\setminus \Omega_R|}{(\eta R)^{n+2s_2}} \leq c_3R^{n-2s_2} ,
			\end{equation}
			for some positive~$c_3=c_3(n,s_2,\eta)$.
			
			Additionally, by the change of variable~$z:=y-x$,
			\begin{equation*}
\iint_{\substack{\Omega_R\times B_{(1+\eta)R}^c \\ \{|x-y|\geq\eta R\}}} \frac{dx\,dy}{|x-y|^{n+2s_2}} \leq
\iint_{ \Omega_R\times B_{\eta R}^c } \frac{dx\,dz}{|z|^{n+2s_2}}
\leq c_3 R^{n-2s_2}.
			\end{equation*}	
			This inequality and~\eqref{eq::estimate_I3_1}, together with~\eqref{C6BIS}, yield that
			\begin{equation} \label{eq::estimate_I3_3}
				\mathcal{I}_3 \leq \kappa_2c_3R^{n-2s_2}.
			\end{equation}
			
As a result of~\eqref{eq::estimate_P_K_B_R}, \eqref{eq::estimate_I1}, \eqref{eq::estimate_I2_3}, and~\eqref{eq::estimate_I3_3}, we obtain that
			\begin{equation*}
				\P_K(\Omega_R) \leq c_1R^{n-1}+c_2R^{n-1}+c_3R^{n-2s_2},
			\end{equation*}
concluding the proof.
		\end{proof}
		
		% ------------------------------------------------------------------------------------------------
		
		\section{Well-posedness of the Plateau problem for the energy~$\F$} \label{sec::good_plateau_problem}
		
		In this section, we use a Sobolev embedding and a simple covering argument to show that the minimization problem for the energy functional~$\F$ with fixed external datum is well-posed. 
		To do this, let us consider a Lipschitz domain~$\Omega\subseteq\R^n$ and a set~$\widetilde{E}\subseteq\R^n$ of finite~$K$-perimeter. Suppose~$K$ satisfies~\eqref{eq::K_invariance}, \eqref{eq::K_integrable}, \eqref{eq::K_behavior}, and~\eqref{eq::K_lower_bound_Q}. Let also~$g\in L^\infty(\R^n)$ be a~$\Z^n$-periodic function satisfying~\eqref{eq::g_zero_avg}. 
		
		We	recall that the functional~$\F$ is defined in~\eqref{1.4BIS} as
		\begin{equation*}
			\F(E,\Omega) := \P_K(E,\Omega)+ \int_{E\cap \mathcal{Q}(\Omega)_1}g(x)\,dx
		\end{equation*}
	and we investigate the problem:
		\begin{equation}\label{eq::plateau_F}
			\begin{split}
				&\mbox{find~$E\subseteq\R^n$ such that~$E\setminus\Omega=\widetilde{E}\setminus \Omega$}\\
				&\mbox{and}\quad\F(E,\Omega)\leq \F(F,\Omega),\quad\mbox{for every~$F$ s.t. } F\setminus\Omega=\widetilde{E}\setminus \Omega,
			\end{split}
		\end{equation}
		
		We claim that:
		
		\begin{proposition}[Well-posedness of the minimization problem~\eqref{eq::plateau_F}] \label{prop::good_plateau_problem}
		There exists a minimizer for~$\F$ in~$\Omega$ with external datum~$\widetilde{E}\setminus\Omega$.
		\end{proposition}
		
This result is a consequence of a standard compactness argument, whose precise statement goes as follows:
		
		\begin{lemma}\label{lemma::F_compactness}
			Let~$\{E_j\}_j$ be a sequence of sets such that~$E_j\setminus\Omega=\widetilde{E}\setminus\Omega$ and
			\begin{equation*}
				\sup_j \F(E_j,\Omega) < +\infty.
			\end{equation*}
			
			Then, there exist a set~$E\subseteq\R^n$ and a subsequence~$\{E_{j_m}\}_m$ such that~$E\setminus\Omega=\widetilde{E}\setminus\Omega$ and
			\begin{equation*}
				E_{j_m}\to E \mbox{ in } L^1_{\loc}(\R^n),\quad\mbox{as }m\to+\infty.
			\end{equation*}
		\end{lemma}
		
		\begin{proof}
			Let us consider the covering~$\{B_{\delta/2}(x)\}_{x\in\overline{\Omega}}$ of~$\Omega$ made of balls of radius~$\delta/2$. By compactness, there exist finitely many points~$x_1,\dots,x_N\in\overline{\Omega}$ such that, setting~$B^{(\ell)}:=B_{\delta/2}(x_\ell)\cap \Omega$, we have
			\begin{align*}
				&\Omega\subseteq \bigcup_{\ell=1}^N B^{(\ell)},\\
				\mbox{and}\quad& B^{(\ell)}\cap B^{(\ell+1)}\neq \varnothing, \mbox{ for every }\ell\in\{1,\dots,N-1\}.
			\end{align*}
			
			Now, observe that, in virtue of~\eqref{eq::K_behavior},
			for any~$\ell\in\{1,\dots,N\}$ and for any set~$F$,
\begin{equation*}
\begin{split}
&\frac{\kappa_1}2[\chi_F]^2_{H^{s_1}(B^{(\ell)}}
=\frac{\kappa_1}2\iint_{B^{(\ell)}\times B^{(\ell)}} \frac{|\chi_F(x)-\chi_{F}(y)|^2} {|x-y|^{n+s_1}}\,dx\,dy\\&\qquad
\le\frac1{2}\iint_{B^{(\ell)}\times B^{(\ell)}} |\chi_F(x)-\chi_{F}(y)|^2 K(x,y)\,dx\,dy\\&\qquad
= \iint_{B^{(\ell)}\times B^{(\ell)}} \chi_F(x)\chi_{F^c}(y)K(x,y)\,dx\,dy \leq \iint_{\Omega\times \Omega} \chi_F(x)\chi_{F^c}(y)K(x,y)\,dx\,dy \\
					&\qquad \leq \P_K(F,\Omega) = \F(F,\Omega) - \int_{F\cap\mathcal{Q}(\Omega)_1}g(x)\,dx \leq \F(F,\Omega) + \norm{g}_{L^\infty(\Omega)}|\Omega|,
				\end{split}
			\end{equation*}		
			where~$[\cdot]_{H^s}$ denotes the~$(2s)$-Gagliardo seminorm.
			
			In particular, we deduce that, for all~$j\in\N$,
			\begin{equation*}
				\frac{\kappa_1}2[\chi_{E_j}]^2_{H^{s_1}(B^{(1)})} \leq \sup_j\F(E_j,\Omega) + \norm{g}_{L^\infty(\Omega)}|\Omega|.
			\end{equation*}
			Hence, by the compactness of the Sobolev embedding~$H^{s_1}(B^{(1)})\hookrightarrow\hookrightarrow L^2(B^{(1)})\subseteq L^1(B^{(1)})$
			(see~\cite[Corollary~7.2]{MR2944369}), there exist a set~$E_1\subseteq B^{(1)}$ and a subsequence~$\{E_{j_m}\}_m$ such that~$E_{j_m}\to E_1$ in~$L^1(B^{(1)})$ and pointwise in~$B^{(1)}$.
			
			Similarly,
			\begin{equation*}
				\frac{\kappa_1}2[\chi_{E_{j_m}}]^2_{H^{s_1}(B^{(2)})} \leq \sup_m\F(E_{j_m},\Omega) + \norm{g}_{L^\infty(\Omega)}|\Omega|.
			\end{equation*}
			Therefore, up to considering a further subsequence, we also have that~$E_{j_m}\to E_2$ in~$L^1(B^{(2)})$ and pointwise in~$B^{(2)}$, for some set~$E_2\subseteq B^{(2)}$. Moreover, by the uniqueness of the limit, it follows that~$E_1=E_2$ in~$B^{(1)}\cap B^{(2)}$.
			
			Repeating the same argument in every ball~$B^{(\ell)}$, up to considering further subsequences, we infer that~$E_{j_m}\to E$ in~$L^1(\Omega)$, where
			\begin{equation*}
				E := \left(\bigcup_{\ell=1}^N E_\ell\right) \cup \widetilde{E}\setminus\Omega,
			\end{equation*}
			concluding the proof.
		\end{proof}
		
		\begin{proof}[Proof of Proposition~\ref{prop::good_plateau_problem}]
			First, notice that the energy functional~$\F$ is uniformly bounded from below. Indeed, for any set~$F$ we have
			\begin{equation*}
				\F(F,\Omega) \geq \int_{F\cap\mathcal{Q}(\Omega)_1} g(x)\,dx \geq -\norm{g}_{L^\infty(\Omega)}|\Omega|.
			\end{equation*}
			Thus, we consider a minimizing sequence~$\{E_j\}_j$ such that~$E_j\setminus\Omega=\widetilde{E}\setminus\Omega$, and
			\begin{equation*}
				\F(E_j,\Omega)\to \inf\{\F(F,\Omega) \quad\mbox{s.t. }F\setminus\Omega=\widetilde{E}\setminus\Omega\},\quad\mbox{as }j\to+\infty.
			\end{equation*}
			
Now we take
$$ F:=\begin{cases}
\Omega &\quad {\mbox{ in }}\Omega,\\
\widetilde{E}\setminus\Omega &\quad {\mbox{ in }}\R^n\setminus\Omega
\end{cases} $$ and we observe that,
for any~$j$ large enough,
\begin{eqnarray*}
&&	\F(E_j,\Omega)\le \F(F,\Omega)
=\P_K(F,\Omega)+ \int_{F\cap \mathcal{Q}(\Omega)_1}g(x)\,dx \le \P_K(\Omega)+\norm{g}_{L^\infty(\Omega)}|\Omega|
	<+\infty.
			\end{eqnarray*}
	Therefore, the desired result follows from Lemma~\ref{lemma::F_compactness} and the lower semi-continuity of the functional~$\F$.		
		\end{proof}

\section*{Acknowledgments} \label{sec::acknowledgments}	
		%\addcontentsline{toc}{section}{\nameref{sec::acknowledgments}}
		Matteo Novaga is member of the INDAM-GNAMPA, and was supported by Next Generation EU, PRIN 2022E9CF89.
		Serena Dipierro and Enrico Valdinoci are members of the Australian Mathematical Society (AustMS).
		Serena Dipierro and Riccardo Villa are supported by the Australian Research Council Future Fellowship FT230100333 ``New perspectives on nonlocal equations''.
		Enrico Valdinoci is supported by the Australian Laureate Fellowship FL190100081 ``Minimal surfaces, free boundaries and partial differential equations''.
		Riccardo Villa is supported by a Scholarship for International Research Fees at the University of Western Australia.
		
	\end{appendix}

\end{document}